\documentclass{amsart}

\usepackage{hyperref}
\usepackage{amsmath}
\usepackage{amssymb}
\usepackage{amsfonts}
\usepackage{amsthm}
\usepackage{tikz,tikz-cd,float}
\usepackage{latexsym}
\usepackage{esint}
\usepackage{mathtools}
\usepackage{mathrsfs}
\usepackage{graphicx}
\usepackage{bbm}
\usepackage{color}
\usepackage{bm}
\usepackage{todonotes}
\usepackage[top=1in, bottom=1.25in, left=1.10in, right=1.10in]{geometry}
\usepackage{enumerate}

\usepackage{caption}
\usepackage{subcaption}

\newtheorem{theorem}{Theorem}[section]
\newtheorem{lemma}[theorem]{Lemma}
\newtheorem{proposition}[theorem]{Proposition}
\newtheorem{corollary}[theorem]{Corollary}

\theoremstyle{definition}
\newtheorem{definition}[theorem]{Definition}

\newtheorem{assumption}[theorem]{Assumption}

\theoremstyle{remark}
\newtheorem{remark}[theorem]{Remark}

\DeclareMathOperator*{\cardinality}{cardinality}

\numberwithin{equation}{section}

\usepackage{stmaryrd}
\def\p{\partial}
\def\normal{{\hat {\mathbf{n}}}}
\def\rmh{{h}}
\def\bk{{\bf k}}
\def\yhat{{\hat {\mathbf{y}}}}

\def\domain{{\mathcal D}}
\def\reals{{\mathbb R}}

\newcommand{\bx}{\mathbf{x}}

\newcommand{\bu}{\mathbf{u}}

\newcommand{\bff}{\mathbf{f}}

\newcommand{\dd}{\,\mathrm{d}}

\newcommand{\dx}{\, \mathrm{d} \mathbf{x}}

\newcommand{\dt}{\, \mathrm{d}t}

\begin{document}

\title{Theoretical and  computational analysis of the thermal quasi-geostrophic model}

\author{D Crisan, DD Holm, E Luesink, PR Mensah, W Pan}

\address{DEPARTMENT OF MATHEMATICS, IMPERIAL COLLEGE, LONDON SW7 2AZ, UK.}
\email{d.crisan@imperial.ac.uk, d.holm@imperial.ac.uk, el1616@ic.ac.uk,} \email{p.mensah@imperial.ac.uk, wei.pan@imperial.ac.uk}



\date{\today}



\begin{abstract}
This work involves theoretical and numerical analysis of the Thermal Quasi-Geostrophic (TQG) model of submesoscale geophysical fluid dynamics (GFD). Physically, the TQG model involves thermal geostrophic balance, in which the Rossby number, the Froude number and the stratification parameter are all of the same asymptotic order. The main analytical contribution of this paper is to construct local-in-time unique strong solutions for the TQG model. For this, we show that solutions of its regularized version $\alpha$-TQG converge to solutions of TQG as its smoothing parameter $\alpha \rightarrow 0$ and we obtain blowup criteria for the $\alpha$-TQG model.
The main contribution of the computational analysis is to verify the rate of convergence of $\alpha$-TQG solutions to TQG solutions as $\alpha \rightarrow 0$ for example simulations in appropriate GFD regimes.
\end{abstract}

\maketitle

\section{Introduction}

\subsection{Purpose}
The thermal quasi-geostrophic (TQG) equations comprise a mesoscale model of ocean dynamics in a solution regime near thermal geostrophic balance. In thermal geostrophic balance, three forces -- the Coriolis, hydrostatic pressure gradient and buoyancy-gradient forces -- sum to zero. Numerical simulations of the TQG equations show the onset of instability at high wavenumbers which creates small coherent structures which resemble submesoscale (1–20 km) features observed in satellite ocean color images as seen in figure \ref{fig: submesoscales}. 

\begin{figure}[h!]
\centering
\includegraphics[width=0.45\textwidth]{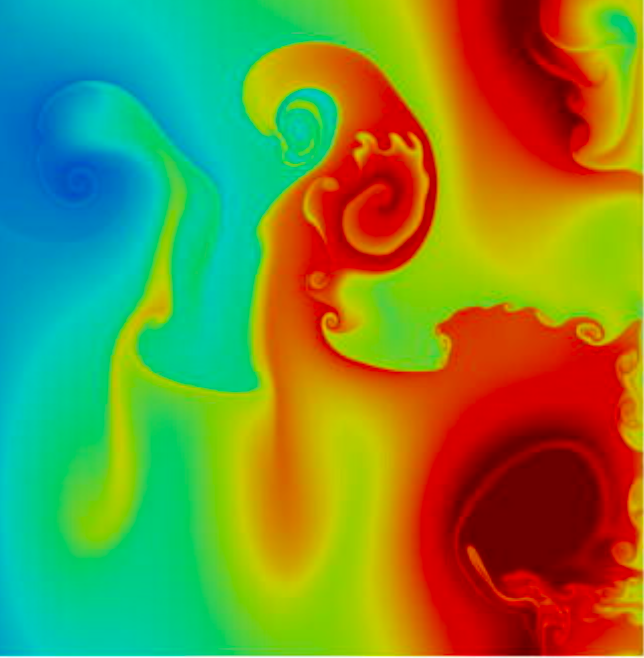}\
\includegraphics[width=0.463\textwidth]{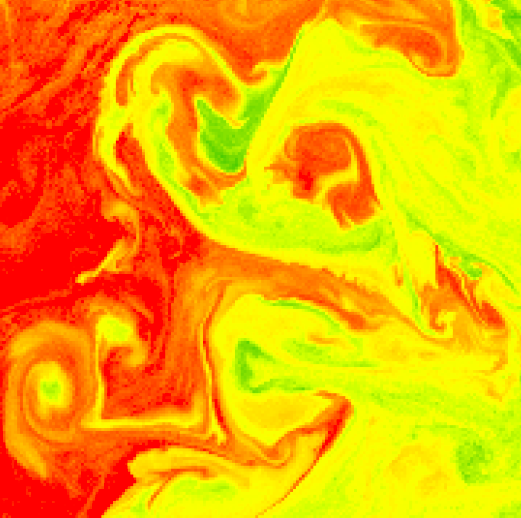}
\caption{Comparison of a computational simulation of TQG solutions with a satellite observation of ocean colour on a section of the Lofoten Vortex, courtesy of \url{https://ovl.oceandatalab.com/} which illustrates the configurations of submesoscale currents obtained from ESA Sentinel-3 OLCI instrument observations of chlorophyll on the surface of the Norwegian Sea in the Lofoten Basin, near the Faroe Islands.}
\label{fig: submesoscales}
\end{figure}

On the left panel of Figure \ref{fig: submesoscales} one sees submesoscale features which are prominently displayed in computational simulations of TQG equations for sea-surface height (SSH). The right panel of Figure \ref{fig: submesoscales} shows the surface of the Lofoten Basin off the coast of Norway near the Faroe Islands. In crossing the Lofoten Basin, warm saline Atlantic waters create buoyancy fronts as they meet the cold currents of the Arctic Ocean. Figure \ref{fig: submesoscales} displays several features of submesoscale currents surveyed in  \cite{mcwilliams2019survey}. High resolution (4km) computational simulations of the Lofoten Vortex have recently discovered that its time-mean circulation is primarily barotropic, \cite{volkov2015formation}, thereby making the flow in the Lofoten Basin a reasonable candidate for investigation using vertically averaged dynamics such as the TQG dynamical system. Both images show a plethora of multiscale features involving shear interactions of vortices, fronts, plumes, spirals, jets and Kelvin-Helmholtz roll-ups.
The submesoscale features persist and interact strongly with each other in Kelvin-Helmholtz roll-up dynamics, instead of simply cascading energy to higher wavenumbers. 
This observation means that the instabilities which create these submesoscale features quickly regain stability without cascading them to ever smaller scales.
The present work aims to understand these features of the TQG solution dynamics, both analytically and numerically. 
\medskip

\paragraph{\bf The contributions of this paper.}  

The main analytical contribution of this paper is the construction of a unique local strong solution of the TQG model. In particular, we show that there is a unique maximal solution of the TQG equation defined on a (possibly infinite) time interval $[0, T_{\mathrm{max}})$. The solution will be shown to exist in a suitable Sobolev space. More precisely, provided that the initial data resides in a chosen Sobolev space, the solution at time $t>0$ will remain in this space as long as $t<T_{\mathrm{max}}$. Should $t<T_{\mathrm{max}}$ be finite, then the solution will blow up in the chosen Sobolev norm.

As a second analytical contribution, we show that the TQG model depends continuously on the initial condition. This dependence only holds in a slightly weaker norm than the one corresponding to the space where the solution resides.
This is a useful property from a numerical perspective. It implies that initial small errors when simulating the TQG model will stay small at any subsequent time.

The third analytical contribution is to construct a regularized version of the TQG model, termed the $\alpha$-TQG model. This model is constructed in a similar manner as the $\alpha$ model for the Euler and Navier--Stokes equations, see \cite{FHT2001, FHT2002, marsden2003the}.
The $\alpha$-TQG equations have a unique maximal solution which is also only continuous with respect to the initial conditions in a larger space with weaker norm. 
In addition, we show that the $\alpha$-TQG solution converges to the TQG solution as $\alpha \rightarrow 0$ in a norm that depends on two physical parameters, the vorticity and the gradient of buoyancy.   
We also identify the rate of convergence as a function of the $\alpha$-parameter on a time interval where both the TQG solution as well as the $\alpha$-TQG solution are shown to exist for any $\alpha>0$.

The blow-up phenomenon is important, particularly if it is observed in numerical simulations. Therefore, blow-up criteria (in other words, criteria required for the solution to blow up) are important. In this paper we state three characterizations of the blow-up time. The fourth analytical contribution of this paper is to show that blow-up occurs in the $\alpha$-TQG model, if either:
\begin{enumerate}
    \item the $L^\infty(\mathbb{T}^2)$-norm of the buoyancy gradient $\nabla b$ blows up at $T_{\max}<\infty$; 
    \item the $L^\infty(\mathbb{T}^2)$-norm of the velocity gradient $\nabla \mathbf{u}$ blows up at $T_{\max}<\infty$ or that;    
    \item the $W^{1,2}(\mathbb{T}^2)$-norm of the buoyancy gradient $\nabla b$ blows up at $T_{\max}<\infty$.
\end{enumerate}
 
The contribution of this paper from a numerical perspective is to describe our spatial and temporal discretisation methods for approximating TQG and $\alpha$-TQG solutions, and to analyse aspects of numerical conservation properties with respect to the theoretical conserved quantities of the $\alpha$-TQG system. 
Example simulation results are included and are used to verify numerically the theoretical convergence results.
Additionally, we provide linear stability analysis results for the $\alpha$-TQG system. Given that we have convergence of $\alpha$-TQG solutions to TQG, the linear stability results can be viewed as generalisations of those shown in \cite{holm2021stochastica} for the TQG system.


\subsection{Brief history of the TQG model}
The history of the TQG model goes back about half a century, perhaps first elucidated by O'Brien and Reid \cite{o1967non} as sketched, in \cite{beron2021nonlinear}. Briefly put, the TQG model generalises the classical QG equations by introducing horizontal gradients of buoyancy which alter the geostrophic balance to include the inhomogeneous thermal effects which influence buoyancy. Indeed, O'Brien and Reid \cite{o1967non} write that their work was inspired by observations that the passage of hurricanes could draw enough heat from the ocean to significantly lower the sea-surface temperature in the Gulf of Mexico. Since O'Brien and Reid \cite{o1967non} introduced their two-layer model, further developments of it have been applied to a variety of ocean processes, particularly to equatorial dynamics. For more details of the theoretical model developments, see Ripa \cite{ripa1993conservation, ripa1995improving, ripa1999validity} and for developments of applications in oceanography see \cite{anderson1985role, beier1997numerical, mccreary1997coastal, schopf1983equatorial}, as well as other citations in \cite{beron2021nonlinear}. In  particular,  Ripa  refers  to  the  TQ  models  as inhomogeneous-layer (IL) models and his papers explain rational derivations of theories with increasing vertical structure IL1, IL2, etc.  The TQG model analysed here and derived systematically from asymptotic expansions in small dimensionless parameters of the Hamilton’s principle for the rotating, stratified Euler equations in \cite{holm2021stochastica} is equivalent to Ripa’s model IL0QG \cite{ripa1996low} recently analyzed in \cite{beron21multilayer}.

\subsection{A sketch of the derivation of the TQG model}\label{subsec-TQGderived}
We have explained that certain thermal effects in the mesoscale ocean have historically been modelled by the thermal quasi-geostrophic (TQG) equations. TQG is characterised by several  dimensionless numbers arising from the dimensional parameters of planetary rotation, gravity and buoyancy. These are the familiar Rossby number, Froude number and stratification parameter.  The Rossby number is the ratio of a typical horizontal velocity divided by the product of the rotation frequency and a typical horizontal length scale. The Froude number is the ratio of a typical horizontal velocity divided by the velocity of the fastest propagating gravity wave, which in turn is given by the square root of the gravity times the typical vertical length scale. The final dimensionless number is the stratification parameter, which specifies the typical size of the buoyancy stratification. 

The regime in which the thermal quasi-geostrophic equations are derived is characterised by the thermal geostrophic balance. This balance arises because the Rossby number, the Froude number and the stratification parameter all have a similar amplitude. Preserving this three-fold balance requires simultaneously adapting the Froude number and the stratification parameter to match any change in Rossby number, for example, so that the dimensionless parameters will still have the same size. We consider the non-dissipative case, because of the large scales of mesoscale ocean dynamics. As mentioned earlier, the mesoscale dynamics has high wavenumber instabilities, which in principle can generate submesoscale effects. At smaller scales, viscous dissipation and thermal diffusivity will come into play as well. However, in what follows, dissipative effects will be neglected. The derivation of the TQG model involves a series of coordinated approximations to the rotating, stratified Euler model, as is illustrated in the diagram below.
\begin{figure}[H]
\scriptsize
\centering
\begin{tikzcd}
[row sep = 5em, 
column sep=2em, 
cells = {nodes={top color=green!20, bottom color=blue!20,draw=black}},
arrows = {draw = black, rightarrow, line width = .015cm}]
& &
{\begin{matrix}
\text{Rotating, stratified} \\ \text{Euler equations}
\end{matrix}}
\arrow[d,"{\begin{matrix} \text{small buoyancy}\\ \text{stratification}\end{matrix}}", shorten <= 1mm, shorten >= 1mm, red!80] 
&\\
& {\begin{matrix} \text{Primitive} \\ \text{equations} \end{matrix}} \arrow[d,"{\begin{matrix} \text{vertical average,}\\ \text{small wave amplitude}\end{matrix}}"', shorten <= 1mm, shorten >= 1mm, red!80]
& 
{\begin{matrix} \text{Euler--Boussinesq}\\ \text{equations}\end{matrix}} \arrow[dr, "{\begin{matrix} \text{vertical average,}\\ \text{very small wave amplitude}\end{matrix}}", shorten <= 4mm, shorten >= 4mm, dashed] \arrow[d, "{\begin{matrix} \text{vertical average,}\\ \text{small wave amplitude}\end{matrix}}"', shorten <= 1mm, shorten >= 1mm, red!80] \arrow[l,"\begin{tabular}{l}\tiny \text{hydrostatic} \\\tiny \text{approx.} \\ \text{ } \end{tabular}"', shorten <= 1mm, shorten >= 1mm, red!80] &
\\
&
{\begin{matrix} \text{Thermal rotating}\\ \text{shallow water} \end{matrix}}
\arrow[dd,"\text{restrict to 1D}", shorten <= 1mm, shorten >= 1mm, dashed] \arrow[dl,"\begin{tabular}{l}
\tiny \text{asymptotic expansion around} \\
\tiny \text{thermal geostrophic balance}
\end{tabular}"', shorten < =4mm, shorten >= 4mm, red!80]
& 
{\begin{matrix} \text{Thermal rotating}\\ \text{Green--Naghdi equations} \end{matrix}} \arrow[dd, "\text{restrict to 1D}", shorten <= 1mm, shorten >= 1mm, dashed] \arrow[l,"\begin{tabular}{l} \\ \\ \tiny \text{hydrostatic}\\ \tiny \text{approx.} \end{tabular}", shorten <= 1mm, shorten >= 1mm, red!80] \arrow[r, "\begin{tabular}{l} \\ \\ \tiny \text{rigid lid} \\ \text{ } \end{tabular}"', shorten <= 1mm, shorten >= 1mm, dashed]
& 
{\begin{matrix} \text{Thermal rotating}\\ \text{great lake equations} \end{matrix}} \arrow[d, "\begin{tabular}{l} \tiny \text{hydrostatic} \\ \tiny\text{approx.} \end{tabular}", shorten <= 1mm, shorten >= 1mm, dashed] 
\\
{\text{Thermal Lagrangian 1}}
\arrow[d, "\text{strict asymptotics}"', shorten <= 1mm, shorten >= 1mm, red!80]
&
&
& 
{\begin{matrix} \text{Thermal rotating}\\ \text{lake equations}\end{matrix}}
\\
{\begin{matrix} \text{Thermal quasi-geostrophic} \\ \textbf{equations} \end{matrix}}
& {\begin{matrix} \text{Saint-Venant}\\ \text{family}\end{matrix}}
& {\begin{matrix} \text{Camassa-Holm}\\ \text{family}\end{matrix}} \arrow[l,"\begin{tabular}{l} \\ \tiny \text{hydrostatic}\\ \tiny \text{approx.} \end{tabular}", shorten <= 1mm, shorten >= 1mm, dashed] 
&
\end{tikzcd}
\caption{The tree of model derivations can function as a roadmap of geophysical fluid dynamics. The solid red arrows indicate the sequence of approximations that lead to the thermal quasi-geostrophic equations.}
\label{fig:introductiontree}
\end{figure}
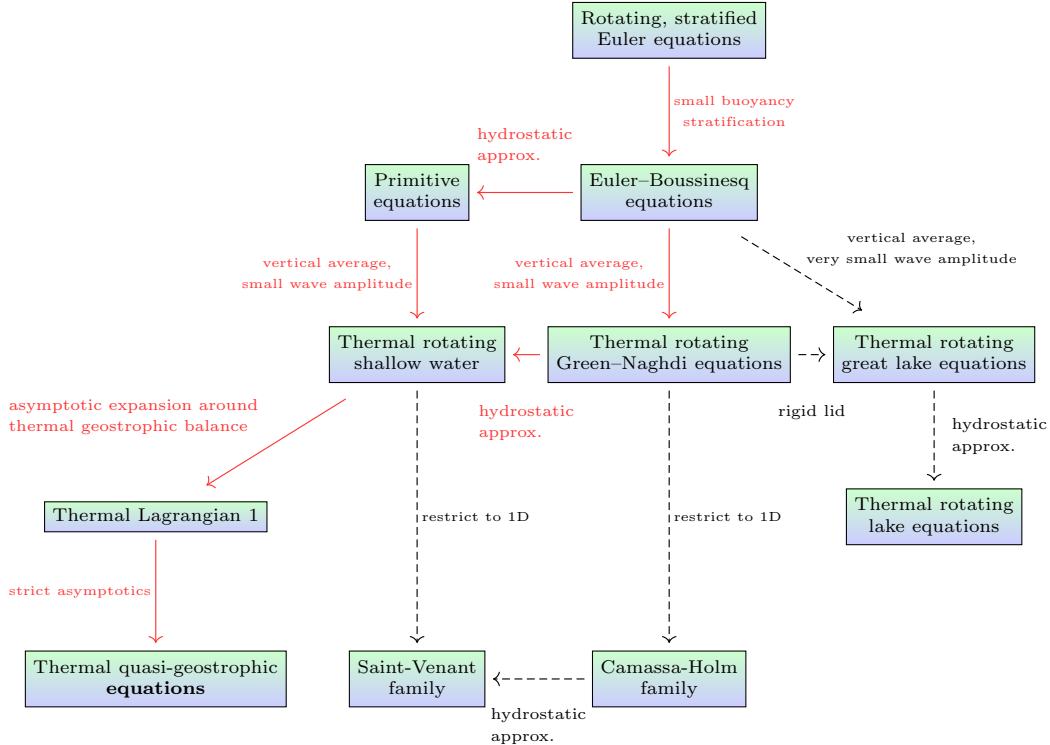
A discussion of the right-most three columns tree in Figure \ref{fig:introductiontree} can be found in \cite{holm2021stochastic}. The derivation of the thermal quasi-geostrophic equations treated here from the thermal rotating shallow water equations on the middle left of the figure can be found in \cite{holm2021stochastica}. The derivation of the thermal quasi-geostrophic equations following the solid red arrows in Figure \ref{fig:introductiontree} starts with the Lagrangian for the rotating, stratified Euler equations at the top of the figure. After identifying the small dimensionless parameters in the Euler Lagrangian, the Lagrangians for the successive approximate models can be derived by inserting asymptotic expansions into the Euler Lagrangian. In the ocean, the buoyancy stratification is typically small. Hence, it makes sense to apply the Boussinesq approximation in the Lagrangian for the rotating, stratified Euler equations. This approximation yields the Lagrangian for the Euler--Boussinesq equations. At this point, one has a choice of two routes. The first route begins by making the hydrostatic approximation, which leads to the Lagrangian for the primitive equations. Then, upon vertically integrating the Lagrangian for the primitive equations, one finds the Lagrangian for the thermal rotating shallow water equations. The alternative route first vertically integrates the Lagrangian for the Euler--Boussinesq equations to obtain the Lagrangian for the thermal and rotating version of the Green--Naghdi equations. By subsequently making the hydrostatic approximation in the Lagrangian for the thermal rotating Green--Naghdi equations, the alternative route arrives at the Lagrangian for the thermal rotating shallow water equations. 

The Lagrangian for the thermal rotating shallow water (TRSW) equations is the starting point in \cite{holm2021stochastica} for the derivation of the thermal quasi-geostrophic (TQG) equations. The typical values of the dimensionless numbers for mesoscale ocean problems lead to thermal geostrophic balance. This balance implies an algebraic expression for the balanced velocity field in terms of the horizontal gradients of the free surface elevation and the buoyancy. Upon expanding the TRSW Lagrangian around this balance and truncating, one obtains the Lagrangian for the thermal Lagrangian 1 (L1) model. The Lagrangian for the thermal L1 model is not hyperregular, though. Hence, the Legendre transformation is for thermal L1 is not invertible. Hence, no Hamiltonian description of this model is available via the Legendre transformation. Nonetheless, an application of the Euler-Poincar\'e theorem \cite{holm1998euler} to the thermal L1 Lagrangian yields the corresponding equations of motion. By identifying the leading order terms in the thermal L1 equations and truncating the asymptotic expansion, one finally obtains the TQG equations. However, the truncations of the asymptotic expansions in the thermal L1 equations also prevent the resulting TQG equations from possessing a Hamilton's principle derivation. This feature is unlike the other models in figure \ref{fig:introductiontree}, which all arise from approximated Lagrangians in their corresponding action integrals for Hamilton's principle. However, as it turns out, the thermal quasi-geostrophic equations do possess a Hamiltonian formulation in terms of a non-standard Lie-Poisson bracket. The Lie-Poisson bracket for the quasi-geostrophic equations can be obtained via a linear change of variables from the usual semi-direct product Lie-Poisson bracket for fluids. The resulting Hamiltonian formulation for the thermal quasi-geostrophic equations is important for the future application of stochastic advection by Lie transport (SALT), introduced in \cite{holm2015variational}, which requires either a Lagrangian, or a Hamiltonian interpretation of the equations of motion. 

By applying the method of Stochastic Advection by Lie Transport (SALT) to the Hamiltonian formulation of the TQG equations, one obtains a stochastic version of them which preserves the infinite family of integral conserved quantities. More details on the stochastic version can be found in the conclusion section as well as in \cite{holm2021stochastica}. A discussion of the remaining models in Figure \ref{fig:introductiontree} can be found in \cite{holm2021stochastic}.

The TQG model on the two dimensional flat torus $\mathbb{T}^2$ can be formulated in two equivalent ways. The first formulation is the advective formulation, in which the equations of motion are given by
\begin{align}
\frac{\partial}{\partial t}b + \mathbf{u}\cdot\nabla b &= 0,
\label{ce}\\
\frac{\partial}{\partial t}\omega + \mathbf{u}\cdot\nabla(\omega-b) &= -\mathbf{u}_{h_1}\cdot\nabla b,
\label{me}\\
\omega = (\Delta-1)\psi + f_1&
\,, \quad {\rm div} \mathbf{u} = 0
\,.
\label{constrtXXX}
\end{align}
Here $b$ is the vertically averaged buoyancy, $\mathbf{u}$ is the thermal geostrophically balanced velocity field, $\omega$ is the potential vorticity, $\psi$ is the streamfunction, $h_1$ is the spatial variation around a constant bathymetry profile and $f_1$ is the spatial variation around a constant background rotation rate. The velocity $\mathbf{u}$ and the streamfunction $\psi$ are related by the equation 
\begin{equation}
\label{constrtVelocity}
\mathbf{u} = \nabla^\perp \psi,
\end{equation}
where $\nabla^\perp = (-\partial_y,\partial_x)$. The vector field $\mathbf{u}_{h_1}$ is defined by 
\begin{equation}
\label{constrt}
\mathbf{u}_{h_1} := \frac{1}{2}\nabla^\perp h_1.
\end{equation}
Alternatively, one can formulate the thermal quasi-geostrophic (TQG) equations in vorticity-streamfunction form. This formulation is given by
\begin{align}
\frac{\partial}{\partial t} b + J(\psi,b) &= 0,
\label{ceZero}\\
\frac{\partial}{\partial t}\omega + J(\psi,\omega-b) &= -\frac{1}{2}J(h_1,b),
\label{meZero}\\
\omega &= (\Delta-1)\psi + f_1.
\label{constrtZero}
\end{align}
The operator $J(a,b)=\nabla^\perp a\cdot \nabla b=a_x b_y - b_x a_y$ is the Jacobian of two smooth functions $a$ and $b$ defined on the $(x,y)$ plane.
The equivalence of the two formulations \eqref{ce}--\eqref{me} and \eqref{ceZero}--\eqref{meZero} follows from the Jacobian operator relation $J(\psi,a) = \mathbf{u}\cdot\nabla a$, in which $\psi$ is the streamfunction associated to the velocity vector field $\mathbf{u}$. The scalar functions $f_1$ and $h_1$ relate to the usual Coriolis parameter and bathymetry profile in the following way
\begin{equation}
\begin{aligned}
h(\mathbf{x}) &= 1 + {\rm Ro}\, h_1(\mathbf{x}),\\
f(\mathbf{x}) &= 1 + {\rm Ro}\, f_1(\mathbf{x}),
\end{aligned}
\label{eq:expansionhf}
\end{equation}
where ${\rm Ro} = U(f_0 L)^{-1}$ is the Rossby number, expressed in terms of the typical horizontal velocity $U$, typical rotation frequency $f_0$ and typical horizontal length scale $L$. This means that the bathymetry and Coriolis parameter become constant as the Rossby number tends to zero. Equations \eqref{eq:expansionhf} are necessary to derive the thermal quasi-geostrophic equations from the thermal L1 equations, as shown in \cite{holm2021stochastica}. The expansion \eqref{eq:expansionhf} contains the $\beta$-plane approximation provided that the boundary conditions are appropriate. Namely, on the $\beta$-plane, one requires that $\beta f_0^{-1} = \mathcal{O}({\rm Ro})$ and $f_1(\mathbf{x})=y$. An additional relation can be helpful when using the thermal quasi-geostrophic equations as a model for mesoscale ocean dynamics. In particular, the streamfunction is related to the free surface elevation $\zeta$ and buoyancy $b$ via the definition $\psi:=\zeta + \frac{1}{2}b$. This definition of the streamfunction is useful to relate to observational data. The free surface elevation is a quantity that can be measured with satellite altimetry. These measurements can then be used for data assimilation and model calibration. However, the definition of the streamfunction in terms of the velocity field is not necessary to formulate the model as a closed set of equations, since the system \eqref{ceZero}--\eqref{constrtZero} is already a closed set of equations. In the expansion \eqref{eq:expansionhf}, we will henceforth drop the subscript 1 on $h_1(\mathbf{x})$ and $f_1(\mathbf{x})$ for notational convenience.

\subsection{TQG versus Rayleigh-B\'enard convection with viscosity and thermal diffusivity}
The TQG equations in \eqref{ceZero}-\eqref{constrtZero} can be compared to the equations for Rayleigh-B\'enard convection in a vertical plane, as remarked in \cite{holm2021stochastica}. Recent results of \cite{cao2021algebraic} show that bounds exist for the enstrophy and temperature gradient of solutions lying in the attractor for the planar Rayleigh-B\'enard convection problem which are algebraic in the viscosity and thermal diffusivity. This result is a significant improvement over previously established estimates. This result also provides a motivation to study the TQG equations with viscous dissipation and thermal diffusivity. The equations for non-dissipative Rayleigh-B\'enard convection in the vertical plane are given by
\begin{equation}
\begin{aligned}
    \frac{\partial}{\partial t}\omega + J(\psi,\omega) &= \alpha g T_z,\\
    \frac{\partial}{\partial t}T + J(\psi,T) &= 0,\\
    \omega &= \Delta\psi,
\end{aligned}
\label{eq:rayleighbenard}
\end{equation}
where $\omega$ denotes the vorticity, $\psi$ is the stream function, $T$ is the temperature, $g$ is gravity and $\alpha$ is the thermal expansion coefficient. The Rayleigh-B\'enard equations \eqref{eq:rayleighbenard} have a long and illustrious history in mathematical analysis. The resemblance of the Rayleigh-B\'enard equations to the TQG equations suggests that perhaps the TQG equations will also provide a fruitful challenge to mathematical analysis. 

In TQG, the buoyancy plays the role of the temperature in the Rayleigh-B\'enard equations. By rearranging the TQG equations \eqref{constrtZero} such that the buoyancy terms all appear on the right hand side, one has
\begin{equation}
\begin{aligned}
    \frac{\partial}{\partial t}\omega + J(\psi,\omega) &= \frac{1}{2}J(\psi,b)+\frac{1}{2}J(\zeta-h,b),\\
    \frac{\partial}{\partial t}b + J(\psi,b) &= 0,\\
    \omega &= \Delta\psi-\psi + f.
\end{aligned}
\label{eq:tqgforcingfrom}
\end{equation} 
There are similarities and also several differences between the Rayleigh-B\'enard equations \eqref{eq:rayleighbenard} and the TQG equations \eqref{eq:tqgforcingfrom}. For example, the equation that relates the vorticity to the stream function is a Helmholtz equation for TQG and a Poisson equation for Rayleigh-B\'enard. In addition, in Rayleigh-B\'enard convection the forcing term on the right hand side of the vorticity equation depends only on the derivative of the temperature in the vertical direction. In TQG the forcing terms depends on the derivatives of the buoyancy in both directions. Thus, TQG arises as an interesting and challenging extension of the celebrated Rayleigh-B\'enard problem.

\subsection{Plan for the rest of the paper}
We now give the plan for the rest of the paper.
We collect preliminary tools in Section \ref{sec:prelim}. This includes notations and analytical properties of function spaces used throughout this paper. We also collect useful estimates that will be used at various stages and give  precise definitions of the concept of solutions used in our analysis. We finally end Section \ref{sec:prelim} with statements of the main results. In particular, we state that both the TQG and $\alpha$-TQG  equations  admit  unique strong solutions for a finite period of time and these solutions are stable in a larger space with weaker norm. Furthermore, a maximum time for these solutions exists. 

Since the proof of local well-posedness is the same for the TQG and $\alpha$-TQG, we will avoid duplication by devoting Section \ref{sec:construct}  to the construction of   solutions for  the less regular TQG. The construction  relies heavily on the standard energy method. Since we are constructing strong solutions (rather than weak ones), we differentiate the equations in space and then we test the resulting equations with the required differential of the solution to obtain the required bounds. We then end  Section \ref{sec:construct}  by showing that the unique solution constructed has a maximum time of existence and hence, is a maximal solution.

In Section \ref{sec:converge} we  show that any family of maximal solutions of the $\alpha$-TQG models converges strongly with $\alpha \rightarrow 0$ to the unique maximal solution of the TQG on a common existence time, provided  they share the same data.

Next, since our solutions are local in nature, we establish in Section \ref{sec:blowup}, conditions under which this solution may blow up in the sense of Beale--Kato--Majda \cite{beale1984remarks}. In particular, we show that in order to control the solution of $\alpha$-TQG, the essential supremum in space of both the buoyancy gradient and the velocity gradient should be integrable over the anticipated time interval. Once either of these gradients blow up, the solution ceases to exist. Alternatively, in order to control the solution, it suffices to control the $H^1$-Sobolev norm of the buoyancy gradient.


Section \ref{sec: numerics} is devoted to numerical methods and simulation results. We begin by describing the finite element method we use for the spatial derivatives, and aspects of its numerical conservation properties with respect to theoretical results. We then describe the finite difference discretisation method used for the time derivative. Next, we discuss $\alpha$-TQG linear thermal Rossby wave stability analysis, which can be seen as generalising the results shown in \cite{holm2021stochastica} for the TQG system. Then in the last part of the subsection, we discuss our numerical simulation setup and its results. In particular, we numerically verify the theoretical convergence rate for $\alpha\hbox{-TQG}\to\hbox{TQG}$ derived in Section \ref{sec:converge}.

\section{Preliminaries and main results}
\label{sec:prelim}
In this section, we fix the notation, collect some preliminary material on function spaces and present the main analytical results.
\begin{remark}
Although we will be working on the $2$-dimensional torus, with minimal effort, the same analysis will work on the whole plane $\mathbb{R}^2$ subject to a far-field condition. The case of a bounded domain with boundary conditions is however outside the scope of the analytical aspect of this work. See Section \ref{sec: numerics} for numerical works in this regard.
\end{remark}
\subsection{Notations}
Our independent variables consists of spatial points $x:=\mathbf{x}=(x, y)\in \mathbb{T}^2$ on the $2$-torus $\mathbb{T}^2$   and a time variable  $t\in [0,T]$ where $T>0$. For functions $F$ and $G$, we write $F \lesssim G$  if there exists  a generic constant $c>0$  such that $F \leq c\,G$.
We also write $F \lesssim_p G$ if the  constant  $c(p)>0$ depends on a variable $p$. The symbol $\vert \cdot \vert$ may be used in four different context. For a scalar function $f\in \mathbb{R}$, $\vert f\vert$ denotes the absolute value of $f$. For a vector $\bff\in \mathbb{R}^2$, $\vert \bff \vert$ denotes the Euclidean norm of $\bff$. For a square matrix $\mathbb{F}\in \mathbb{R}^{2\times 2}$, $\vert \mathbb{F} \vert$ shall denote the Frobenius norm $\sqrt{\mathrm{trace}(\mathbb{F}^T\mathbb{F})}$. Finally, if $S\subseteq  \mathbb{R}^2$ is  a (sub)set, then $\vert S \vert$ is the $2$-dimensional Lebesgue measure of $S$.
\\
For $k\in \mathbb{N}\cup\{0\}$ and $p\in [1,\infty]$, we denote by $W^{k,p}(\mathbb{T}^2)$, the Sobolev space of Lebesgue measurable functions whose weak derivatives up to order $k$ belongs to $L^p(\mathbb{T}^2)$. Its associated  norm is
\begin{align}
\Vert v \Vert_{W^{k,p}(\mathbb{T}^2)} =\sum_{\vert \beta\vert\leq k} \Vert \partial^\beta v \Vert_{L^{p}(\mathbb{T}^2)},
\end{align}
where $\beta$ is a $2$-tuple multi-index of nonnegative integers  of length $\vert \beta \vert \leq k$.
The Sobolev space $W^{k,p}(\mathbb{T}^2)$ is a Banach space. Moreover, $W^{k,2}(\mathbb{T}^2)$ is a Hilbert space when endowed with the inner product
\begin{align}
\langle u,v \rangle_{W^{k,2}(\mathbb{T}^2)} =\sum_{\vert \beta\vert\leq k} \langle \partial^\beta u\,,\, \partial^\beta v \rangle,
\end{align}
where $\langle\cdot\,,\,\rangle$ denotes the standard $L^2$-inner product.  
In general, for $s\in\mathbb{R}$, we will define the Sobolev space $H^s(\mathbb{T}^2)$  as consisting of  distributions $v$ defined on $\mathbb{T}^2$ for which the norm
\begin{align}
\label{sobolevNorm}
\Vert  v\Vert_{H^s(\mathbb{T}^2)}
=
 \bigg(\sum_{\xi \in \mathbb{Z}^2} \big(1+\vert \xi\vert^2  \big)^s\vert  \widehat{v}(\xi)\vert^2
  \bigg)^\frac{1}{2}
  \equiv
  \Vert  v\Vert_{W^{s,2}(\mathbb{T}^2)}
\end{align}
defined in frequency space is finite. Here, $\widehat{v}(\xi)$ denotes the Fourier coefficients  of $v$.
To shorten notation, we will write $\Vert  \cdot\Vert_{s,2}$ for $\Vert \cdot\Vert_{W^{s,2}(\mathbb{T}^2)}$
and/or $\Vert \cdot\Vert_{H^s(\mathbb{T}^2)}$.
When $k=s=0$, we get the usual $L^2(\mathbb{T}^2)$ space whose norm we will denote by $\Vert \cdot \Vert_2$ for simplicity.
We will also use a similar convention for norms $\Vert \cdot \Vert_p$ of general $L^p(\mathbb{T}^2)$ spaces for any $p\in [1,\infty]$ as well as for the inner product $\langle\cdot,\cdot \rangle_{k,2}:=\langle\cdot,\cdot \rangle_{W^{k,2}(\mathbb{T}^2)}$ when $k\in \mathbb{N}$. Additionally, we will denote by  $W^{k,p}_{\mathrm{div}}(\mathbb{T}^2;\mathbb{R}^2)$, the space of weakly divergence-free vector-valued functions in $W^{k,p}(\mathbb{T}^2)$.
Finally, we define the following space
\begin{align}
\mathcal{M}=W^{3,2}(\mathbb{T}^2) \times W^{2,2}(\mathbb{T}^2)  
\end{align}
endowed with the norm
\begin{align}
\Vert(b,\omega) \Vert_{\mathcal{M}}:= \Vert b \Vert_{3,2} +\Vert \omega \Vert_{2,2}
.
\end{align}
\subsection{Preliminary estimates}
We begin this section with the following result which follow from a direct computation  using the definition \eqref{sobolevNorm} of the Sobolev norms.
\begin{lemma} 
Let $k\in \mathbb{N}\cup\{0\}$ and assume that the triple $(\bu,w)$ satisfies
\begin{align}
\label{constrt0}
\bu =\nabla^\perp \psi, 
\qquad
w=(\Delta -1) \psi.
\end{align}
If $w\in W^{k,2}(\mathbb{T}^2)$, then the following estimate
\begin{align}
\label{lem:MasterEst}
\Vert \bu \Vert_{k+1,2}^2  &\lesssim \Vert w \Vert_{k,2}^2
\end{align}
holds.
\end{lemma}
Let us now recall some Moser-type calculus. See \cite{klainerman1981singular, kato1990liapunov, kato1988commutator}.
\begin{lemma}[Commutator estimates]
Let $\beta$ be a $2$-tuple multi-index of nonnegative integers   such that $\vert \beta \vert \leq k$ holds for  $k\in\{1,2\}$. Let $p,p_2,p_3\in(1,\infty)$ and $p_1,p_4 \in (1,\infty]$ be such that
\begin{align*}
\frac{1}{p}=\frac{1}{p_1}+\frac{1}{p_2}=\frac{1}{p_3}+\frac{1}{p_4}.
\end{align*}
For $u \in W^{k,p_3}(\mathbb{T}^2) \cap W^{1,p_1}(\mathbb{T}^2)$ and $v \in W^{k-1,p_2}(\mathbb{T}^2) \cap L^{p_4} (\mathbb{T}^2)$,  
we have
and
\begin{equation}\label{E6}
\left\| \partial^\beta (uv) - u \partial^\beta v \right\|_{p} \lesssim 
\left( \| \nabla u \|_{p_1} \|  v \|_{k-1,p_2} + \|  u \|_{k,p_3}
\| v \|_{p_4}  \right).
\end{equation}
\end{lemma}
\subsection{Main results}
Our current goal  is to construct a solution for the  system of equations \eqref{ce}--\eqref{constrt}. To do this, we first make the following assumption on our set of data.
\begin{assumption}
\label{dataAllThroughout}
Let $\bu_h \in W^{3,2}_{\mathrm{div}}(\mathbb{T}^2;\mathbb{R}^2)$ and $f\in W^{2,2}(\mathbb{T}^2)$ and assume that $(b_0, \omega_0) \in \mathcal{M}$. 
\end{assumption}
Unless otherwise stated, Assumption \ref{dataAllThroughout} now holds throughout the rest of the paper.
We are now in a position to make precise, exactly what we mean by a solution.
\begin{definition}[Local strong solution]\label{def:solution}
Let $T>0$ be a constant. We call the triple $(b, \omega,T) $
a  \textit{local strong solution} or simply, a \textit{local solution} or a \textit{solution} to the system \eqref{ce}--\eqref{constrt} if the following holds.
\begin{itemize}
\item The buoyancy $b$ satisfies $b \in  C([0,T]; W^{3,2}(\mathbb{T}^2))$
and the equation
\begin{align*}
b(t) &= b_0 -  \int_0^{t} \mathrm{div} (b\bu)\,\dd \tau, 
\end{align*}
holds for all $t\in[0,T]$;
\item  the potential vorticity $\omega$ satisfies $\omega \in  C([0,T]; W^{2,2}(\mathbb{T}^2))$ 
and the equation
\begin{align*}
\omega (t)  &= \omega_0 - \int_0^t \Big[ \mathrm{div}  ((\omega-b)\mathbf{u} ) 
+
\mathrm{div} (b\bu_h ) \Big] \,\dd \tau
\end{align*}
holds for all $t\in[0,T]$.
\end{itemize}
\end{definition}
\begin{remark} 
We remark that the  regularity of the solution $(b, \omega, T) $ and its  data together with the integral equations  immediately imply that $b$ and $\omega$ are differentiable in time.
Indeed, by using the fact that $W^{3,2}(\mathbb{T}^2)$ and $W^{2,2}(\mathbb{T}^2)$ are Banach algebras, we immediately deduce from the integral equations for $b$ and $\omega$ above that
\begin{align*}
    b \in  C^1([0,T]; W^{2,2}(\mathbb{T}^2)), \qquad \omega \in  C^1([0,T]; W^{1,2}(\mathbb{T}^2)).
\end{align*}
It also follow from the integral equations for the buoyancy and potential vorticity above that the initial conditions are  $b(0,x)=b_0(x)$ and $\omega(0,x)=\omega_0(x)$. The corresponding differential forms \eqref{ce}--\eqref{me} are  clearly immediate from the integral representations. 
\end{remark}
\begin{remark}
Since we are working on the torus, and the velocity fields are defined by \eqref{constrtVelocity}-\eqref{constrt}, we have in
particular, $\int_{\mathbb{T}^2}\bu_h \dx=0$ and $\int_{\mathbb{T}^2}\bu \dx=0$. From the latter, we get that $q$ and $f$ have zero averages. Consequently, we will assume that all
functions under consideration have zero averages.
\end{remark}
\begin{definition}[Maximal solution]\label{def:MaxSolution}
We call  $(b, \omega,  T_{\max}) $
a  \textit{maximal solution} to the system \eqref{ce}--\eqref{constrt} if:
\begin{itemize}
\item there exists an increasing sequence of time steps $(T_n)_{n\in \mathbb{N}}$ whose limit is $T_{\max}\in (0,\infty]$;
\item  for each $n\in \mathbb{N}$, the triple $(b, \omega, T_n) $ is
a local strong solution to the system \eqref{ce}--\eqref{constrt} with initial condition $(b_0, \omega_0) $;
\item if $T_{\max}<\infty$, then
\begin{align}
\label{limsupSol}
\limsup_{T_n\rightarrow T_{\max}} \Vert (b,\omega)(T_n) \Vert_{\mathcal{M}}^2 =\infty.
\end{align} 
\end{itemize}
We shall call $T_{\max}>0$ the \textit{maximal time}.
\begin{remark}
Condition \eqref{limsupSol} means that the solution breaks down at the limit point $T_{\max}$.
\end{remark}
\end{definition}
We are now in a position to state our first main result.
\begin{theorem}[Existence of local solutions]
\label{thm:main}
There exists  a  solution $(b, \omega, T) $ of \eqref{ce}--\eqref{constrt} under Assumption \ref{dataAllThroughout}.
\end{theorem}
Once we have constructed a local solution, we can show that this solution is continuously dependent on its initial state in a more general class of function space. This choice of class  appears to be the strongest space in which the analysis may be performed. More details will follow in the sequel but first, we give the statement on the continuity property of strong solutions with respect to its data.
\begin{theorem}[Continuity at low regularity] 
\label{thm:stability}
Let $\bu_h \in W^{3,2}_{\mathrm{div}}(\mathbb{T}^2;\mathbb{R}^2)$ and $f\in W^{2,2}(\mathbb{T}^2)$. Assume that  $(b^1, \omega^1, T^1) $ and  $(b^2,\omega^2, T^2)$ are  solutions of \eqref{ce}--\eqref{constrt}   with initial conditions $(b_0^1, \omega_0^1) \in \mathcal{M}$ and $(b_0^2,\omega_0^2)\in \mathcal{M}$ respectively. Then there exists a constant
\begin{align*}
c=c\big( \Vert b_0^1\Vert_{3,2}, \Vert \omega_0^1\Vert_{2,2}, \Vert b_0^2\Vert_{3,2},  \Vert \omega_0^2 \Vert_{2,2}, \Vert \bu_h\Vert_{3,2}, \Vert f\Vert_{2,2} \big)
\end{align*}
such that
\begin{equation}
\begin{aligned}
 \Vert  (b^1 -b^2)(t) \Vert_{2,2}^2
+
\Vert (\omega^1 - \omega^2 )(t) \Vert_{1,2}^2  
\leq \exp(cT)\big(\Vert  b_0^1 - b_0^2 \Vert_{2,2}^2
+ \Vert \omega_0^1 - \omega_0^2 \Vert_{1,2}^2 
 \big) 
\end{aligned}
\end{equation} 
holds for all $t\in[0,T]$ where $T=\min\{T^1,T^2\}$.
\end{theorem}
\begin{remark}
For the avoidance of doubt, we make clear that the final estimate for the differences in Theorem \ref{thm:stability} above is stated in terms of a larger space $W^{2,2}(\mathbb{T}^2) \times W^{1,2}(\mathbb{T}^2)$ with smaller norms than the space of existence $W^{3,2}(\mathbb{T}^2) \times W^{2,2}(\mathbb{T}^2)$. In order to obtain this bound however, in particular, we require  the initial conditions of one of the solution to be bounded in the stronger space of existence, i.e, boundedness of  $\Vert b_0^1\Vert_{3,2}, \Vert \omega_0^1\Vert_{2,2}$ rather than in the weaker space for which we obtain our final stability estimate. Furthermore, we also require the potential vorticity $\omega^2$ (at all time of existence) of the second solution to be also bounded in the stronger space of existence $W^{2,2}(\mathbb{T}^2)$. Explicitly, the only flexibility we have has to do with the second buoyancy $b^2$ for which it appears that we are able to relax to live in the weaker space $W^{2,2}(\mathbb{T}^2)$. Unfortunately however, because of the highly coupled nature of the system of equations under study, needing $\omega^2 \in W^{2,2}(\mathbb{T}^2)$ for all times  automatically requires that $b^2 \in W^{3,2}(\mathbb{T}^2)$ for all times. Subsequently, this means that we require boundedness of the initial condition of the second solution in the stronger space of existence, i.e, boundedness of  $\Vert b_0^2 \Vert_{3,2}, \Vert \omega_0^2\Vert_{2,2}$ in addition to that of the first solution.
 The requirement of needing both pair of initial conditions to live in a stronger space is in contrast to simpler looking models like the Euler equation where it suffices to require just having one initial condition to have stronger regularity. Unfortunately, we can not do better by our method of proof (which is to derive estimates for equations solved by the differences, i.e., the energy method) but we do not claim that other methods for deriving analogous estimates may not yield better result either.
\end{remark}
As a consequence of  Theorem \ref{thm:stability}, the following statement about uniqueness of the strong solution is immediate.
\begin{corollary}[Uniqueness]
\label{cor:Uniqueness}
Let $(b^1,\omega^1, T^1)$ and $(b^2,\omega^2, T^2)$ be  two solutions of \eqref{ce}--\eqref{constrt} under Assumption \ref{dataAllThroughout}. 
Then 
the difference $(b^1-b^2,\omega^1-\omega^2)$ satisfies the equation
\begin{align}
\Vert (b^1-b^2)(t) \Vert_{2,2}^2+ \Vert  (\omega^1-\omega^2)(t)\Vert_{1,2}^2=0
\end{align}
for all $t\in[0,T]$ where $T=\min\{T^1,T^2\}$.
\end{corollary}
\noindent Finally, we can show that a \textit{maximal solution} of \eqref{ce}--\eqref{constrt}, in the sense of Definition \ref{def:MaxSolution}, also exists.
\begin{theorem}[Existence of maximal solution]
\label{thm:MaxSol}
There exist a unique maximal solution $(b, \omega, T_{\max}) $ of \eqref{ce}--\eqref{constrt} under Assumption \ref{dataAllThroughout}.
\end{theorem}
\noindent\textbf{The $\alpha$-TQG model :}
In the following, we consider a `regularized' version of \eqref{ce}--\eqref{constrt}.  Since our notion of a solution to \eqref{ce}--\eqref{constrt} involves the pair $(b,\omega)$ (and not explicitly in terms  of $\bu$), henceforth, we will use the triple $(b^\alpha,\omega^\alpha, T)$ to denote the corresponding solution to the following  $\alpha$-TQG model for the avoidance of confusion. To be precise, for fixed $\alpha>0$, we will be exploring the pair $(b^\alpha,w^\alpha)$ that solves
\begin{align}
\frac{\partial}{\partial t} b^\alpha + {\bu}^\alpha \cdot \nabla b^\alpha =0,
\label{ceAlpha}
\\
\frac{\partial}{\partial t}\omega^\alpha + {\bu}^\alpha\cdot \nabla( \omega^\alpha-b^\alpha) = -\bu_h \cdot \nabla b^\alpha \label{meAlpha}
\end{align}
but where now,  
\begin{align}
\label{constrtAlpha}
{\bu}^\alpha =\nabla^\perp \psi^\alpha, 
\qquad
\bu_h =\frac{1}{2} \nabla^\perp h,
\qquad
\omega^\alpha=(\Delta -1) (1-\alpha \Delta ) \psi^\alpha +f.
\end{align}
Note that at least formally, this implies that
\begin{align}
\label{utilde}
{\bu}^\alpha= \nabla^\perp (1-\alpha \Delta )^{-1}(\Delta-1)^{-1}(\omega^\alpha -f).
\end{align}
We also note that even though the velocity fields $\bu$ in \eqref{constrt} and ${\bu}^\alpha$ in \eqref{utilde} are different, the coupled equations \eqref{ce}--\eqref{me} and \eqref{ceAlpha}--\eqref{meAlpha} are exactly of the same form. 
In the following, we claim that the exact same result, Theorem \ref{thm:MaxSol} holds for the $\alpha$-TQG model \eqref{ceAlpha}--\eqref{constrtAlpha} introduced above.
\begin{theorem}
\label{thm:mainAlpha}
Fix $\alpha>0$. There exist a unique maximal solution $(b^\alpha,\omega^\alpha, T_{\max}^\alpha)$ of \eqref{ceAlpha}--\eqref{constrtAlpha} under Assumption \ref{dataAllThroughout}.
\end{theorem}
\begin{remark}
\label{remark1} 
The fact that Theorem \ref{thm:mainAlpha} holds true is hardly surprising considering that we are basically treating the exact same model as \eqref{ce}--\eqref{me} albeit  different velocity fields. 
Heuristically, one observes that for very small $\alpha$, $\psi^\alpha \sim \tilde{\psi} \sim \psi$ and thus, both models are the same. Nevertheless, rigorously, the key lies in the estimate  \eqref{lem:MasterEst} which still holds true when one replaces $\bu$ with ${\bu}^\alpha$ and where now, the triple $({\bu}^\alpha,\omega^\alpha,f)$ satisfies
\begin{align}
\label{constrtxx}
{\bu}^\alpha =\nabla^\perp \psi^\alpha, 
\qquad
(\omega^\alpha -f)=(\Delta -1)(1-\alpha \Delta ) \psi^\alpha.
\end{align}
for a given $f$ having sufficient regularity.
\end{remark}
Indeed, Theorem \ref{thm:mainAlpha} follow from the following lemma.
\begin{lemma}
\label{lem:mainAlpha}
Under Assumption \ref{dataAllThroughout},  there exists a unique solution $(b^\alpha,\omega^\alpha, T^\alpha)$ of \eqref{ceAlpha}--\eqref{constrtAlpha} for any $\alpha>0$. Moreover,  
$(b^\alpha,\omega^\alpha, T_{\max}^\alpha)$ is a unique maximal solution of \eqref{ceAlpha}--\eqref{constrtAlpha}.
\end{lemma}
\begin{proof}
The existence of a unique solution $(b^\alpha,\omega^\alpha, T^\alpha)$ of \eqref{ceAlpha}--\eqref{constrtAlpha} for any $\alpha>0$ follow from Proposition \ref{prop:main}, Lemma \ref{lem:cweakTocstrong} and Corollary \ref{cor:Uniqueness}
with $T^\alpha = T$. We will prove all these results later for the TQG and one will observe that nothing changes for the $\alpha$-TQG. Note that since \eqref{ce}--\eqref{constrt} and \eqref{ceAlpha}--\eqref{constrtAlpha} share the same data, Assumption \ref{dataAllThroughout}, the local existence time $T$ is accordingly independent of $\alpha>0$.
\\
The construction of the maximal solution $(b^\alpha,\omega^\alpha, T_{\max}^\alpha)$   of \eqref{ceAlpha}--\eqref{constrtAlpha} also follow the same gluing argument used in showing the existence of the maximal solution $(b,\omega, T_{\max})$ for \eqref{ce}--\eqref{constrt} in Section \ref{subset:maximal}.
\end{proof}

\section{Construction of strong solutions}
\label{sec:construct}
In this section, we give a construction of a solution, in the sense of Definition \ref{def:solution}, to the thermal quasi-geostrophic system of equations \eqref{ce}--\eqref{constrt}. We will achieve this goal by rewriting our set of equations \eqref{ce}--\eqref{constrt} in an abstract form  and show that the resulting operator acting on the pair $(b,\omega)$ satisfies the assumptions of Theorem \ref{thm:appen} in the appendix. The details are as follows. 
We consider
\begin{align}
\label{abstractEq}
\frac{\partial}{\partial t} \binom{b}{\omega} 
+
\mathcal{A}
\begin{pmatrix}
 b \\
 \omega
 \end{pmatrix}
 =0,
\qquad
t\geq0
 ,
\qquad
b
\big\vert_{t=0}
=
b_0
 ,
\qquad
\omega
\big\vert_{t=0}
=
\omega_0
\end{align}
where the operator
\begin{align}
\mathcal{A}
:=
\begin{bmatrix}
\bu\cdot\nabla & 0    \\[0.3em]
(\bu_h-\bu)\cdot\nabla &\bu\cdot \nabla
\end{bmatrix}
\end{align}
is defined such that
\begin{align}
\bu =\nabla^\perp \psi, 
\qquad
\bu_h =\frac{1}{2} \nabla^\perp h,
\qquad
\omega=(\Delta -1) \psi +f.
\end{align}
\subsection{Estimates for convective terms}
\label{sec:estConv}
In the following, we let $k\in \mathbb{N}\cup\{0\}$ and recall that $W^{k,p}(\mathbb{T}^2)$ is a Banach algebra when $kp>2$. By using Sobolev embeddings and \eqref{lem:MasterEst}, we obtain the following estimates.
\begin{enumerate}
\item[(1b)] Let $k\in \{0,1,2\}$.  If $b\in W^{k+1,2}(\mathbb{T}^2)$, $\omega\in W^{2,2}(\mathbb{T}^2)$ and $\bu\in W^{3,2}_{\mathrm{div}}(\mathbb{T}^2;\mathbb{R}^2)$ solves \eqref{constrt} for a given $f\in W^{2,2}(\mathbb{T}^2)$ and a given $\bu_h\in W^{3,2}_{\mathrm{div}}(\mathbb{T}^2;\mathbb{R}^2)$, then 
\begin{align*}
&\Vert \bu\cdot\nabla b\Vert_{k,2} 
\lesssim
\Vert \bu\Vert_{2,2}\Vert b \Vert_{k+1,2}
\lesssim
\big(1+\Vert \omega\Vert_{1,2} \big)\Vert b \Vert_{k+1,2}
\end{align*}
and
\begin{align*}
&\Big\vert \big\langle\bu\cdot\nabla b\,,\, b \big\rangle_{k,2} \Big\vert
\lesssim
\Vert \bu\Vert_{3,2}\Vert b \Vert_{k,2}^2
\lesssim
\big(1+\Vert \omega\Vert_{2,2} \big)\Vert b \Vert_{k,2}^2,
\\
&\Big\vert \big\langle\bu\cdot\nabla b\,,\, \omega \big\rangle_{k,2} \Big\vert
\lesssim
\Vert \bu\Vert_{2,2}\Vert b \Vert_{k+1,2}\Vert \omega\Vert_{k,2}
\lesssim
\big(1+\Vert \omega\Vert_{1,2} \big)\Vert b \Vert_{k+1,2}\Vert \omega\Vert_{k,2},
\\
&\Big\vert \big\langle\bu _h\cdot\nabla b\,,\, \omega \big\rangle_{k,2} \Big\vert
\lesssim
\Vert \bu_h\Vert_{2,2}\Vert b \Vert_{k+1,2}\Vert \omega\Vert_{k,2}
\lesssim
\Vert b \Vert_{k+1,2}\Vert \omega\Vert_{k,2}.
\end{align*}
\item[(1b)] Let $k\geq 3$.  If $b\in W^{k+1,2}(\mathbb{T}^2)$, $\omega\in W^{k,2}(\mathbb{T}^2)$ and $\bu\in W^{k+1,2}_{\mathrm{div}}(\mathbb{T}^2;\mathbb{R}^2)$ solves \eqref{constrt} for a given $f\in W^{k,2}(\mathbb{T}^2)$ and a given $\bu_h\in W^{k+1,2}_{\mathrm{div}}(\mathbb{T}^2;\mathbb{R}^2)$, then 
\begin{align*}
&\Vert \bu\cdot\nabla b\Vert_{k,2} 
\lesssim
\Vert \bu\Vert_{2,2}\Vert b \Vert_{k+1,2}
+\Vert \bu\Vert_{k,2}\Vert b \Vert_{3,2}
\lesssim
\big(1+\Vert \omega\Vert_{k-1,2} \big)\Vert b \Vert_{k+1,2}
\end{align*}
and
\begin{align*}
&\Big\vert \big\langle\bu\cdot\nabla b\,,\, b \big\rangle_{k,2} \Big\vert
\lesssim
\Vert \bu\Vert_{3,2}\Vert b \Vert_{k,2}^2
+
\Vert \bu\Vert_{k,2}\Vert b \Vert_{k,2}\Vert b \Vert_{3,2}
\lesssim
\big(1+\Vert \omega\Vert_{k-1,2} \big)\Vert b \Vert_{k,2}^2,
\\
&\Big\vert \big\langle\bu\cdot\nabla b\,,\, \omega \big\rangle_{k,2} \Big\vert
\lesssim
\big(\Vert \bu\Vert_{2,2}\Vert b \Vert_{k+1,2}
+
\Vert \bu\Vert_{k,2}\Vert b \Vert_{3,2}\big) \Vert \omega \Vert_{k,2}
\lesssim
\big(1+\Vert \omega\Vert_{k-1,2} \big)\Vert b \Vert_{k+1,2} \Vert \omega \Vert_{k,2},
\\
&\Big\vert \big\langle\bu_h\cdot\nabla b\,,\, \omega \big\rangle_{k,2} \Big\vert
\lesssim
\big(\Vert \bu_h\Vert_{2,2}\Vert b \Vert_{k+1,2}
+
\Vert \bu_h\Vert_{k,2}\Vert b \Vert_{3,2}\big) \Vert \omega \Vert_{k,2}
\lesssim
\Vert b \Vert_{k+1,2} \Vert \omega \Vert_{k,2}.
\end{align*}
\item[(2a)] Let $k\in\{0,1,2\}$. If $\omega\in W^{k+1,2}(\mathbb{T}^2)$ and that $\bu\in W^{3,2}_{\mathrm{div}}(\mathbb{T}^2;\mathbb{R}^2)$ solves \eqref{constrt} for a given $f\in W^{2,2}(\mathbb{T}^2)$, then
\begin{align*}
\Vert \bu\cdot\nabla \omega\Vert_{k,2}
\lesssim
\Vert \bu\Vert_{2,2}\Vert \omega \Vert_{k+1,2}
\lesssim
\big(1+\Vert \omega\Vert_{1,2} \big)\Vert \omega \Vert_{k+1,2}
\end{align*}
and
\begin{align*}
&\Big\vert \big\langle\bu\cdot\nabla \omega\,,\, \omega \big\rangle_{k,2} \Big\vert
\lesssim
\Vert \bu\Vert_{3,2}\Vert \omega \Vert_{k,2}^2
\lesssim
\big(1+\Vert \omega\Vert_{2,2} \big)\Vert \omega \Vert_{k,2}^2.
\end{align*}
\item[(2b)] Let $k\geq 3$.  If $\omega\in W^{k+1,2}(\mathbb{T}^2)$ and $\bu\in W^{k+1,2}_{\mathrm{div}}(\mathbb{T}^2;\mathbb{R}^2)$ solves \eqref{constrt} for a given $f\in W^{k,2}(\mathbb{T}^2)$, then
\begin{align*}
&\Vert \bu\cdot\nabla \omega\Vert_{k,2} 
\lesssim
\Vert \bu\Vert_{2,2}\Vert \omega \Vert_{k+1,2}
+\Vert \bu\Vert_{k,2}\Vert \omega \Vert_{3,2}
\lesssim
\big(1+\Vert \omega\Vert_{k-1,2} \big)\Vert \omega \Vert_{k+1,2}
\end{align*}
and
\begin{align*}
&\Big\vert \big\langle\bu\cdot\nabla \omega\,,\, \omega \big\rangle_{k,2} \Big\vert
\lesssim
\Vert \bu\Vert_{3,2}\Vert \omega \Vert_{k,2}^2
+
\Vert \bu\Vert_{k,2}\Vert \omega \Vert_{k,2}\Vert \omega \Vert_{3,2}
\lesssim
\big(1+\Vert \omega\Vert_{k-1,2} \big)\Vert \omega \Vert_{k,2}^2.
\end{align*}
\end{enumerate}
With these estimates in hand, we can now proceed to prove the existence of a local strong solution of \eqref{abstractEq}.
\subsection{Proof of local existence}
Our proof of a local strong solution to \eqref{ce}--\eqref{constrt} will follow from the following proposition. Its proof will follow the argument presented in \cite{kato1984nonlinear} (See Appendix, Section \ref{sec:appendix}) for the construction of a local solution to the Euler equation.
\begin{proposition}
\label{prop:main}
Take Assumption \ref{dataAllThroughout}.
There exists   a  solution $(b, \omega, T) $ of \eqref{abstractEq} such that:
\begin{enumerate}
\item the time $T>0$ satisfy the bound
\begin{align}
\label{requireTimeA}
T<\frac{1}{2c(1+\Vert(b_0,\omega_0)\Vert^2_{\mathcal{M}})}, \quad c=c(\Vert\bu_h\Vert_{3,2},\Vert f \Vert_{2,2});
\end{align}
\item the pair $(b,\omega)$ is of class
\begin{align}
\label{weakContBAndOmegaA}
b \in C_w([0,T];W^{3,2}(\mathbb{T}^2)),
\qquad
\omega \in C_w([0,T];W^{2,2}(\mathbb{T}^2)) ;
\end{align}
\item the pair $(b,\omega)$ satisfies the bound
\begin{align}
\label{weakContBAndOmega1A}
\Vert (b, \omega)(t) \Vert_\mathcal{M}^2 
\leq
\frac{\Vert(b_0,\omega_0)\Vert^2_{\mathcal{M}}+2c(1+\Vert(b_0,\omega_0)\Vert^2_{\mathcal{M}})\,t}{1-2c(1+\Vert(b_0,\omega_0)\Vert^2_{\mathcal{M}})\,t},\quad\quad t\in (0,T).
\end{align}
\end{enumerate}
\end{proposition}
\begin{proof}
Let us consider the following triplet $\{V,H,X\}$ given as follows. Take $X:=W^{1,2}(\mathbb{T}^2)\times L^2(\mathbb{T}^2)$
and let $H:=W^{(3,2),2}(\mathbb{T}^2)\times W^{(2,1),2}(\mathbb{T}^2)$ be a Hilbert space endowed with the inner product
\begin{equation}
\begin{aligned}
\big\langle (b,\omega),
(b',\omega') \big\rangle_H
=
 \langle b\,,\,b' \rangle_{3,2}
 +
  \langle b\,,\,b' \rangle_{2,2}
+
 \langle \omega\,,\,\omega' \rangle_{2,2}
 +
  \langle \omega\,,\,\omega' \rangle_{1,2}.
\end{aligned}
\end{equation}
We note that for $(b,\omega) \in \mathcal{M}$, we have that
\begin{equation}
\begin{aligned}
\label{equiNorm}
\Vert (b,\omega)\Vert_{ \mathcal{M}}
\leq
\Vert(b,\omega)\Vert_H
\leq
2
\Vert (b,\omega)\Vert_{ \mathcal{M}} 
\end{aligned}
\end{equation}
and thus, $H$ is equivalent to $  \mathcal{M}$. Also, by virtue of the estimates shown in Section \ref{sec:estConv}, we can conclude that the operator $\mathcal{A}$ is weakly continuous from  $H$ into $X$.
\\
Next, we let $V$ be the domain of an unbounded selfadjoint operator $S\geq 0$ in $X$ with $\mathrm{domain}(S)\subset H$ and $\langle S u,v\rangle =\langle u,v \rangle_H$ for $u\in \mathrm{domain}(S)$. More precisely, 
$
S=
\sum_{\vert \beta \vert\leq 3}(-\partial)^\beta\partial^\beta
$
subject to periodic boundary conditions. With this definition of $V$ in hand, we can again conclude from the estimates  in Section \ref{sec:estConv} that $\mathcal{A}$ maps  $V$ into $H$.
Our goal now is to show that for any $(b,\omega)\in V$, there exists an increasing function $\rho(r)\geq 0$ of $r\geq0$ such that
\begin{equation}
\begin{aligned}
\Big\vert \big\langle \mathcal{A}(b,\omega),
(b,\omega) \big\rangle_H \Big\vert
\leq \rho
\big(\Vert (b,\omega)\Vert^2_H \big).
\end{aligned}
\end{equation}
To achieve this goal, we further refine the following estimates from Section \ref{sec:estConv}. In particular, we recall that if $b\in W^{k+1,2}(\mathbb{T}^2)$ and $\omega\in W^{2,2}(\mathbb{T}^2)$ and that $\bu\in W^{3,2}_{\mathrm{div}}(\mathbb{T}^2;\mathbb{R}^2)$ solves \eqref{constrt} for a given $f\in W^{2,2}(\mathbb{T}^2)$ and a given $\bu_h\in W^{3,2}_{\mathrm{div}}(\mathbb{T}^2;\mathbb{R}^2)$, then for  $k\in \{0,1, 2\}$,
\begin{equation}
\begin{aligned}
\label{convEst1}
&\Big\vert \big\langle\bu\cdot\nabla b\,,\, b \big\rangle_{k,2} \Big\vert
\lesssim
\big(1+\Vert \omega\Vert_{2,2} \big)\Vert b \Vert_{k+1,2}^2
\lesssim
\big(1+ \Vert (b,\omega)\Vert^2_H \big)\Vert b \Vert_{3,2}^2
,
\\
&\Big\vert \big\langle\bu\cdot\nabla \omega\,,\,\omega \big\rangle_{k,2} \Big\vert
\lesssim
\big(1+\Vert \omega\Vert_{2,2} \big)\Vert \omega \Vert_{2,2}^2
\lesssim
\big(1+ \Vert (b,\omega)\Vert^2_H \big)\Vert \omega \Vert_{2,2}^2
,
\\
&\Big\vert \big\langle\bu\cdot\nabla b\,,\,\omega \big\rangle_{k,2} \Big\vert
\lesssim
\big(1+\Vert \omega\Vert_{2,2} \big)\Vert b \Vert_{k+1,2}\Vert \omega \Vert_{2,2}
\lesssim
\big(1+ \Vert (b,\omega)\Vert^2_H \big)\big(\Vert b \Vert_{3,2}^2+ \Vert \omega \Vert_{2,2}^2 \big)
,
\\
&\Big\vert \big\langle\bu_h\cdot\nabla b\,,\,\omega \big\rangle_{k,2} \Big\vert
\lesssim
\Vert b \Vert_{k+1,2}\Vert \omega \Vert_{2,2}
\lesssim
\big(\Vert b \Vert_{3,2}^2+ \Vert \omega \Vert_{2,2}^2 \big)
\lesssim
\big(1+ \Vert (b,\omega)\Vert^2_H \big).
\end{aligned}
\end{equation}
In the above estimates, we have used inequalities such as $x\lesssim 1+x^2$ for $x\geq0$ and Young's inequalities.
Furthermore, if $b\in W^{3,2}(\mathbb{T}^2)$, $\omega\in W^{2,2}(\mathbb{T}^2)$ and $\bu\in W^{3,2}_{\mathrm{div}}(\mathbb{T}^2;\mathbb{R}^2)$,  we also have that
\begin{equation}
\begin{aligned}
\label{convEst2}
&\Big\vert \big\langle\bu\cdot\nabla b\,,\, b \big\rangle_{3,2} \Big\vert
\lesssim
\big(1+\Vert \omega\Vert_{2,2} \big)\Vert b \Vert_{3,2}^2
\lesssim
\big(1+ \Vert (b,\omega)\Vert^2_H \big)\Vert b \Vert_{3,2}^2.
\end{aligned}
\end{equation}
We can therefore conclude that for $(b,\omega)\in H$ (and in particular, for any $(b,\omega)\in V$),
\begin{equation}
\begin{aligned}
\Big\vert \big\langle \mathcal{A}(b,\omega),
(b,\omega) \big\rangle_H \Big\vert
\lesssim
\big(1+ \Vert (b,\omega)\Vert^2_H \big)^2
\end{aligned}
\end{equation}
with a constant depending only on $\Vert\bu_h\Vert_{3,2}$ and $\Vert f \Vert_{2,2}$.
Since $\rho(r)=c(1+r)^2\geq0$ is an increasing function of $r\geq 0$, we can  conclude from the result in the Appendix, Section \ref{sec:appendix}, that  for a given $(b_0,\omega_0) \in H \equiv \mathcal{M}$, there exist a  solution $(b, \omega, T) $ of \eqref{abstractEq} of class
\begin{align}
\label{weakContBAndOmega}
b \in C_w([0,T];W^{3,2}(\mathbb{T}^2)),
\qquad
\omega \in C_w([0,T];W^{2,2}(\mathbb{T}^2))
\end{align}
satisfying the bound (recall \eqref{equiNorm})
\begin{align}
\label{weakContBAndOmega1}
\Vert (b, \omega)(t) \Vert_\mathcal{M}^2 \leq
\Vert (b, \omega)(t) \Vert_H^2 \leq r(t),\quad\quad t\in (0,T)
\end{align}
where $r$ is an increasing function on $(0,T)$.
The time $T>0$ in \eqref{weakContBAndOmega}--\eqref{weakContBAndOmega1} depends only on $\Vert\bu_h\Vert_{3,2}$ and $\Vert f \Vert_{2,2}$ by way of the function $\rho(\cdot)$ as well as on $\Vert(b_0,\omega_0)\Vert_{\mathcal{M}}$. To be precise, for $\rho(r)=c(1+r)^2\geq0$ where $c>0$ depends only on $\Vert\bu_h\Vert_{3,2}$ and $\Vert f \Vert_{2,2}$, we obtain
$T>0$ by solving the equation
\begin{align}
\label{ode}
\frac{\dd}{\dd t}r = 2 \rho(r),\qquad \qquad r(0)= \Vert(b_0,\omega_0)\Vert^2_{\mathcal{M}}
\end{align}
which yields
\begin{align}
\label{r}
r(t) = \frac{\Vert(b_0,\omega_0)\Vert^2_{\mathcal{M}}+2c(1+\Vert(b_0,\omega_0)\Vert^2_{\mathcal{M}})\,t}{1-2c(1+\Vert(b_0,\omega_0)\Vert^2_{\mathcal{M}})\,t}
.
\end{align}
We therefore require
\begin{align}
\label{requireTime}
T<\frac{1}{2c(1+\Vert(b_0,\omega_0)\Vert^2_{\mathcal{M}})}
\end{align}
for a solution of \eqref{ode} to exist. This finishes the proof. 
\end{proof}
\begin{remark}
\label{rem:strongContZero}
Note that $r(t)$ converges  to $\Vert(b_0,\omega_0)\Vert^2_{\mathcal{M}}$ as $t\rightarrow 0^+$. Therefore, given the bound \eqref{weakContBAndOmega1} and \eqref{r}, we can conclude that 
\begin{align*}
    \limsup_{t\rightarrow 0^+} \Vert (b,\omega)(t) \Vert^2_{\mathcal{M}}\leq \Vert (b_0,\omega_0)\Vert^2_{\mathcal{M}}.
\end{align*}
On the other hand, because the pair $(b,\omega)$ is of class \eqref{weakContBAndOmega}, we can also conclude that
\begin{align*}
    \liminf_{t\rightarrow 0^+} \Vert (b,\omega)(t) \Vert^2_{\mathcal{M}}\geq \Vert (b_0,\omega_0)\Vert^2_{\mathcal{M}}
\end{align*}
and thus, we obtain strong right continuity at time $t=0$. Since our system \eqref{ce}--\eqref{constrt} is time reversible (see Remark \ref{timeReverse} below), we also obtain strong left continuity at $t=0$ so that we are able to conclude that \eqref{weakContBAndOmega} is actually strongly continuous at $t=0$.
\end{remark}
\begin{remark}
\label{timeReverse}
By time-reversible, we mean that if $b(t,x)$  and $\omega(t,x)$  (with $\omega(t,x)$ related to $\bu(t,x)$ by \eqref{constrt}) solves \eqref{ce}--\eqref{constrt} with data $b_0(x)$, $\omega_0(x)$, $\bu_h(x)$ and $f(x)$, then $-b(-t,x)$ and  $-\omega(-t,x)$  (with $-\omega(-t,x)$ related to $-\bu(-t,x)$ by \eqref{constrt}) solves \eqref{ce}--\eqref{constrt} with data $-b_0(x)$, $-\omega_0(x)$, $-\bu_h(x)$ and $-f(x)$ respectively.
\end{remark}
\begin{remark}
We remark that the function $\rho$ and as such, the function $r$ solving \eqref{ode} is not unique. In fact, one can construct countably many (if not infinitely many) of these functions by varying the estimates for the convective terms on the left-hand sides of \eqref{convEst1}--\eqref{convEst2}. For example, a different choice of H\"older conjugates for estimating the aforementioned terms will lead to a different $\rho$ and $r$. The importance of this remark lies in the fact that $\rho$ and $r$ determines the longevity of the  solution $(b,\omega, T)$ to \eqref{ce}--\eqref{constrt}.  One can therefore tune $T>0$  by modifying $\rho$ and $r$ accordingly. For example, we may also obtain the following
\begin{align}
\rho_1(r)=c(r^2+r), \qquad \quad \rho_2(r)=c(r^\frac{3}{2}+r)
\end{align}
for which the corresponding solutions to \eqref{ode} are
\begin{align}
&r_1(t)=\frac{r(0)e^{2ct}}{r(0)+1-r(0)e^{2ct}},
\\
&r_2(t) =\frac{s^2+s\pm 2s\sqrt{s}}{s^2-2s+1}, \quad \text{where} \quad s=s(t)=\frac{r(0)}{(\sqrt{r(0)}+1)^2}e^{2ct}.
\end{align}
respectively. Again, the constants depends only on $\Vert\bu_h\Vert_{3,2}$ and $\Vert f \Vert_{2,2}$. The second solution $r_2$ is  unsuitable since for one, it is twofold. The first solution $r_1$ however means that we require a time
\begin{align*}
T_1 <\frac{1}{2c}\ln\bigg(1+\frac{1}{\Vert(b_0,\omega_0)\Vert^2_{\mathcal{M}}}\bigg)
\end{align*}  for a solution of \eqref{ode} to exist. 
Therefore, if we  optimize the constant $c>0$ so that they are the same throughout and we compare $r_1$ with $r$ given in \eqref{r}, we are able to conclude that $r$ yields a longer time of existence of a  solution $(b,\omega, T)$ to \eqref{ce}--\eqref{constrt} as compared to $r_1$ since
\begin{align*}
\frac{1}{2c}\ln\bigg(1+\frac{1}{\Vert(b_0,\omega_0)\Vert^2_{\mathcal{M}}}\bigg)
 \leq
\frac{1}{2c(1+\Vert(b_0,\omega_0)\Vert^2_{\mathcal{M}})}. 
\end{align*}
\end{remark}
In the next section, we show an estimate for the difference of two solutions constructed above  after which we are able to strengthen the weak continuity \eqref{weakContBAndOmega} to a strong one for all times of existence.

\subsection{Difference estimate}
\label{subsec:unique} 
In the following, we let $(b^1,\omega^1, T)$ and $(b^2,\omega^2, T)$ be  two solutions of \eqref{ce}--\eqref{constrt}  sharing the same data $\bu_h \in W^{3,2}_{\mathrm{div}}(\mathbb{T}^2;\mathbb{R}^2)$,  $f\in W^{2,2}(\mathbb{T}^2)$ and $(b_0, \omega_0) $. So in particular,
\begin{align}
\label{inPart}
\sup_{t\in[0,T)}\Vert b^i(t,\cdot) \Vert_{3,2}<\infty, \qquad \sup_{t\in[0,T)}\Vert \omega^i(t,\cdot) \Vert_{2,2}<\infty
\end{align}
hold for each $i=1,2$.
We now set 
  $b^{12}:=b^1-b^2$, $\omega^{12}:=\omega^1-\omega^2$ and $\bu^{12}:=\bu^1-\bu^2$ so that  $(b^{12}, \bu^{12}, \omega^{12})$ satisfies
\begin{align}
\frac{\partial}{\partial t} b^{12} + \bu^{12} \cdot \nabla b^1 +\bu^2\cdot\nabla b^{12} =0,
\label{ceDiff}
\\
\frac{\partial}{\partial t}\omega^{12} + \bu^{12}\cdot \nabla( \omega^1-b^1) +\bu^2\cdot \nabla (\omega^{12} - b^{12}) = -\bu_h \cdot \nabla b^{12}, \label{meDiff}
\end{align}
where
\begin{align}
\label{constrtDiff}
\bu^{12} =\nabla^\perp \psi^{12}, 
\qquad
\bu_h =\frac{1}{2} \nabla^\perp h,
\qquad
\omega^{12}=(\Delta -1) \psi^{12}.
\end{align}
We now show the following stability and uniqueness results in the  space $W^{2,2}(\mathbb{T}^2) \times W^{1,2}(\mathbb{T}^2)$ which is larger than the space $W^{3,2}(\mathbb{T}^2) \times W^{2,2}(\mathbb{T}^2)$ of existence of the individual solutions. Indeed, it suffices to show uniqueness in the even larger space $W^{1,2}(\mathbb{T}^2) \times L^2(\mathbb{T}^2)$ but we avoid doing so since we wish to concurrently show the proof of Theorem \ref{thm:stability} as well.  
\begin{proof}[Proof of Theorem \ref{thm:stability} and Corollary \ref{cor:Uniqueness}]
If we apply $\partial^\beta $ to \eqref{ceDiff} with $\vert \beta \vert\leq 2$, we obtain
\begin{align}
\frac{\partial}{\partial t} \partial^\beta b^{12}  +\bu^2\cdot\nabla \partial^\beta b^{12}  = S_1 +  S_2 -\bu^{12} \cdot \nabla \partial^\beta b^1,
\label{ceDiff1}
\end{align} 
where 
\begin{align*}
S_1&:= \bu^{12} \cdot \partial^\beta  \nabla b^1
-
\partial^\beta(\bu^{12} \cdot \nabla  b^1 ),
\\
S_2&:= \bu^2\cdot \partial^\beta  \nabla b^{12}
-
\partial^\beta(\bu^2 \cdot \nabla  b^{12} )
\end{align*}
are such that
\begin{align}
\label{estS1}
\Vert S_1 \Vert_2 &\lesssim \Vert \nabla \bu^{12} \Vert_{4}\Vert b^1 \Vert_{2,4}
+
\Vert \nabla b^1 \Vert_{\infty} \Vert \bu^{12} \Vert_{2,2}
\lesssim \Vert  \bu^{12} \Vert_{2,2}\Vert b^1 \Vert_{3,2},
\\
\label{estS2}
\Vert S_2 \Vert_2 &\lesssim \Vert \nabla \bu^2 \Vert_{\infty}\Vert b^{12} \Vert_{2,2}
+
\Vert \nabla b^{12} \Vert_{4} \Vert \bu^2 \Vert_{2,4}
\lesssim \Vert  \bu^2 \Vert_{3,2}\Vert b^{12} \Vert_{2,2}.
\end{align}
Additionally, the following estimate for the $L^2$ inner product of the last term \eqref{ceDiff1} with $\partial^\beta b^{12}$ holds
\begin{equation}
\begin{aligned}
\Big\vert\big\langle \bu^{12} \cdot \nabla \partial^\beta b^1 \, ,\, \partial^\beta b^{12} \big\rangle \Big\vert
\lesssim
\Vert b^1 \Vert_{3,2}
\Vert \bu^{12} \Vert_\infty
 \Vert b^{12} \Vert_{2,2}.
\end{aligned}
\end{equation}
If we test \eqref{ceDiff1} with $\partial^\beta b^{12}$ and sum over $\vert \beta \vert\leq 2$, we obtain from the above estimates together with \eqref{lem:MasterEst}  and \eqref{inPart},
\begin{equation}
\begin{aligned}
\label{b32estUniqY}
\frac{\dd}{\dd t}
\Vert  b^{12} \Vert_{2,2}^2 
&\lesssim
\Vert \omega^{12} \Vert_{1,2}^2 +
\Vert  b^{12} \Vert_{2,2}^2.
\end{aligned}
\end{equation}
Next, we apply $\partial^\beta$ to \eqref{meDiff} with $\vert \beta\vert \leq 1$ to get
\begin{equation}
\begin{aligned}
\frac{\partial}{\partial t}\partial^\beta \omega^{12} +\bu^2\cdot \nabla \partial^\beta \omega^{12}  =&
- \bu^{12}\cdot \nabla \partial^\beta( \omega^1
 - b^1)
+(\bu^2-\bu_h)\cdot \nabla \partial^\beta  b^{12}
\\&
+S_3+ \ldots+S_7\label{meDiff1}
\end{aligned}
\end{equation}
where
\begin{align*}
S_3&:= \bu^2 \cdot \partial^\beta  \nabla \omega^{12}
-
\partial^\beta(\bu^2 \cdot \nabla  \omega^{12} ),
\\
S_4&:= \bu^{12}\cdot \partial^\beta  \nabla \omega^1
-
\partial^\beta(\bu^{12} \cdot \nabla  \omega^1 ),
\\
S_5&:= -\bu^{12}\cdot \partial^\beta  \nabla b^1
+
\partial^\beta(\bu^{12} \cdot \nabla  b^1 ),
\\
S_6&:= -\bu^2\cdot \partial^\beta  \nabla b^{12}
+
\partial^\beta(\bu^2 \cdot \nabla  b^{12} ),
\\
S_7&:= \bu_h \cdot \partial^\beta  \nabla b^{12}
-
\partial^\beta(\bu_h \cdot \nabla  b^{12} )
\end{align*}
are such that
\begin{align}
\label{estS3}
\Vert S_3 \Vert_2 &\lesssim 
 \Vert \bu^2 \Vert_{1,\infty}\Vert \omega^{12} \Vert_{1,2}
 \lesssim 
\Vert \omega^{12} \Vert_{1,2},
\\
\label{estS4}
\Vert S_4 \Vert_2 &\lesssim \Vert  \bu^{12} \Vert_{1,4}\Vert \omega^1 \Vert_{1,4}
\lesssim \Vert  \bu^{12} \Vert_{2,2} \lesssim 
\Vert \omega^{12} \Vert_{1,2},
\\
\label{estS5}
\Vert S_5 \Vert_2 &\lesssim \Vert  \bu^{12} \Vert_{1,4}\Vert b^1 \Vert_{1,4}\lesssim \Vert  \bu^{12} \Vert_{2,2} \lesssim 
\Vert \omega^{12} \Vert_{1,2},
\\
\label{estS6}
\Vert S_6 \Vert_2 &\lesssim 
 \Vert \bu^2 \Vert_{1,\infty}\Vert b^{12} \Vert_{1,2}
 \lesssim 
 \Vert b^{12} \Vert_{1,2} \lesssim 
 \Vert b^{12} \Vert_{2,2},
 \\
\label{estS7}
\Vert S_6 \Vert_2 &\lesssim 
 \Vert \bu_h \Vert_{1,\infty}\Vert b^{12} \Vert_{1,2}
 \lesssim 
 \Vert b^{12} \Vert_{1,2} \lesssim 
 \Vert b^{12} \Vert_{2,2}.
\end{align}
Also, since $\mathrm{div}(\bu^2)=0$
\begin{equation}
\begin{aligned}
&\Big\vert\big\langle \bu^2\cdot \nabla \partial^\beta \omega^{12} \, ,\, \partial^\beta \omega^{12} \big\rangle \Big\vert
=0,
\\&
\Big\vert\big\langle \bu^{12}\cdot \nabla \partial^\beta( \omega^1
 - b^1) \, ,\, \partial^\beta \omega^{12} \big\rangle \Big\vert
\lesssim
\Vert \omega^1- b^1 \Vert_{2,2}
\Vert \bu^{12} \Vert_\infty
 \Vert \omega^{12} \Vert_{1,2},
\\&
\Big\vert\big\langle (\bu^2-\bu_h)\cdot \nabla \partial^\beta  b^{12} \, ,\, \partial^\beta \omega^{12} \big\rangle \Big\vert
\lesssim
\Vert \bu^2-\bu_h \Vert_\infty
\Vert b^{12} \Vert_{2,2}
 \Vert \omega^{12} \Vert_{1,2}.
\end{aligned}
\end{equation}
Therefore, 
\begin{equation}
\begin{aligned}
\label{w32estUniqY}
\frac{\dd}{\dd t}
\Vert  \omega^{12} \Vert_{1,2}^2 
&\lesssim
\ \Vert \omega^{12} \Vert_{1,2}^2 +
\Vert  b^{12} \Vert_{2,2}^2.
\end{aligned}
\end{equation} 
From \eqref{b32estUniqY} and \eqref{w32estUniqY}, we can conclude from Gr\"onwall's lemma that
\begin{equation}
\begin{aligned}
 \Vert  b^{12}(t, \cdot) \Vert_{2,2}^2
+
\Vert \omega^{12}(t, \cdot) \Vert_{1,2}^2  
\leq \exp(ct)\big(\Vert  b^{12}_0\Vert_{2,2}^2
+ \Vert \omega^{12}_0 \Vert_{1,2}^2 
 \big) 
\end{aligned}
\end{equation} 
for all $t\geq 0$ where the constant $c>0$ depends  on $ (\Vert b^1\Vert_{3,2}, \Vert \omega^1\Vert_{2,2}, \Vert \omega^2\Vert_{2,2}, \Vert \bu_h\Vert_{3,2}, \Vert f\Vert_{2,2})$. This finishes the proof of Theorem \ref{thm:stability}.
To obtain uniqueness, i.e. Corollary \ref{cor:Uniqueness}, we use the fact that $b^{12}_0=0$ and $\omega^{12}_0=0$.
\end{proof}
\noindent Having shown uniqueness, we can conclude that the weakly continuous functions \eqref{weakContBAndOmegaA} are indeed strongly continuous  by using time-reversibility of \eqref{ce}--\eqref{me}.
\begin{lemma}
\label{lem:cweakTocstrong}
Take the assumptions of Proposition \ref{prop:main} to be true. Then
\begin{align*}
b \in C([0,T];W^{3,2}(\mathbb{T}^2)),
\qquad
\omega \in C([0,T];W^{2,2}(\mathbb{T}^2)).
\end{align*}
\end{lemma}
\begin{proof}
Firstly, we recall that
the solution  we constructed in Proposition \ref{prop:main} is actually strongly continuous at time $t=0$, recall Remark \ref{rem:strongContZero}. Now consider a solution $(b,\omega, T_0)$ of \eqref{ce}--\eqref{me} where $T_0 \in [0,T]$ is fixed but arbitrary. Then by \eqref{weakContBAndOmega1A}, this solution will satisfy the bound
\begin{align}
\label{bwt0}
\Vert (b, \omega)(T_0) \Vert^2_{\mathcal{M}}
\leq \frac{\Vert (b_0,\omega_0) \Vert^2_{\mathcal{M}} + 2c(1+\Vert(b_0,\omega_0)\Vert^2_{\mathcal{M}})T_0 }{1-2c(1+\Vert(b_0,\omega_0)\Vert^2_{\mathcal{M}})T_0 }
\end{align}
$c=c(\Vert\bu_h\Vert_{3,2}, \Vert f \Vert_{2,2})$. We can now construct a new solution $(b^1,\omega^1, T_0+T_1)$ by taking $(b, \omega)(T_0)$ as an initial condition. A repetition of the argument leading to \eqref{bwt0} means that this new pair $(b^1,\omega^1)$ solving \eqref{ce}--\eqref{me} will satisfy the inequality
\begin{align}
\label{scT0}
\Vert (b^1, \omega^1)(t) \Vert^2_{\mathcal{M}}
\leq \frac{\Vert(b,\omega)(T_0)\Vert^2_{\mathcal{M}} + 2c(1+\Vert(b,\omega)(T_0)\Vert^2_{\mathcal{M}})t }{1-2c(1+\Vert(b,\omega)(T_0)\Vert^2_{\mathcal{M}})t }
\end{align}
on $[T_0, T_0+T_1]$ where 
\begin{align*}
T_1<\frac{1}{2c(1+\Vert(b,\omega)(T_0)\Vert^2_{\mathcal{M}})}.
\end{align*}
Furthermore, $(b^1,\omega^1)$ must coincide with $(b,\omega)$ on $[T_0, T_0+T_1] \cap [0,T]$ by uniqueness and the fact that they agree at $T_0\in[0,T]$.
Subsequently, we obtain from \eqref{scT0}, strong right-continuity of $\Vert (b^1, \omega^1) \Vert_{\mathcal{M}}$ at time $t=T_0$ just as was done for $t=0$ in Remark \ref{rem:strongContZero}. Again, by uniqueness, this implies that $\Vert (b, \omega) \Vert_{\mathcal{M}}$ is also strongly right-continuous at time $t=T_0$.
Since $T_0$ was chosen arbitrarily, this means that    $(b,\omega)$ is strongly right-continuous on $[0,T]$. Since our system is reversible, we can repeat the above argument for the backward equation from which we obtain strong left-continuity on $[0,T]$. We have thus shown that weakly continuous solution \eqref{weakContBAndOmegaA} is indeed strongly continuous.
\end{proof}
%

\subsection{Maximal solution}
\label{subset:maximal}
We now end the section with a proof of the existence of a \textit{maximal solution} to the TQG model.
\begin{proof}[Proof of Theorem \ref{thm:MaxSol}]
By Proposition \ref{prop:main}, we found a time
\begin{align}
T<\frac{1}{2c(1+\Vert(b_0,\omega_0)\Vert^2_{\mathcal{M}})}
\end{align}
with $c=c(\Vert\bu_h \Vert_{3,2}, \Vert f \Vert_{2,2})$ such that $(b, \omega, T)$ is a unique solution of \eqref{ce}--\eqref{constrt}.
At time $T>0$ given above, we now choose $(b_T,\omega_T)\in \mathcal{M}$, $\bu_h \in W^{3,2}_{\mathrm{div}}(\mathbb{T}^2;\mathbb{R}^2)$ and $f\in W^{2,2}(\mathbb{T}^2)$ as our new data where
\begin{align*}
b_T:=b(T,\cdot), \qquad \omega_T:=\omega(T,\cdot).
\end{align*}
By repeating the argument, we can also find 
\begin{align}
\tilde{T}<\frac{1}{2c(1+\Vert(b_T,\omega_T)\Vert^2_{\mathcal{M}})}
\end{align}
with $c=c(\Vert\bu_h \Vert_{3,2}, \Vert f \Vert_{2,2})$, such that  with the new initial conditions $(b_T, \omega_T) $, $(b, \omega)$ uniquely solves \eqref{ce}--\eqref{constrt}, in the sense of Definition \ref{def:solution}, on $[T,T+ \tilde{T}]$. By a gluing argument, we obtain a solution $(b,\omega, T_1)$ of \eqref{ce}--\eqref{constrt} with $T_1=T+\tilde{T}$.
By this iterative procedure:
\begin{itemize}
    \item we obtain an increasing family $(T_n)){n\in \mathbb{N}}$ of time steps whose limit is $T_{\max}\in(0, \infty]$;
    \item for each $n\in\mathbb{N}$, $(b,\omega, T_n)$ is a solution of \eqref{ce}--\eqref{constrt};
    \item if $T_{\max}<\infty$, then
\begin{align*}
\limsup_{T_n\rightarrow T_{\max}} \Vert (b,\omega)(T_n) \Vert_{\mathcal{M}}^2 =\infty
\end{align*} 
\end{itemize}
since otherwise, we can repeat the procedure above to  obtain $\tilde{\tilde{T}}>0$ satisfying
\begin{align*}
\tilde{\tilde{T}}<\frac{1}{2c(1+\Vert(b,\omega)(T_{\max})\Vert^2_{\mathcal{M}})}
\end{align*}
with $c=c(\Vert\bu_h \Vert_{3,2}, \Vert f \Vert_{2,2})$, such that  $(b,\omega, T_{\max}^{\max})$ is a solution of \eqref{ce}--\eqref{constrt} with $T_{\max}^{\max}=T_{\max}+\tilde{\tilde{T}}$. This will 
contradict the fact that $T_{\max}$ is the maximal time.
\end{proof}

\section{Convergence of $\alpha$-TQG to TQG}
\label{sec:converge}
\noindent In this section, for a given $\bu_h \in W^{3,2}_{\mathrm{div}}(\mathbb{T}^2;\mathbb{R}^2)$ and $f\in W^{2,2}(\mathbb{T}^2)$ serving as data for both the TQG \eqref{ce}--\eqref{constrt} and the $\alpha$-TQG \eqref{ceAlpha}--\eqref{constrtAlpha}, we aim to show that any family $(b^\alpha,\omega^\alpha, T_{\max}^\alpha)_{\alpha>0}$ of unique maximal solutions to the $\alpha$-TQG model \eqref{ceAlpha}--\eqref{constrtAlpha} converges strongly to the unique maximal solution $(b,\omega, T_{\max})$ of the TQG model \eqref{ce}--\eqref{constrt} as  $\alpha\rightarrow0$ provided that both system share the same initial conditions and are defined on a common time interval $[0,T]$. Our main result is the following.
\begin{proposition} 
\label{prop:alphaConv}
Under Assumption \ref{dataAllThroughout},
let $(b^\alpha, \omega^\alpha, T_{\max}^\alpha)_{\alpha>0}$ be a family of unique maximal solutions of \eqref{ceAlpha}--\eqref{constrtAlpha} and let  $(b, \omega, T_{\max})$ be the unique maximal solution of  \eqref{ce}--\eqref{constrt}. Then
\begin{equation}
\begin{aligned}
\label{rateOfConv}
\sup_{t\in [0,T]}\big(\Vert  ({b}^\alpha
  -b)(t) \Vert_{1,2}^2
&+
\Vert  ({\omega}^\alpha -\omega)(t) \Vert_{2}^2 \big)
\lesssim
\alpha^2
T\exp(cT)\big[1+\exp(cT) \big]
\end{aligned}
\end{equation}
for any  $T<\min\{T_{\max},T_{\max}^\alpha\}$ and 
\begin{align}
\label{rateOfConvZero}
\sup_{t\in[0,T]}\Vert
({b}^\alpha - b)(t) \Vert_{1,2} \rightarrow 0
, \quad\qquad
\sup_{t\in[0,T]}\Vert 
({\omega}^\alpha - \omega)(t) \Vert_{2}\rightarrow 0
\end{align}
as $\alpha \rightarrow0$.
\end{proposition}
We will now devote the entirety of this section to the proof of Proposition \ref{prop:alphaConv}.
In order to achieve this goal, we first introduce a unique maximal solution $(\overline{b}^\alpha,\overline{\omega}^\alpha, \overline{T}^\alpha_{\max})$ with initial conditions $b_0 \in W^{3,2}(\mathbb{T}^2)$ and $\omega_0\in W^{2,2}(\mathbb{T}^2)$  of the following ``intermediate $\alpha$-TQG" equation given by
\begin{align}
\frac{\partial}{\partial t} \overline{b}^\alpha + \overline{\bu}^\alpha \cdot \nabla \overline{b}^\alpha =0,
\label{ceAlphaInter}
\\
\frac{\partial}{\partial t}\overline{\omega}^\alpha + \overline{\bu}^\alpha\cdot \nabla( \overline{\omega}^\alpha-\overline{b}^\alpha) = -\bu_h \cdot \nabla \overline{b}^\alpha, \label{meAlphaInter}
\end{align}
where like  \eqref{utilde},
\begin{align} 
\label{constrtAlphaInter}
\overline{\bu}^\alpha =\nabla^\perp (1-\alpha \Delta )^{-1}(\Delta-1)^{-1}(\omega -f), 
\quad
\bu_h =\frac{1}{2} \nabla^\perp h,
\end{align}
and where $\omega$(and not $\omega^\alpha$) satisfies the original potential vorticity equation \eqref{me} strongly in the PDE sense. 
The existence of a unique maximal solution to the above intermediate system, denoted by $\bf{i} \alpha$-TQG for short, follow from the proof of existence for the original TQG and it particular, for a uniform-in-$\alpha$ set of data $(\bu_h,f,b_0,\omega_0)$, we have the following uniform-in-$\alpha$ estimates
\begin{align}
\label{unifInAlpha}
\sup_{\alpha>0}\sup_{t\in[0,\overline{T}^\alpha_{\max} )}\Vert \overline{b}^\alpha \Vert_{3,2}\lesssim 1,
\quad
\sup_{\alpha>0}\sup_{t\in[0,\overline{T}^\alpha_{\max})}\Vert \overline{\omega}^\alpha \Vert_{2,2}\lesssim 1
\end{align}
for the maximal solution of \eqref{ceAlphaInter}--\eqref{constrtAlphaInter}. Indeed, exactly as in Theorem \ref{thm:mainAlpha}, we also have  the following result.
\begin{lemma}
\label{lem:iAlpha}
Fix $\alpha>0$. There exist a unique maximal solution $(\overline{b}^\alpha,\overline{\omega}^\alpha, \overline{T}^\alpha_{\max})$ of \eqref{ceAlphaInter}--\eqref{constrtAlphaInter} under Assumption \ref{dataAllThroughout}. In particular, uniformly of $\alpha>0$, the inequality
\begin{align}
\label{unifInAlpha}
\Vert (\overline{b}^\alpha, \overline{\omega}^\alpha)(t) \Vert_{\mathcal{M}} \lesssim 1
\end{align}
holds for $t<\overline{T}^\alpha_{\max}$.
\end{lemma}
As a next step, we show that any family  $(\overline{b}^\alpha,\overline{\omega}^\alpha, \overline{T}^\alpha_{\max})_{\alpha>0}$ of maximal solutions to the intermediate $\mathbf{i}\alpha$-TQG model \eqref{ceAlphaInter}--\eqref{constrtAlphaInter} (rather than of the $\alpha$-TQG model \eqref{ceAlpha}--\eqref{constrtAlpha}) converges strongly to the unique maximal solution of the TQG model \eqref{ce}--\eqref{constrt} as  $\alpha\rightarrow0$ on the time interval $[0,T]$ where $T<\min\{\overline{T}^\alpha_{\max}, T_{\max} \}$. In order to achieve this goal, we replicate the uniqueness argument in Section \ref{subsec:unique} by setting 
  $b^{12}:= \overline{b}^\alpha
  -b$, $\omega^{12}:= \overline{\omega}^\alpha -\omega$ and $\bu^{12}:= \overline{\bu}^\alpha-\bu$ so that  $(b^{12}, \bu^{12}, \omega^{12})$ satisfies
\begin{align}
\frac{\partial}{\partial t} b^{12} + \bu^{12} \cdot \nabla \overline{b}^\alpha + \bu \cdot\nabla b^{12} =0,
\\
\frac{\partial}{\partial t}\omega^{12} + \bu^{12}\cdot \nabla( \overline{\omega}^\alpha - \overline{b}^\alpha) +\bu \cdot \nabla (\omega^{12} - b^{12}) = -\bu_h \cdot \nabla b^{12}.
\end{align}
Similar to the uniqueness argument, if we apply $\partial^\beta $ to the equation for $b^{12}$ above where now $\vert \beta \vert\leq 1$, we obtain
\begin{align}
\frac{\partial}{\partial t} \partial^\beta b^{12}  + \bu \cdot\nabla \partial^\beta b^{12}  = S_1 +  S_2 -\bu^{12} \cdot \nabla \partial^\beta \overline{b}^\alpha,
\label{ceDiff5}
\end{align} 
where 
\begin{align*}
S_1&:= \bu^{12} \cdot \partial^\beta  \nabla \overline{b}^\alpha
-
\partial^\beta(\bu^{12} \cdot \nabla  \overline{b}^\alpha )
=
-\partial\bu^{12}\cdot \nabla \overline{b}^\alpha
,
\\
S_2&:=\bu \cdot \partial^\beta  \nabla b^{12}
-
\partial^\beta(\bu \cdot \nabla  b^{12} )=-\partial\bu\cdot \nabla b^{12}
\end{align*}
are such that
\begin{align}
\Vert S_1 \Vert_2 \lesssim 
 \Vert  \bu^{12} \Vert_{1,2}\Vert \overline{b}^\alpha \Vert_{3,2},
\qquad \qquad
\Vert S_2 \Vert_2 \lesssim
\Vert  \bu \Vert_{3,2}\Vert b^{12} \Vert_{1,2}.
\end{align}
Also,
\begin{equation}
\begin{aligned}
\Big\vert\big\langle \bu^{12} \cdot \nabla \partial^\beta \overline{b}^\alpha  \, ,\, \partial^\beta b^{12} \big\rangle \Big\vert
 \lesssim
\Vert \bu^{12} \Vert_{1,2}
\Vert \overline{b}^\alpha  \Vert_{3,2}
 \Vert b^{12} \Vert_{1,2}.
\end{aligned}
\end{equation}
Collecting the information above yields the estimate
\begin{equation}
\begin{aligned}
\label{b32estUniq}
\frac{\dd}{\dd t}
\Vert  b^{12} \Vert_{1,2}^2 
&\lesssim
\big(
\Vert  \bu^{12} \Vert_{1,2}\Vert \overline{b}^\alpha \Vert_{3,2},
+
 \Vert \bu \Vert_{3,2}\Vert b^{12} \Vert_{1,2}
 \big)\Vert b^{12} \Vert_{1,2}.
\end{aligned}
\end{equation} 
For the equation for $\omega^{12}$, we aim to derive a square-integrable estimate in space. For this, we first note that
\begin{align}
&\Big\vert\big\langle \bu \cdot \nabla  \omega^{12} \, ,\,  \omega^{12} \big\rangle \Big\vert
=0.
\end{align}
Next, we have that
\begin{align}
\Big\vert\big\langle (\bu -\bu_h)\cdot \nabla   b^{12} \, ,\,  \omega^{12} \big\rangle \Big\vert
\lesssim
\Vert \bu -\bu_h \Vert_\infty
\Vert b^{12} \Vert_{1,2}
 \Vert \omega^{12} \Vert_{2}.
\end{align}
Finally, we also have that
\begin{equation}
\begin{aligned}
\Big\vert\big\langle \bu^{12}\cdot \nabla (\overline{\omega}^\alpha - \overline{b}^\alpha) \, ,\,  \omega^{12} \big\rangle \Big\vert
&\lesssim
\Vert \bu^{12} \Vert_{4}
\Vert \nabla(\overline{\omega}^\alpha - \overline{b}^\alpha) \Vert_{4}
 \Vert \omega^{12} \Vert_{2}
 \\&
\lesssim
\Vert \bu^{12} \Vert_{1,2}
\Vert\overline{\omega}^\alpha - \overline{b}^\alpha \Vert_{2,2}
 \Vert \omega^{12} \Vert_{2}.
\end{aligned}
\end{equation}
Therefore,
\begin{equation}
\begin{aligned}
\label{w32estUniq}
\frac{\dd}{\dd t}
\Vert  \omega^{12} \Vert_{2}^2 
&\lesssim
\big(\Vert \bu -\bu_h \Vert_\infty
\Vert b^{12} \Vert_{1,2}
 +
 \Vert \bu^{12} \Vert_{1,2}
\Vert\overline{\omega}^\alpha - \overline{b}^\alpha \Vert_{2,2}
\big) \Vert \omega^{12} \Vert_{2}.
\end{aligned}
\end{equation} 
Summing up with the estimate for $b^{12}$ then yield
\begin{equation}
\begin{aligned}
\frac{\dd}{\dd t}\big(
\Vert  b^{12} \Vert_{1,2}^2
+
\Vert  \omega^{12} \Vert_{2}^2 
\big) 
&\lesssim
\big(
\Vert  \bu^{12} \Vert_{1,2}\Vert \overline{b}^\alpha \Vert_{3,2},
+
 \Vert  \bu \Vert_{3,2}\Vert b^{12} \Vert_{1,2}
 \big)\Vert b^{12} \Vert_{1,2}
 \\&+
\big(\Vert \bu -\bu_h \Vert_\infty
\Vert b^{12} \Vert_{1,2}
 +
 \Vert \bu^{12} \Vert_{1,2}
\Vert\overline{\omega}^\alpha - \overline{b}^\alpha \Vert_{2,2}
\big) \Vert \omega^{12} \Vert_{2}.
\end{aligned}
\end{equation} 
If we use \eqref{lem:MasterEst} and \eqref{unifInAlpha}, we obtain
\begin{equation}
\begin{aligned}
\label{concludeDiff}
\frac{\dd}{\dd t}\big(
\Vert  b^{12} \Vert_{1,2}^2
+
\Vert  \omega^{12} \Vert_{2}^2 
\big) 
&\lesssim
\big(
\Vert  \bu^{12} \Vert_{1,2}
+
\Vert b^{12} \Vert_{1,2}
 \big)\Vert b^{12} \Vert_{1,2}
 \\&+
\big(
\Vert b^{12} \Vert_{1,2}
 +
 \Vert \bu^{12} \Vert_{1,2}
\big) \Vert \omega^{12} \Vert_{2}
\\&
\lesssim
\Vert  \bu^{12} \Vert_{1,2}^2
+
\Vert b^{12} \Vert_{1,2}^2
 +
 \Vert \omega^{12} \Vert_{2}^2.
\end{aligned}
\end{equation}
Now recall that the velocity fields for the $\alpha$-TQG, the $\bf{i}\alpha$-TQG, and the TQG are given by 
\begin{align} 
{\bu}^\alpha =\nabla^\perp (1-\alpha \Delta )^{-1}(\Delta-1)^{-1}(\omega^\alpha -f), 
\\
\overline{\bu}^\alpha =\nabla^\perp (1-\alpha \Delta )^{-1}(\Delta-1)^{-1}(\omega -f), 
\\
\bu =\nabla^\perp (\Delta-1)^{-1}(\omega -f)
\end{align}
respectively. As a result, in particular, the difference
\begin{align*}
\overline{\bu}^\alpha -\bu =\alpha \Delta\nabla^\perp (1-\alpha \Delta )^{-1}(\Delta-1)^{-1}(\omega -f)
\end{align*}
enjoys two extra order of regularity. Furthermore, since each of these individual velocity fields satisfies the bound \eqref{lem:MasterEst}, it immediately follows that
\begin{equation}
\begin{aligned}
\Vert \overline{\bu}^\alpha -\bu \Vert_{1,2}^2
\lesssim
\alpha^2
\Vert \bu \Vert_{3,2}^2
\lesssim
 \alpha^2
\Vert \omega - f \Vert_{2,2}^2
\lesssim
 \alpha^2\big(
\Vert \omega  \Vert_{2,2}^2 + \Vert f \Vert_{2,2}^2 \big) \lesssim \alpha^2.
\end{aligned}
\end{equation}
This means that for the comparison of the TQG and $\bf{i}\alpha$-TQG, we can  conclude from \eqref{concludeDiff} that
\begin{equation}
\begin{aligned}
\frac{\dd}{\dd t}\big(
\Vert  \overline{b}^\alpha -b \Vert_{1,2}^2
&+
\Vert  \overline{\omega}^\alpha -\omega  \Vert_{2}^2 
\big) 
\lesssim
\alpha^2
+
\Vert \overline{b}^\alpha -b \Vert_{1,2}^2
 +
 \Vert \overline{\omega}^\alpha -\omega  \Vert_{2}^2
\end{aligned}
\end{equation}
and by Gr\"onwall's lemma, (note that $b^{12}_0=0$ and $\omega^{12}_0=0$)
\begin{equation}
\begin{aligned}
\label{decayRateTQGandIATQG}
\Vert  (\overline{b}^\alpha -b)(t) \Vert_{1,2}^2
&+
\Vert  (\overline{\omega}^\alpha -\omega)(t) \Vert_{2}^2 
\lesssim
\alpha^2
T\exp(cT),
\qquad
t\in[0,T].
\end{aligned}
\end{equation}
The above gives the decay rate of the difference of the solution to the $\mathbf{i}\alpha$-TQG and TQG equations. Our next goal is to obtain a decay rate for the difference between the $\mathbf{i}\alpha$-TQG and the  $\alpha$-TQG equations. Recall that they share the same initial data.
For this, we use the estimate
\begin{align*}
\Vert \overline{\bu}^\alpha - {\bu}^\alpha \Vert_{1,2}^2
\lesssim
\Vert \omega - \omega^\alpha \Vert_{2}^2
\lesssim
\Vert \overline{\omega}^\alpha - \omega  \Vert_{2}^2
+
\Vert  \overline{\omega}^\alpha  - \omega^\alpha \Vert_{2}^2
\end{align*}
so that by setting 
  $b^{12}:= \overline{b}^\alpha
  -b^\alpha$, $\omega^{12}:= \overline{\omega}^\alpha -\omega^\alpha$ and $\bu^{12}:= \overline{\bu}^\alpha-{\bu}^\alpha$,
we obtain from \eqref{concludeDiff}, 
\begin{equation}
\begin{aligned}
\frac{\dd}{\dd t}\big(
\Vert  b^{12} \Vert_{1,2}^2
+
\Vert  \omega^{12} \Vert_{2}^2 
\big) 
&\lesssim
\Vert  \bu^{12} \Vert_{1,2}^2
+
\Vert b^{12} \Vert_{1,2}^2
 +
 \Vert \omega^{12} \Vert_{2}^2
 \\&
 \lesssim
\Vert  \overline{\omega}^\alpha -\omega \Vert_{2}^2
+
\Vert b^{12} \Vert_{1,2}^2
 +
 \Vert \omega^{12} \Vert_{2}^2.
\end{aligned}
\end{equation}
Since both systems share the same initial conditions, by Gr\"onwall's lemma  and \eqref{decayRateTQGandIATQG}, it follows that
\begin{equation}
\begin{aligned}
\label{decayRateATQGandIATQG}
\Vert  (\overline{b}^\alpha
  -b^\alpha)(t) \Vert_{1,2}^2
+
\Vert  (\overline{\omega}^\alpha -\omega^\alpha)(t) \Vert_{2}^2 
&\leq \exp(cT)\int_0^T c\,\Vert (\overline{\omega}^\alpha -\omega)(t) \Vert_2^2 \dd t
\\&\lesssim
\alpha^2
T\big[\exp(cT) \big]^2.
\end{aligned}
\end{equation}
By the triangle inequality, it follows from \eqref{decayRateTQGandIATQG} and \eqref{decayRateATQGandIATQG} that 
\begin{equation}
\begin{aligned}
\label{decayRateATQGandTQG}
\Vert  ({b}^\alpha
  -b)(t) \Vert_{1,2}^2
&+
\Vert  ({\omega}^\alpha -\omega)(t) \Vert_{2}^2 
\lesssim
\alpha^2
T\exp(cT)\big[1+\exp(cT) \big]
\end{aligned}
\end{equation}
so that
\begin{align*}
\sup_{t\in[0,T]}\Vert
({b}^\alpha - b)(t) \Vert_{1,2} \rightarrow 0
, \quad\qquad
\sup_{t\in[0,T]}\Vert 
({\omega}^\alpha - \omega)(t) \Vert_{2}\rightarrow 0
\end{align*}
as $\alpha \rightarrow0$. This ends the proof.

\begin{remark}
We conjecture that Proposition \ref{prop:alphaConv} may be extended to the stronger space $W^{2,2}(\mathbb{T}^2) \times W^{1,2}(\mathbb{T}^2)$ in which we showed the stability and uniqueness of the maximal solution. The cost, however, may be the loss of the corresponding decay rate \eqref{rateOfConv}. The ideas in the proof of Proposition $6$ of Terence Tao's note : \href{https://terrytao.wordpress.com/2018/10/09/254a-notes-3-local-well-posedness-for-the-euler-equations/#more-10770}{local-well-posedness-for-the-euler-equations} may suffice.
\end{remark}

\section{Conditions for blowup of the $\alpha$-TQG model}
\label{sec:blowup}
In the following, we give  Beale--Kato--Majda \cite{beale1984remarks} type conditions under which we expect the strong solution of the $\alpha$-TQG to blowup. 
\begin{theorem}
\label{thm:BKM}
Fix $\alpha>0$. Under Assumption \ref{dataAllThroughout}, let $(b^\alpha, \omega^\alpha,  T_{\max}^\alpha) $ be
a  \textit{maximal solution} of   \eqref{ceAlpha}--\eqref{constrtAlpha}. If $T_{\max}^\alpha<\infty$, then
\begin{align*}
\int_0^{T_{\max}^\alpha} \Vert \nabla b^\alpha\Vert_{\infty}\dt = \infty.
\end{align*}
\end{theorem}
\begin{proof}
First of all, fix $\alpha>0$ and let $T_{\max}^\alpha >0$ be  the maximal time so that
\begin{align}
\label{blowupTimeOfSol}
\limsup_{T_n\rightarrow T_{\max}^\alpha} \Vert (b^\alpha,\omega^\alpha)(T_n) \Vert_{\mathcal{M}}^2 =\infty.
\end{align} 
We now suppose that
\begin{align}
\label{mnot}
\int_0^{T_{\max}^\alpha} \Vert \nabla b^\alpha\Vert_{\infty} \dt = K<\infty
\end{align}
and show that
\begin{align}
\label{trueEstToBeContra}
\Vert (b^\alpha,\omega^\alpha)(t) \Vert_{\mathcal{M}}^2 \lesssim 1, \qquad t<T_{\max}^\alpha
\end{align}
holds yielding a contradiction to \eqref{blowupTimeOfSol}.
\\
Before we show the estimate \eqref{trueEstToBeContra}, we first require preliminary estimates for $\Vert\nabla \bu^\alpha\Vert_\infty$ and $\Vert \omega^\alpha\Vert_2$. To obtain these estimates, we note that the Fourier multiplier $m^\alpha(\xi)$  defined below satisfies the bound
\begin{align}
\label{multiplier}
m^\alpha(\xi):=\frac{(1+\vert\xi\vert^2)\vert\xi\vert^2}{(1+\alpha\vert\xi\vert^2)^2} \lesssim_\alpha 1
\end{align}
for all fixed $\alpha>0$ and in particular, the bound may be taken uniformly of all $\alpha\geq1$. Due to \eqref{multiplier} and the continuous embedding $W^{3,2}(\mathbb{T}^2)\hookrightarrow W^{1,\infty}(\mathbb{T}^2)$, we can conclude that
\begin{align}
\label{uInfty}
\Vert \nabla \bu^\alpha \Vert_\infty \lesssim \Vert \bu^\alpha\Vert_{3,2}\lesssim_\alpha \Vert \omega^\alpha - f \Vert_2 \lesssim \Vert \omega^\alpha - f \Vert_{2,2}.
\end{align}
On the other hand, if we test \eqref{meAlpha} with $\omega^\alpha$ and
use \eqref{lem:MasterEst} for $w=(1-\alpha\Delta)^{-1}(\omega^\alpha-f)$ and $k=0$,
we obtain
\begin{equation}
\begin{aligned}
\label{o2est}
\frac{\dd}{\dd t}
\Vert  \omega^\alpha \Vert_{2}^2 
&\lesssim
\Big(
\Vert \bu^\alpha \Vert_{2}\Vert \omega^\alpha \Vert_{2}
+
\Vert \bu_h \Vert_{2} \Vert \omega^\alpha \Vert_{2}
\Big) 
\Vert \nabla  b^\alpha \Vert_{\infty}
\lesssim
\Big(
1
+
 \Vert \omega^\alpha \Vert_{2}^2
\Big) 
\Vert \nabla  b^\alpha \Vert_{\infty}
\end{aligned}
\end{equation}
for a constant depending only on $\Vert\bu_h\Vert_2$ and $\Vert f \Vert_2$. It therefore follow from \eqref{o2est} and \eqref{mnot} that
\begin{align}
\label{o2estTime}
\Vert \omega^\alpha(t) \Vert_2^2\lesssim \big(1+ \Vert\omega_0\Vert_2^2 \big)\exp(cK), \qquad t<T_{\max}^\alpha
\end{align}
which when combined with the first estimate in \eqref{uInfty} yields
\begin{align}
\label{uInfty1}
\Vert \nabla \bu^\alpha(t) \Vert_\infty \lesssim_\alpha 1+\big(1+ \Vert\omega_0\Vert_2^2 \big)\exp(cK)
\lesssim_\alpha \big(1+ \Vert\omega_0\Vert_2^2 \big)\exp(cK), \qquad t<T_{\max}^\alpha.
\end{align}
Recall that $f\in W^{2,2}(\mathbb{T}^2)$ by assumption.  
\\
With these preliminary estimates in hand, we can proceed to derive estimates for $(b^\alpha,\omega^\alpha)$ in the space $\mathcal{M}$ of existence of a solution.
Since the space of smooth functions is dense in $\mathcal{M}$, it suffices to show our result for a smooth solution pair $(b^\alpha,\omega^\alpha)$.
\\
To achieve our goal, we apply $\partial^\beta $ to \eqref{ceAlpha} for $\vert \beta\vert \leq3$ to obtain
\begin{align}
\frac{\partial}{\partial t}\partial^\beta b^\alpha + \bu^\alpha \cdot \nabla\partial^\beta  b^\alpha  = R_1
\label{ce1}
\end{align} 
where 
\begin{align*}
R_1:= \bu^\alpha \cdot \partial^\beta  \nabla b^\alpha
-
\partial^\beta(\bu^\alpha \cdot \nabla  b^\alpha ).
\end{align*}
Now since $\mathrm{div}\bu^\alpha =0$, if we multiply \eqref{ce1} by $\partial^\beta b^\alpha$ and sum over the multiindex $\beta$ so that $\vert \beta\vert \leq3$,  we obtain 
\begin{equation}
\begin{aligned}
\label{b32est}
\frac{\dd}{\dd t}
\Vert  b^\alpha \Vert_{3,2}^2 
&\lesssim
\Big(
\Vert \nabla \bu^\alpha \Vert_{\infty}\Vert b^\alpha \Vert_{3,2}
+
\Vert \nabla b^\alpha \Vert_{\infty} \Vert \bu^\alpha \Vert_{3,2}
\Big) 
\Vert  b^\alpha \Vert_{3,2}
\\&
\lesssim
(\Vert \nabla \bu^\alpha \Vert_{\infty}
+
\Vert \nabla b^\alpha \Vert_{\infty})(1+\Vert (b^\alpha, \omega^\alpha)  \Vert_{\mathcal{M}}^2)
\end{aligned}
\end{equation} 
where we have used \eqref{lem:MasterEst} for $w=(1-\alpha\Delta)^{-1}(\omega^\alpha-f)$ and $k=2$.
\\
Next, we find a bound for $\Vert \omega^\alpha \Vert^2_{2,2}$. For this, we apply $\partial^\beta $ to \eqref{meAlpha} for $\vert \beta \vert \leq2$ and we obtain
\begin{align}
\frac{\partial}{\partial t}\partial^\beta\omega^\alpha + \bu^\alpha\cdot \nabla\partial^\beta( \omega^\alpha-b^\alpha) +\bu_h \cdot \nabla\partial^\beta b^\alpha =  R_2 + R_3+R_4
\label{me1}
\end{align} 
where 
\begin{align*}
R_2&:= \bu^\alpha\cdot \partial^\beta\nabla \omega^\alpha
-
\partial^\beta(\bu^\alpha \cdot \nabla  \omega^\alpha ),
\\
R_3&:= -\bu^\alpha\cdot \partial^\beta\nabla b^\alpha
+
\partial^\beta(\bu^\alpha \cdot \nabla  b^\alpha ),
\\
R_4&:= \bu_h\cdot \partial^\beta\nabla b^\alpha
-
\partial^\beta(\bu_h \cdot \nabla  b^\alpha )
\end{align*}
are such that
\begin{align}
\big\vert \big\langle R_2 \, ,\,\partial^\beta\omega^\alpha \big\rangle
 \big\vert \lesssim
\Vert R_2 \Vert_2 \Vert \omega^\alpha \Vert_{2,2}&\lesssim 
 \Big(\Vert \nabla \bu^\alpha \Vert_{\infty}\Vert \omega^\alpha \Vert_{2,2}
+\Vert \omega^\alpha \Vert_{2,2} (1+\Vert \omega^\alpha \Vert_{2}) \Big)\Vert \omega^\alpha \Vert_{2,2} \label{estR2}
,
\\
\big\vert \big\langle R_3 \, ,\,\partial^\beta\omega^\alpha \big\rangle
 \big\vert \lesssim
\Vert R_3 \Vert_2 \Vert \omega^\alpha \Vert_{2,2}&\lesssim 
\Big(\Vert \nabla \bu^\alpha \Vert_{\infty} \Vert b^\alpha \Vert_{3,2}
+
\Vert \nabla b^\alpha \Vert_{\infty} (1+\Vert \omega^\alpha \Vert_{2,2})
\Big)\Vert \omega^\alpha \Vert_{2,2},\label{estR3}
\\
\big\vert \big\langle R_4 \, ,\,\partial^\beta\omega^\alpha \big\rangle
 \big\vert \lesssim
\Vert R_4 \Vert_2 \Vert \omega^\alpha \Vert_{2,2}&\lesssim 
\Big(\Vert b^\alpha \Vert_{3,2}
+
\Vert \nabla b^\alpha \Vert_{\infty}\Big)\Vert \omega^\alpha \Vert_{2,2} \label{estR4}
\end{align}
holds true. Recall that $\bu_h \in W^{3,2}(\mathbb{T}^2)$ by assumption and note that we have used the continuous embedding $W^{3,2}(\mathbb{T}^2)\hookrightarrow W^{2,4}(\mathbb{T}^2)$ and the second inequality in \eqref{uInfty} in order to obtain the term $1+\Vert\omega^\alpha\Vert_2$ for the estimates \eqref{estR2}. 
Next, by using  $\mathrm{div}\bu^\alpha =0$, we have the following identity
\begin{align}
 \big\langle(\bu^\alpha\cdot \nabla\partial^\beta \omega^\alpha)\, ,\,\partial^\beta \omega^\alpha \big\rangle=\frac{1}{2}\int_{\mathbb{T}^2}\mathrm{div}(\bu^\alpha\vert\partial^\beta \omega^\alpha\vert^2)\dd x=0.
\end{align}
Additionally, the following estimates for $L^2$ inner products holds true
\begin{align}
\Big\vert \big\langle ( \bu^\alpha\cdot \nabla\partial^\beta b^\alpha ) \, ,\,\partial^\beta\omega^\alpha \big\rangle \Big\vert
&\lesssim
\Vert \bu^\alpha \Vert_\infty \Vert b^\alpha\Vert_{3,2}^2
+
\Vert \bu^\alpha \Vert_\infty 
\Vert \omega^\alpha\Vert_{2,2}^2,
\\
\Big\vert \big\langle ( \bu_h\cdot \nabla\partial^\beta b^\alpha ) \, ,\,\partial^\beta \omega^\alpha \big\rangle \Big\vert
&\lesssim
\Vert b^\alpha\Vert_{3,2}^2 
+ \Vert \omega^\alpha\Vert_{2,2}^2.
\end{align}
since $\bu_h \in W^{3,2}(\mathbb{T}^2)$. 
If we now collect the estimates above (keeping in mind that $f\in W^{2,2}(\mathbb{T}^2)$ and $\bu_h \in W^{3,2}(\mathbb{T}^2)$), we obtain by multiplying \eqref{me1} by $\partial^\beta \omega^\alpha$ with $\vert \beta\vert \leq2$ and then summing over $\vert \beta\vert \leq2$, the following
\begin{equation}
\begin{aligned}
\label{omeg22}
 \frac{\dd}{\dd t} \Vert \omega^\alpha \Vert_{2,2}^2
&\lesssim
\big( 1+\Vert \bu^\alpha \Vert_{1,\infty}
+
\Vert \nabla b^\alpha \Vert_{\infty} + \Vert\omega^\alpha \Vert_2 \big)\big(1+\Vert (b^\alpha, \omega^\alpha)  \Vert_{\mathcal{M}}^2\big).
\end{aligned}
\end{equation}
Summing up \eqref{b32est} and \eqref{omeg22} and using \eqref{o2estTime}--\eqref{uInfty} yields
\begin{equation}
\begin{aligned}
\label{rewriteB0}
\frac{\dd}{\dd t}
(1+\Vert (b^\alpha, \omega^\alpha)  \Vert_{\mathcal{M}}^2)
&
\lesssim_\alpha
\Big(\big(1+ \Vert\omega_0\Vert_2^2 \big)\exp(cK)
+
\Vert \nabla b^\alpha \Vert_{\infty}\Big)(1+\Vert (b^\alpha, \omega^\alpha)  \Vert_{\mathcal{M}}^2)
\end{aligned}
\end{equation} 
so that by Gr\"onwall's lemma and  \eqref{mnot}, we obtain
\begin{equation}
\begin{aligned}
\Vert (b^\alpha,\omega^\alpha)(t) \Vert_{\mathcal{M}}^2 &\leq \big[ 1 +\Vert (b_0, \omega_0) \Vert_{\mathcal{M}}^2\big]
\exp\Big(c(\alpha)T_{\max}\big(1+ \Vert\omega_0\Vert_2^2 \big)\exp(cK)
+
c(\alpha)K\Big)
\\&
\lesssim_{T_{\max}^\alpha,K,f,\bu_h,\omega_0,b_0,\alpha}1
\end{aligned}
\end{equation}
for all $t<{T_{\max}^\alpha}$ contradicting \eqref{blowupTimeOfSol}.
\end{proof}
\begin{remark}[Global existence for constant-in-space buoyancy]
We note that for the $\alpha$-TQG model,
when the buoyancy is constant in space so that $\nabla b^\alpha=0$, then the equation for the buoyancy decouples from that of the potential vorticity. The potential vorticity equation then reduces to the $2$-dimensional Euler equation albeit a different Biot–Savart law. Nevertheless, by inspecting the proof of Theorem \ref{thm:BKM} above, one observes that the same global-in-time result applies. In particular, the residual estimate \eqref{estR2} still holds true and we obtain from this, a refine version of \eqref{o2estTime}-\eqref{uInfty1} where $\exp(c K)=1$. Recall from \eqref{mnot} that $K=0$ when $\nabla b^\alpha=0$.
\end{remark}

\begin{remark}[Global existence for super diffusive buoyancy]
Fix $\alpha>0$. Note that by adding any super diffusive term $\Lambda^\beta b^\alpha$, $\beta\geq1$ to the right-hand side of \eqref{ceAlpha} so that
\begin{align*}
    \int_0^t \Vert \nabla b^\alpha(s) \Vert_\infty \dd s \lesssim 
    \int_0^t\Vert \Lambda^{\beta/2} b^\alpha(s) \Vert_2^2 \dd s
\end{align*}
holds for all $t<T_{\max}^\alpha$, then we obtain a global solution since in this case, we obtain the energy estimate
\begin{align*}
    \Vert b^\alpha (t) \Vert_2^2
    +
    \int_0^t\Vert \Lambda^{\beta/2} b^\alpha(s) \Vert_2^2 \dd s
    \leq \Vert b_0\Vert_2 \leq \Vert b_0\Vert_{3,2}
\end{align*}
for all $t<T_{\max}^\alpha$.
\end{remark}
Next, we also give an alternative to the blowup condition in Theorem \ref{thm:BKM} above.
\begin{theorem}
\label{thm:BKM1}
Fix $\alpha>0$. Under Assumption \ref{dataAllThroughout}, let $(b^\alpha, \omega^\alpha,  T_{\max}^\alpha) $ be
a  \textit{maximal solution} of   \eqref{ceAlpha}--\eqref{constrtAlpha}. If $T_{\max}^\alpha<\infty$, then
\begin{align*}
\int_0^{T_{\max}^\alpha} \Vert \nabla \bu^\alpha\Vert_{\infty}\dt = \infty.
\end{align*}
\end{theorem}
\begin{proof}
In analogy with the proof of Theorem \ref{thm:BKM}, we fix $\alpha>0$, let $T_{\max}^\alpha >0$ be  the maximal time so that
\begin{align}
\label{blowupTimeOfSol1}
\limsup_{T_n\rightarrow T_{\max}^\alpha} \Vert (b^\alpha,\omega^\alpha)(T_n) \Vert_{\mathcal{M}}^2 =\infty.
\end{align} 
We now suppose that
\begin{align}
\label{mnot1}
\int_0^{T_{\max}^\alpha} \Vert \nabla \bu^\alpha\Vert_{\infty} \dt = K<\infty
\end{align}
and show that
\begin{align}
\label{trueEstToBeContra1}
\Vert (b^\alpha,\omega^\alpha)(t) \Vert_{\mathcal{M}}^2 \lesssim 1, \qquad t<T_{\max}^\alpha
\end{align}
holds yielding a contradiction to \eqref{blowupTimeOfSol1}.
\\
To show this, we first fix $\alpha>0$. If we differentiate \eqref{ceAlpha} in space, multiply the resulting equation by $p\vert \nabla b^\alpha \vert^{p-2}\nabla b^\alpha$ where $p>1$ is fixed and finite in this instant, and then integrate over $\mathbb{T}^2$, we obtain
\begin{equation}
\begin{aligned}
\frac{\dd}{\dt}\Vert \nabla b^\alpha\Vert_p^p \leq p \Vert \nabla \bu^\alpha\Vert_\infty \Vert \nabla b^\alpha\Vert_p^p.
\end{aligned}
\end{equation}
Rather than use Gr\"onwall's lemma at this point, we use the analogous separation of variable technique to obtain
\begin{align*}
\frac{1}{p}\ln \Vert \nabla b^\alpha(t)\Vert_p^p 
\leq 
\frac{1}{p}\ln \Vert \nabla b_0\Vert_p^p
+
\int_0^t \Vert \nabla \bu^\alpha(s)\Vert_\infty \dd s
\end{align*}
so that
\begin{align*}
\Vert \nabla b^\alpha(t)\Vert_p 
\leq 
 \Vert \nabla b_0\Vert_p
 \exp\bigg(\int_0^t \Vert \nabla \bu^\alpha(s)\Vert_\infty \dd s \bigg)
\end{align*}
holds for all $t< T_{\max}^\alpha$ uniformly in $p>1$. Since $b_0 \in W^{3,2}(\mathbb{T}^2)$, by Sobolev embedding and \eqref{mnot1}, we obtain
\begin{align}
\label{balphaUni}
\Vert \nabla b^\alpha(t)\Vert_\infty 
\lesssim 
 \Vert  b_0\Vert_{3,2}
 \exp(K )
\end{align}
for all $t< T_{\max}^\alpha$ so that
\begin{align*}
\int_0^{T_{\max}^\alpha}\Vert \nabla b^\alpha(t)\Vert_\infty \dt
\leq c\,T_{\max}^\alpha 
 \Vert  b_0\Vert_{3,2}
 \exp(K ).
\end{align*}
We can therefore deduce from the estimate \eqref{o2est} that
\begin{align}
\label{o2estTime1}
\Vert \omega^\alpha(t) \Vert_2^2\lesssim \big(1+ \Vert\omega_0\Vert_2^2 \big)\exp \big( cT_{\max}^\alpha 
 \Vert  b_0\Vert_{3,2}
 \exp(K) \big), \qquad t<T_{\max}^\alpha.
\end{align}
Summing up \eqref{b32est} and \eqref{omeg22} and using \eqref{balphaUni}--\eqref{o2estTime1} yields
\begin{equation}
\begin{aligned}
\label{rewriteB0}
\frac{\dd}{\dd t}
(1+\Vert (b^\alpha, \omega^\alpha)  \Vert_{\mathcal{M}}^2)
&
\lesssim
\big(1+\Vert \bu^\alpha \Vert_{1,\infty}+\mathfrak{A}\big)
(1+\Vert (b^\alpha, \omega^\alpha)  \Vert_{\mathcal{M}}^2)
\end{aligned}
\end{equation} 
where
\begin{align*}
\mathfrak{A}
:=
 \Vert  b_0\Vert_{3,2}
 \exp(K )
+
\big(1+ \Vert\omega_0\Vert_2^2 \big)\exp \big(cT_{\max}^\alpha 
 \Vert  b_0\Vert_{3,2} \exp(K) 
 \big).
\end{align*}
We can therefore conclude from Gr\"onwall's lemma and  \eqref{mnot1} that
\begin{equation}
\begin{aligned}
\Vert (b^\alpha,\omega^\alpha)(t) \Vert_{\mathcal{M}}^2 &\leq \big[ 1 +\Vert (b_0, \omega_0) \Vert_{\mathcal{M}}^2\big]
\exp\Big(cT_{\max}^\alpha+cK + cT_{\max}^\alpha \mathfrak{A}\Big)
\\&
\lesssim_{T_{\max}^\alpha,K,f,\bu_h,\omega_0,b_0}1
\end{aligned}
\end{equation}
for all $t<{T_{\max}^\alpha}$ contradicting \eqref{blowupTimeOfSol1}.
\end{proof}
In order the give the final blowup criterion,
let us now recall the following endpoint Sobolev inequality by Br\'{e}zis and Gallouet \cite{brezis1980nonlinear}.
\begin{lemma}
\label{lem:brezis}
If $f\in W^{2,2}(\mathbb{T}^2)$   then
\begin{align*}
\Vert f \Vert_\infty \lesssim (1+ \Vert f \Vert_{1,2})\sqrt{\ln(\mathrm{e}+ \Vert f \Vert_{2,2})}.
\end{align*}
\end{lemma}
\begin{remark} 
Note that the original statement in \cite{brezis1980nonlinear} restricted the size of $\Vert f \Vert_{1,2}$ to being at most one. The current form  for any size of $\Vert f \Vert_{1,2}$ follows immediately as demonstrated in for example \cite{wang2011global, dinvay2020well}.
\end{remark}
With Lemma \ref{lem:brezis} in hand, we can now show the final blowup condition.
\begin{theorem}
\label{thm:BKM3}
Fix $\alpha>0$. Let  $(b^\alpha, \omega^\alpha,T_{\max}^\alpha) $ be 
a  \textit{maximal solution} of   \eqref{ceAlpha}--\eqref{constrtAlpha}. If $T_{\max}^\alpha<\infty$, then
\begin{align*}
\int_0^{T_{\max}^\alpha} \Vert \nabla b^\alpha\Vert_{1,2}\dt = \infty.
\end{align*}
\end{theorem}
\begin{proof}
Fix $\alpha>0$ and let $T_{\max}^\alpha >0$ be  the maximal time so that
\begin{align}
\label{blowupTimeOfSol3}
\limsup_{T_n\rightarrow T_{\max}^\alpha} \Vert (b^\alpha,\omega^\alpha)(T_n) \Vert_{\mathcal{M}}^2 =\infty.
\end{align} 
We now suppose that
\begin{align}
\label{mnot3}
\int_0^{T_{\max}^\alpha}  \Vert \nabla b^\alpha\Vert_{1,2} \dt = K<\infty
\end{align}
and show that
\begin{align}
\label{trueEstToBeContra3}
\Vert (b^\alpha,\omega^\alpha)(t) \Vert_{\mathcal{M}}^2 \lesssim 1, \qquad t<T_{\max}^\alpha
\end{align}
holds and thereby yields a contradiction to \eqref{blowupTimeOfSol3}.
\\
To show this, we first fix $\alpha>0$
and define
\begin{align*}
g(t) := \mathrm{e}+ \Vert (b^\alpha,\omega^\alpha)(t) \Vert_{\mathcal{M}}^2, \qquad t>0.
\end{align*}
With this definition in hand, it follows from
 estimates \eqref{b32est} and \eqref{omeg22} that
\begin{align}
\label{b32estomeg22}
\frac{\dd}{\dt} g \lesssim \big( 1+\Vert \bu^\alpha \Vert_{1,\infty}
 + \Vert\omega^\alpha \Vert_2+
\Vert \nabla b^\alpha \Vert_{\infty} \big) g.
\end{align}
By using Lemma \ref{lem:brezis} and the monotonic property of logarithms however, we can deduce that
\begin{align}
\label{2212}
\Vert \nabla b^\alpha \Vert_\infty
 \lesssim  \big(1+ \Vert \nabla b^\alpha \Vert_{1,2}\big) \ln\, g.
\end{align}
Next, we observe that the second estimate in \eqref{uInfty} yields
\begin{align}
\label{oneUomega1}
 1+\Vert \bu^\alpha \Vert_{1,\infty}+\Vert\omega^\alpha\Vert_2 \lesssim_\alpha  1+ \Vert \omega^\alpha  \Vert_2
\end{align}
for $f \in L^2(\mathbb{T}^2)$.
However, if we test \eqref{meAlpha} with $\omega^\alpha$, 
we also obtain
\begin{equation}
\begin{aligned}
\label{o2est1}
\frac{\dd}{\dd t}
\Vert  \omega^\alpha \Vert_{2}^2 
&\lesssim
\Big(
\Vert \bu^\alpha \Vert_{\infty}\Vert \omega^\alpha \Vert_{2}
+
\Vert \bu_h \Vert_{\infty} \Vert \omega^\alpha \Vert_{2}
\Big) 
\Vert \nabla  b^\alpha \Vert_{2}
\lesssim
\Big(
1
+
 \Vert \omega^\alpha \Vert_{2}^2
\Big) 
\Vert \nabla  b^\alpha \Vert_{1,2}
\end{aligned}
\end{equation}
for a constant depending only on $\Vert\bu_h\Vert_\infty$ and $\Vert f \Vert_2$. It therefore follow from \eqref{mnot3} that
\begin{align}
\label{o2estTime1}
\Vert \omega^\alpha(t) \Vert_2^2\lesssim \big(1+ \Vert\omega_0\Vert_2^2 \big)\exp(cK), \qquad t<T_{\max}^\alpha.
\end{align}
Using the fact that $\ln\, g\geq1$, we can conclude from \eqref{oneUomega1} and \eqref{o2estTime1} that
\begin{align}
\label{oneUomega}
 1+\Vert \bu^\alpha \Vert_{1,\infty}+\Vert\omega^\alpha\Vert_2 \lesssim_\alpha \big(1+ \Vert\omega_0\Vert_2^2 \big)\exp(cK) \ln\, g.
\end{align}
If we now combine this estimate with \eqref{2212}, we can conclude from \eqref{b32estomeg22} that
\begin{align*}
\frac{\dd}{\dt} g \lesssim \big(1+\Vert \omega_0\Vert_{2}+ \Vert \nabla b^\alpha \Vert_{1,2} \big)\exp(cK)\, g \ln\, g
\end{align*}
so that
\begin{align*}
g(t)\leq g(0)^{\exp\int_0^tc(1+\Vert \omega_0 \Vert_{2} + \Vert \nabla b^\alpha(s) \Vert_{1,2}) \exp(cK)\dd s}
\end{align*}
holds for $t<T_{\max}^\alpha$. In particular, given \eqref{mnot3}, we have shown that
\begin{equation}
\begin{aligned}
\Vert (b^\alpha,\omega^\alpha)(t) \Vert_{\mathcal{M}}^2 &\leq \big[ \mathrm{e} +\Vert (b^\alpha_0, \omega^\alpha_0) \Vert_{\mathcal{M}}^2\big]^{
\exp(c(T_{\max}^\alpha+K)  \exp(cK))}
\\&
\lesssim_{T_{\max}^\alpha,K,f,\bu_h,\omega_0,b_0}1
\end{aligned}
\end{equation}
for all $t<{T_{\max}^\alpha}$ contradicting \eqref{blowupTimeOfSol3}.
\end{proof}

\section{Numerical implementation}\label{sec: numerics}
{For numerical implementation, we choose to work in weaker spaces than what the wellposedness theorem dictates. Additionally we allow for boundary conditions in the numerical setup. We recognise these choices create gaps between the theory and the implementation.}

\subsection{Discretisation methods for $\alpha$--TQG and TQG   }\label{sec: numerics methods}

In this subsection, we describe the finite element (FEM) spatial discretisation, and finite difference Runge-Kutta time stepping discretisation methods that are utilised for the $\alpha$-TQG system. 
By setting $\alpha$ to zero we obtain our numerical setup for the TQG system. 

Consider a bounded domain $\domain$. Let $\p \domain$ denote the boundary. We impose Dirichlet boundary conditions
\begin{equation}\label{eq: tqg boundary conditions numerical}
    \psi^\alpha = 0, \    
    \Delta \psi^\alpha = 0, \qquad \text{on } \p \domain.
\end{equation}
For $\domain = \mathbb T^2$, our numerical setup does not change, in which case the boundary flux terms in the discretised equations are set to zero.

\subsubsection{The stream function equation}\label{sec: stream function discretisation}

Let $H^{1}\left(\domain\right)$ denote the Sobolev $W^{1,2}\left(\domain\right)$
space and let $\left\Vert .\right\Vert _{\partial\domain}$ denote
the $L^{2}\left(\partial\domain\right)$ norm. 
Define the space 
\begin{equation}\label{eq: w1 space}
    W^{1}\left(\domain\right):=\left\{ \nu\in H^{1}\left(\domain\right)\left|\left\Vert \nu\right\Vert _{\partial\domain}=0\right.\right\}.
\end{equation}
We express \eqref{constrtAlpha} as two inhomogeneous Helmholtz equations
\begin{align}
    \omega^\alpha - f &= (\Delta -1 ) \tilde\psi \label{eq:inhomogeneousHelmholtz}\\
    -\tilde\psi &= (\alpha \Delta - 1)\psi^\alpha\label{eq:inhomogeneousHelmholtzAlpha}. 
\end{align}
We take $\tilde\psi,\psi^\alpha \in W^1 (\domain)$.
Since the two equations are of the same form, let us first consider \eqref{eq:inhomogeneousHelmholtz}.
Using an arbitrary test function $\phi \in W^1(\domain)$, we obtain the following weak form of \eqref{eq:inhomogeneousHelmholtz},
\begin{align}
\langle\nabla \tilde\psi,\nabla\phi\rangle_\domain +\langle\tilde\psi,\phi\rangle_\domain  &=  -\langle \omega^\alpha - f, \phi \rangle_\domain.
\label{eq: inhomogeneousHelmholtz weak}
\end{align}
Define the functionals
\begin{align}
    L_\alpha(v, \phi) &:= \alpha\langle \nabla v, \nabla \phi \rangle_\domain +\langle v, \phi  \rangle_\domain \label{eq: psi bilinear form}\\
    F_{\cdot}(\phi) &:=  -\langle  \cdot, \phi \rangle_\domain
    \label{eq: psi linear form}
\end{align}
for $v, \phi \in W^1(\domain)$, then \eqref{eq: inhomogeneousHelmholtz weak} can be written as
\begin{equation}\label{eq: TQG weak elliptic}
    L_1(\tilde\psi, \phi) = F_{\omega^\alpha -  f^\alpha}(\phi).
\end{equation}
And similarly for \eqref{eq:inhomogeneousHelmholtzAlpha}, we have
\begin{equation}\label{eq: TQG weak elliptic alpha} 
    L_\alpha(\psi^\alpha, \phi) = F_{-\tilde\psi}(\phi).
\end{equation}
We discretise equations \eqref{eq: TQG weak elliptic} and \eqref{eq: TQG weak elliptic alpha} using a continuous
Galerkin (CG) discretisation scheme. 

Let $\delta$ be the discretisation parameter, and let $\domain_\delta$ denote a space filling triangulation of the domain, that consists of geometry-conforming nonoverlapping elements.
Define the approximation space
\begin{equation} \label{eq: cg space}
    W_\delta^{k}(\domain):=\left\{ \phi_\delta\in W^{1}\left(\domain\right)\ : \ \phi_\delta\in C\left(\domain\right),\left.\phi_\delta\right|_{K}\in\Pi^{k}\left(K\right)\text{ each } K \in \domain_\delta\right\}.
\end{equation}
in which $C(\domain)$ is the space of continuous functions on $\domain$, and  $\Pi^{k}\left(K\right)$ denotes the space of polynomials of degree at most $k$ on element $K\in \domain_\delta$. 

For \eqref{eq: TQG weak elliptic}, given $f_\delta \in W_\delta^k(\domain)$ and $\omega_\delta \in V_\delta^k(\domain)$ (see \eqref{DG space} for the definition of $V_\delta^k(\domain)$), 
our numerical approximation is the solution $\tilde\psi_\delta \in W_\delta^k(\domain)$ that satisfies 
\begin{equation}
L_1(\tilde\psi_\delta, \phi_\delta) = F_{\omega_\delta - f_\delta}(\phi_\delta)
\label{eq: TQG psi eqn weak discretised}
\end{equation}
for all test functions $\phi_\delta \in W_\delta^k(\domain)$. Then, using $\tilde\psi_\delta$, our numerical approximation of $\psi^\alpha$ is the solution $\psi_\delta^\alpha\in W_\delta^k(\domain)$ that satisfies
\begin{equation}\label{eq: TQG psi alpha eqn weak discretised}
L_\alpha(\psi_\delta^\alpha, \phi_\delta) = F_{-\tilde\psi_\delta} (\phi_\delta)
\end{equation}
for all test functions $\phi_\delta \in W_\delta^k(\domain)$.

 For a detailed exposition of the numerical algorithms that solves the discretised problems \eqref{eq: TQG psi eqn weak discretised} and \eqref{eq: TQG psi alpha eqn weak discretised} we point the reader to \cite{Gibson2019, Brenner2008}.

\subsubsection{Hyperbolic equations}\label{sec: buoyancy equation discretisation}
We choose to discretise the hyperbolic buoyancy \eqref{ceAlpha} and potential vorticity \eqref{meAlpha} equations using a discontinuous Galerkin (DG) scheme. For a detailed exposition of DG methods, we refer the interested reader to \cite{Hesthaven2008}.

Define the DG approximation space, denoted by $V_\delta^k(\domain)$, to be the element-wise polynomial space,
\begin{equation}\label{DG space}
    V_\delta^k(\domain) = \left\{ \left. v_\delta \in L^2(\domain) \right|\forall K \in \domain_\delta,\ \exists\phi_\delta\in \Pi^k (K):\ \left.v_\delta\right|_{K}=\left.\phi_\delta\right|_K\right\}.
\end{equation}
We look to approximate $b^\alpha$ and $\omega^\alpha$ in the space $V_\delta^k(\domain)$. Essentially, this means our approximations of $b^\alpha$ and $\omega^\alpha$ are each an direct sum over the elements in $\domain_\delta$. Additional constraints on the numerical fluxes across shared element boundaries are needed to ensure conservation properties and stability. 
Further, note that $W^k_\delta(\domain) \subset V^k_\delta(\domain)$. This inclusion is also needed for ensuring numerical conservation laws, see Section \ref{sec: numerics numerical conservation}.

For the buoyancy equation \eqref{ceAlpha}, we obtain the following variational formulation
\begin{align}
\langle \p_t b^\alpha, \nu_\delta \rangle _{K} & = \langle b^\alpha {\bf u}^\alpha,\nabla\nu_\delta\rangle _K - \langle b^\alpha {\bf u}^\alpha\cdot\normal,\nu_\delta\rangle _{\partial K},\quad K\in \domain_\delta
\label{eq: TQG b alpha equation weak}
\end{align}
where $\nu_\delta\in V^k_\delta(\domain)$ is any test function, $\p K$ denotes the boundary of $K$, and $\normal$ denotes the unit normal vector to $\p K$. Let $b^\alpha_\delta$ be the approximation of $b^\alpha$ in $V_\delta^k$, and let $\bu^\alpha_\delta = \nabla^\perp \psi^\alpha_\delta$ for $\psi^\alpha_\delta \in W^k_\delta$. Our discretised buoyancy equation over each element is given by
\begin{equation}
    \langle \p_t b^\alpha_\delta, \nu_\delta \rangle _{K} = \langle b^\alpha_\delta{\bf u}^\alpha_\delta,\nabla\nu_\delta\rangle _K - \langle b^\alpha_\delta{\bf u}^\alpha_\delta \cdot \normal,\nu_\delta\rangle _{\partial K}, \quad K \in \domain_\delta.
    \label{eq: TQG b alpha eqn weak discretised}
\end{equation}

Similarly, let $\omega^\alpha_\delta \in V_\delta^k(\domain)$ be the approximation of $\omega^\alpha$, and let $h_\delta \in W_\delta^k$. We obtain the following discretised variational formulation that corresponds to \eqref{meAlpha}, 
\begin{align}
\langle \p_t \omega^\alpha_\delta, \nu_\delta \rangle _{K} 
&= \langle (\omega^\alpha_\delta - b^\alpha_\delta){\bf u}^\alpha_\delta,\nabla\nu_\delta\rangle _{K}
-\langle (\omega^\alpha_\delta - b^\alpha_\delta) {\bf u}^\alpha_\delta \cdot\normal,\nu_\delta\rangle _{\partial K} \nonumber \\
& \qquad - \frac12\langle \nabla \cdot (b^\alpha_\delta \nabla^\perp h_\delta), \nu_\delta \rangle _K,
\qquad K\in \domain_\delta,
\label{eq: TQG omega alpha eqn weak discretised}
\end{align}
for test function $\nu_\delta \in V_\delta^k(\domain)$. 

At this point, we only have the discretised problem on single elements. To obtain the global approximation, we sum over all the elements in $\domain_\delta$. In doing so, the $\p K$ terms in \eqref{eq: TQG b alpha eqn weak discretised} and \eqref{eq: TQG omega alpha eqn weak discretised} must be treated carefully.
Let $\p K_\text{ext}$ denote the part of cell boundary that is contained in $\p \domain$. Let $\p K_\text{int}$ denote the part of the cell boundary that is contained in the interior of the domain $\domain\backslash \p \domain$.
On $\p K_\text{ext}$ we simply impose the PDE boundary conditions. However, on $\p K_\text{int}$ we need to consider the contribution from each of the neighbouring elements.
By choice, the approximant $\psi^\alpha_\delta \in W_\delta^k$ is continuous on $\p K$. And since
\begin{equation}
    \bu^\alpha_\delta \cdot \normal = \nabla^\perp \psi^\alpha_\delta \cdot \normal = -\nabla \psi^\alpha_\delta \cdot \hat \tau = -\frac{\dd\psi^\alpha_\delta}{\dd\hat\tau},
\end{equation}
where $\hat\tau$ denotes the unit tangential vector to $\p K$, $\bu^\alpha_\delta \cdot \normal$ is also continuous. This means  $\bu^\alpha_\delta \cdot \normal$ in \eqref{eq: TQG b alpha eqn weak discretised} (also in \eqref{eq: TQG omega alpha eqn weak discretised}) is single valued. However, due to the lack of global continuity constraint in the definition of $V_k^\alpha(\domain)$, $b^\alpha_\delta$ and $\omega^\alpha_\delta$ are multi-valued on $\p K_\text{int}$. Thus, as our approximation of $\omega^\alpha$ and $b^\alpha$ over the whole domain is the sum over $K \in \domain_\delta$,  we have to constrain the flux on the set $\left(\bigcup_{K\in\domain_\delta} \p K \right)\setminus \p \domain$. This is done using appropriately chosen numerical flux fields in the boundary terms of \eqref{eq: TQG b alpha eqn weak discretised} and \eqref{eq: TQG omega alpha eqn weak discretised}.

Let 
{$\nu^{-}:=\lim_{\epsilon\uparrow0} \nu ({\bf x}+\epsilon\normal)$} 
and 
{ $\nu^{+}:=\lim_{\epsilon\downarrow0}\nu ({\bf x}+\epsilon\normal)$}, 
for ${\bf x}\in\partial K$, be the inside and outside (with respect to a fixed element $K$) values respectively, of a function $\nu$ on the boundary.
Let $\hat{f}$ be a \emph{numerical flux} function 
that satisfies the following properties:
\begin{enumerate}[(i)]
    \item consistency 
    \begin{equation}
     \hat f (\nu, \nu, \bu_\delta \cdot \normal) = \nu \bu_\delta\cdot\normal
     \label{eq: nflux consistency}
     \end{equation}
    
    \item conservative
    \begin{equation}
        \hat f (\nu^+, \nu^-, \bu_\delta\cdot \normal) = - \hat f (\nu^-, \nu^+, -\bu_\delta\cdot \normal)
        \label{eq: nflux conservative}
    \end{equation}
    
    \item $L^2$ stable in the enstrophy norm with respect to the buoyancy equation, see \cite[Section 6]{Bernsen2006}.
\end{enumerate}
With such an $\hat f$, we replace 
$b^\alpha_\delta \bu^\alpha_\delta \cdot \normal$ by the numerical flux $\hat{f}(b_\delta^{\alpha, +},b_\delta^{\alpha, -},{\bf u}^\alpha_\delta \cdot\normal)$  in \eqref{eq: TQG b alpha eqn weak discretised}.
Similarly, in \eqref{eq: TQG omega alpha eqn weak discretised}, we replace $(\omega^\alpha_\delta - b^\alpha_\delta)\bu_\delta^\alpha \cdot \normal$ by $\hat{f}((\omega^\alpha_\delta - b^\alpha_\delta)^{+}, (\omega^\alpha_\delta - b^\alpha_\delta)^{-}, \bu_\delta^\alpha \cdot \normal)$.

\begin{remark}
For a general nonlinear conservation law, one has to solve what is called the \emph{Riemann problem} for the numerical flux, see \cite{Hesthaven2008} for details.
In our setup, we use the following local Lax-Friedrichs flux, which is an approximate Riemann solver,
\begin{equation}
\hat{f}(\nu^+, \nu^-, \bu_\delta\cdot\normal) = \bu_\delta\cdot\normal
\{\{ \nu  \}\} - \frac{|\bu_\delta\cdot\normal|}{2}\llbracket\nu\rrbracket
\end{equation}
where
\begin{equation}
\{\{\nu\}\}:=\frac{1}{2}(\nu^{-}+\nu^{+}),\qquad
\llbracket\nu\rrbracket:=\normal^{-}\nu^{-}+\normal^{+}\nu^{+}.
\end{equation}
\end{remark}

Finally, our goal is to find $b^\alpha_\delta, \omega^\alpha_\delta \in V_\delta^k(\domain)$ such that for all {$\nu_\delta\in V_\delta^k(\domain)$} we have
\begin{align}
\sum_{K\in \domain_\delta} \langle \p_t b^\alpha_\delta, \nu_\delta\rangle _{K} 
&=
\sum_{K\in \domain_\delta}
\big\{
\langle b^\alpha_\delta\nabla^{\perp} \psi^\alpha_\delta, \nabla \nu_\delta\rangle _{K}  -
\langle \hat{f}_{b^\alpha}(b_\delta^{\alpha, +},b_\delta^{\alpha, -},\nabla^{\perp}\psi^\alpha_\delta \cdot \normal), \nu_\delta^{-}\rangle _{\partial K}
\big\}, 
\label{eq: TQG b eqn weak discretised numerical flux}
\\
\sum_{K\in \domain_\delta} \langle \p_t \omega^\alpha_\delta, \nu_\delta \rangle _{K} 
&= 
\sum_{K\in \domain_\delta} 
\big\{
\langle (\omega^\alpha_\delta - b^\alpha_\delta)\nabla^{\perp}\psi^\alpha_\delta,\nabla\nu_\delta\rangle _{K} - \frac12\langle \nabla \cdot (b^\alpha_\delta \nabla^\perp h_\delta), \nu_\delta \rangle _K
\label{eq: TQG omega eqn weak discretised numerical flux}
\\
&\qquad -
\langle \hat{f}_{\omega^\alpha}((\omega^\alpha_\delta - b^\alpha_\delta)^{+}, (\omega^\alpha_\delta - b^\alpha_\delta)^{-}, \nabla^\perp \psi_\delta^\alpha \cdot \normal),
\nu_\delta^ -
\rangle _{\partial K}
\big\}
\nonumber
\end{align}
with $\psi_\delta^\alpha \in W_\delta^k(\domain)$ being the numerical approximation to the stream function. 

\begin{remark}
In \eqref{eq: TQG b eqn weak discretised numerical flux} and \eqref{eq: TQG omega eqn weak discretised numerical flux} we do not explicitly distinguish $\p K_\text{ext}$ and $\p K_\text{int}$ because for the boundary conditions \eqref{eq: tqg boundary conditions numerical}, the $\p K_\text{ext}$ terms vanish. 
\end{remark}

\subsubsection{Numerical conservation}
\label{sec: numerics numerical conservation}
Conservation properties of the TQG system was first shown in \cite{holm2021stochastica}. More specifically, the TQG system conserves energy, and an infinite family of quantities called casimirs. Proposition \ref{prop: TQG alpha conservation} below describes the conservation properties of the $\alpha$--TQG system. We note that the form of the conserved energy and casimirs are the same as that of the TQG system.
Although the result is stated for the system with boundary conditions, it is easy to show that the same result holds for $\mathbb T^2$ .

\begin{proposition}[$\alpha$--TQG conserved quantities]\label{prop: TQG alpha conservation}
On a bounded domain $\domain$ with boundary $\p \domain$, consider the $\alpha$-TQG system \eqref{ceAlpha} -- \eqref{constrtAlpha} 
with boundary conditions
\begin{equation}\label{eq: tqg boundary conditions}
    \normal\cdot \bu^\alpha = 0, \  \normal\times \nabla b^\alpha = 0,\   
    \frac \dd {\dd t}
    \int_{\p \domain} \nabla \psi^\alpha \cdot \normal = 0, \ 
    \frac \dd {\dd t}
    \int_{\p \domain}  \nabla \Delta \psi^\alpha \cdot \normal = 0, \ 
    \normal \times \nabla \Delta \psi^\alpha = 0
\end{equation}
We have
\begin{equation}\label{eq: tqg energy}
     E^\alpha(t) := - \frac12  \int_\domain \{(\omega^\alpha - f)\psi^\alpha + \rmh b^\alpha \} \dd \bx,
\end{equation}
and
\begin{equation}\label{eq: tqg casimirs}
     C^\alpha_{\Psi, \Phi}(t) :=   \int_\domain \{ \Psi(b^\alpha) + \omega^\alpha \Phi(b^\alpha) \} \dd \bx ,\qquad \forall \Psi, \Phi \in C^\infty
\end{equation}
are conserved, i.e. $\dd_t E^\alpha(t) = 0$ and $\dd_t C^\alpha_{\Psi, \Phi}(t) = 0$.
\end{proposition}

\begin{proof}
For the energy $E^\alpha(t)$, we obtain
\begin{align*}
    \frac \dd {\dd t} E^\alpha(t) &= -\frac12 \frac \dd {\dd t} \int_\domain \{\psi^\alpha (\Delta - 1)(1-\alpha \Delta)\psi^\alpha + \rmh b^\alpha \} \dd \bx \\
    &=\int_\domain \{-\psi^\alpha \ \p_t \omega^\alpha - \frac12 h \p_t b^\alpha \} \dd\bx.
\end{align*}
Substitute $\p_t \omega^\alpha$ and $\p_t b^\alpha$ using \eqref{ceAlpha} and \eqref{meAlpha},
the result then follows from direct calculations.

Similarly for the casimirs, we obtain the result by directly evaluating $\dd_t C^\alpha_{\Psi, \Phi}(t)$.
\end{proof}

\begin{remark}
In \eqref{eq: tqg boundary conditions}, except for the integral Neumann boundary conditions, we effectively have Dirichlet boundary conditions for $\psi^\alpha$, $b^\alpha$ and $\Delta \psi^\alpha$. The integral Neumann boundary conditions can be viewed as imposing the Kelvin theorem on the boundary -- consider
\begin{equation}
    \frac \dd {\dd t}
    \int_{\p \domain}  \nabla (\psi^\alpha - \alpha \Delta \psi^\alpha) \cdot \normal = 0,
\end{equation}
and apply the divergence theorem.

In our FEM discretisation, we impose only Dirichlet boundary conditions.
If the integral Neumann boundary conditions are not imposed, then for Proposition \ref{prop: TQG alpha conservation} to hold, we necessarily require $\psi^\alpha$ and $\Delta \psi^\alpha$ on $\p \domain$ to be \emph{zero}.
\end{remark}

We now analyse the conservation properties of our spatially discretised $\alpha$-TQG system. 
Let $\langle \cdot, \cdot \rangle_{H^1(\domain)}$ be the $H^1(\domain)$ inner product.
Define the numerical total energy
\begin{equation}\label{eq: numerical energy}
    E^\alpha_\delta(t) = 
    \frac12
    \langle \psi^\alpha_\delta
    , \tilde \psi_\delta
    \rangle_{H^1(\domain)}
    - 
    \frac12 
    \langle \rmh_\delta, b^\alpha_\delta \rangle_\domain,
\end{equation}
in which $\psi_\delta^\alpha, \tilde{\psi}_\delta \in W_\delta^\alpha(\domain)$ and $b_\delta^\alpha\in V_\delta^k(\domain)$ are the numerical approximations of $\psi^\alpha$, $\tilde{\psi}$ and $b^\alpha$ respectively.
Define the numerical casimir functional
\begin{equation}\label{eq: TQG numerical casimirs}
C_{\delta}^\alpha(t; \Psi, \Phi ):= \langle\Psi(b_\delta^\alpha),1 \rangle_\domain
+ \langle \omega_\delta^\alpha \Phi(b_\delta^\alpha) \rangle_\domain, 
\quad \Psi, \Phi \in C^\infty.
\end{equation}

\begin{lemma} 
With boundary conditions \eqref{eq: tqg boundary conditions numerical}
\begin{equation}
    \frac \dd {\dd t} E^\alpha_\delta(t) = 0
\end{equation}
for $\alpha = 0$.
\end{lemma}
In other words, the semi-discrete discretisation conserves energy in the TQG case.
\begin{proof}
We consider $\dd/\dd_t$ evaluated at an arbitrary fixed value $t_0$.
First note that when $\alpha=0$, we have $\psi^\alpha_\delta = \tilde \psi_\delta$. Then, from \eqref{eq: psi bilinear form} we have
\begin{align}
    2E_\delta^\alpha(t) 
    &= L_1 (\psi_\delta^\alpha, {\psi}^\alpha_\delta) - \langle \rmh_\delta, b_\delta^\alpha \rangle_\Omega
    \label{eq: TQG total energy numerical}
\end{align}
For the first term in \eqref{eq: TQG total energy numerical}, following \eqref{eq: TQG psi eqn weak discretised} we have
\begin{align}
    \left.L_1(\p_t \psi_\delta, \phi_\delta)\right|_{t=t_0} = - \sum_{K\in \mathcal \domain_\delta} \left.\langle\p_t \omega^\alpha_\delta, \phi_\delta \rangle_K \right|_{t=t_0},
    \quad 
    \forall \phi_\delta \in W_\delta^k(\domain).
\end{align}
 Thus, substituting in the discretised equation \eqref{eq: TQG omega eqn weak discretised numerical flux} for $\omega^\alpha$, we obtain
\begin{align}
    - L_1(\left.\p_t \psi_\delta, \phi_\delta)\right|_{t=t_0}
&= 
\sum_{K\in \domain_\delta} 
\left[
\big\{
\langle (\omega^\alpha_\delta - b^\alpha_\delta)\nabla^{\perp}\psi^\alpha_\delta,\nabla\phi_\delta\rangle _{K} - \frac12\langle \nabla \cdot (b^\alpha_\delta \nabla^\perp h_\delta), \phi_\delta \rangle _K
\right.
\nonumber
\\
&\qquad -
\left.
\langle \hat{f}_{\omega^\alpha}((\omega^\alpha_\delta - b^\alpha_\delta)^{+}, (\omega^\alpha_\delta - b^\alpha_\delta)^{-}, \nabla^\perp \psi_\delta^\alpha \cdot \normal),
\phi_\delta^ -
\rangle _{\partial K}
\big\}
\right]_{t=t_0},
\label{eq: TQG kinetic energy time differential}
\end{align}
in which 
we can choose $\phi_\delta=[\psi^\alpha_\delta]_{t=t_0}$. This choice is consistent with the discretisation space $V_\delta^k(\domain)$ for $\omega^\alpha$, since $W_\delta^k(\domain) \subset V_\delta^k(\domain)$.

Similarly, for the second term in \eqref{eq: TQG total energy numerical}, using \eqref{eq: TQG b eqn weak discretised numerical flux} we obtain
\begin{equation}
    \left.\frac \dd {\dd t}\right\vert_{t=t_0} \langle b^\alpha_\delta, \nu_\delta \rangle_\domain
    = \sum_{K\in \domain_\delta} \left[ 
    \big\{
        \langle b^\alpha_\delta\nabla^{\perp} \psi^\alpha_\delta, \nabla \nu_\delta\rangle _{K}  -
        \langle \hat{f}_{b^\alpha}(b_\delta^{\alpha, +},b_\delta^{\alpha, -},\nabla^{\perp}\psi^\alpha_\delta \cdot \normal), \nu_\delta^{-}\rangle _{\partial K}
    \big\}
    \right]_{t=t_0},
    \label{eq: TQG potential energy time differential}
\end{equation}
where we can choose $\nu_\delta = \rmh_\delta$.

Putting together \eqref{eq: TQG total energy numerical}, \eqref{eq: TQG kinetic energy time differential} and \eqref{eq: TQG potential energy time differential}, and substituting in $\rmh_\delta$ and $[\psi^\alpha_\delta]_{t=t_0}$ for the arbitrary choices, we obtain
\begin{align}
    2 \frac \dd {\dd t}\Big|_{t=t_0} E^\alpha_\delta(t)
    &= 
    - \sum_{K\in \domain_\delta} 
    \left[
    2 \Big\{
    \langle (\omega^\alpha_\delta - b^\alpha_\delta)
    \underbrace{
    \nabla^{\perp}
    \psi^\alpha_\delta,\nabla
    \psi^\alpha_\delta
    }_{=0}
    \rangle _{K} - \frac12\langle \nabla \cdot (b^\alpha_\delta \nabla^\perp h_\delta), \psi^\alpha_\delta
    \rangle _K
    \right.
    \nonumber
    \\
    &\qquad 
    -
    \left.
    \langle \hat{f}_{\omega^\alpha}((\omega^\alpha_\delta - b^\alpha_\delta)^{+}, (\omega^\alpha_\delta - b^\alpha_\delta)^{-}, \nabla^\perp \psi_\delta^\alpha \cdot \normal),
    {\psi^\alpha_\delta}^-
    \rangle _{\partial K}
    \Big\}
    \right.
    \nonumber
    \\
    & 
    +
    \left. 
    \Big\{
        \underbrace{
        \langle
        b^\alpha_\delta\nabla^{\perp} \psi^\alpha_\delta, \nabla \rmh_\delta\rangle _{K}  
        }_{=- \langle
        b^\alpha_\delta\nabla^{\perp}
        \rmh_\delta, \nabla \psi^\alpha_\delta\rangle _{K}  }
        -
        \langle \hat{f}_{b^\alpha}(b_\delta^{\alpha, +},b_\delta^{\alpha, -},\nabla^{\perp}\psi^\alpha_\delta \cdot \normal), \rmh_\delta^{-}\rangle _{\partial K}
    \Big\}
    \right]_{t=t_0}
\end{align}
\begin{align}
    \frac \dd {\dd t} E^\alpha_\delta(t) &= \sum_{K\in \mathcal \domain_\delta} 
    \Big[
    - \langle \hat f_{\omega^\alpha}((\omega^\alpha_\delta-b^\alpha_\delta)^+,(\omega_\delta^\alpha - b^\alpha_\delta)^-, \nabla^\perp \psi^\alpha_\delta\cdot\normal),{\psi^\alpha_\delta}^-\rangle _{\partial K}
    \nonumber \\
    & \qquad 
    - \frac12
    \langle \hat{f}_{b^\alpha}({b^\alpha_\delta}^{+},{b^\alpha_\delta}^{-},\nabla^{\perp}\psi^\alpha_\delta \cdot \normal), \rmh_\delta^{-}\rangle _{\partial K} \Big]_{t=t_0}
    \label{eq: numerical energy temp}
\end{align}
Since $\psi_\delta$ and $\rmh_\delta$ are continuous across element boundaries, we have $\psi_\delta^+ = \psi_\delta^-$ and $\rmh_\delta^+ = \rmh_\delta^-$.
Thus, if the numerical fluxes $\hat f_{\omega^\alpha}$ and $\hat f_{b^\alpha}$ satisfy the conservative property \eqref{eq: nflux conservative}, the sum of flux terms is zero.
\end{proof}

\begin{remark}
In order for our semi-discrete discretisation to conserve numerical energy when $\alpha > 0$, more regularity is required for the approximation space for $\psi^\alpha_\delta$. Additionally, the current scheme coupled with the time stepping algorithm do not conserve  casimirs. For future work, we look to explore casimir conserving schemes for the model.
\end{remark}



\subsubsection{Time stepping}
To discretise the time derivative, we use the strong stability preserving Runge-Kutta of order 3 (SSPRK3) scheme, see \cite{Hesthaven2008}.
Writing the finite element spatial discretisation formally as 
\begin{equation}
    \partial_{t}b^\alpha_\delta={\rm f}_\delta(b^\alpha_\delta)
\end{equation}
where ${\rm f}_\delta$ is the discretisation operator that follows from \eqref{eq: TQG b eqn weak discretised numerical flux}, and
\begin{equation}
     \partial_t \omega = {\rm g}_\delta (\omega, b)
\end{equation}
where ${\rm g}_\delta$ is the discretisation operator that follows from \eqref{eq: TQG omega eqn weak discretised numerical flux}.
Let $b^{\alpha, n}_\delta$ and $\omega^{\alpha,n}_\delta$ denote the approximation of $b^\alpha_\delta$ and $\omega^\alpha_\delta$ at time step $t_n$.
The SSPRK3
time discretisation is as follows 
\begin{subequations}
\begin{align}
    b^{(1)} & = b^{\alpha,n}_\delta + \Delta t \ {\rm f}_\delta (b^{\alpha,n}_\delta)\\
	\omega^{(1)} & =\omega^{\alpha,n}_\delta+\Delta t \  {\rm g}_\delta\left(\omega^{\alpha,n}_\delta, b^{\alpha,n}_\delta \right)\\
	b^{\left(2\right)} 
	&=\frac{3}{4}b^{\alpha,n}_\delta +\frac{1}{4}\left(b^{\left(1\right)}+\Delta t \ {\rm f}_\delta\left(b^{\left(1\right)}\right)\right)\\
	\omega^{\left(2\right)} 
	&= \frac{3}{4}\omega^{\alpha,n}_\delta
	+\frac{1}{4}\left(\omega^{\left(1\right)}+\Delta t \  {\rm g}_\delta\left(\omega^{\left(1\right)}, b^{(1)}\right)\right)\\
	b^{\alpha,n+1}_\delta & =\frac{1}{3}b^{\alpha,n}_\delta+\frac{2}{3}\left(b^{\left(2\right)}+\Delta t \  {\rm f}_\delta\left(b^{\left(2\right)}\right)\right)\\
	\omega^{\alpha,n+1}_\delta & =\frac{1}{3}\omega^{\alpha,n}_\delta
	+\frac{2}{3}\left(\omega^{\left(2\right)}+\Delta t \  {\rm g}_\delta\left(\omega^{\left(2\right)}, b^{(2)}\right)\right)
\end{align}
\end{subequations}
where $\Delta t = t_{n+1}-t_{n}$ each $n$.

\subsection{$\alpha$-TQG linear thermal Rossby wave stability analysis}\label{sec: numerics linear stability analysis}


A dispersion relation for the TQG system was derived in \cite{holm2021stochastica}. There, the authors showed that the linear thermal Rossby waves of the TQG system possess high wavenumber instabilities. More specifically, the Doppler-shifted phase speed of these waves becomes complex at sufficiently high wavenumbers. However, the growth rate of the instability decreases to zero as $|\bk|^{-1}$, in the limit that $|\bk|\to\infty$. Consequently, the TQG dynamics is linearly well-posed. That is, the TQG solution depends continuously on initial conditions. 
In this subsection, using the same equilibrium state as in \cite{holm2021stochastica}, we derive a dispersion relation for the  thermal Rossby wave solutions of the linearised $\alpha$-TQG system. Again the Doppler-shifted phase speed of these waves becomes complex at sufficiently high wavenumbers. However, in this case, the growth rate of the instability decreases to zero as $|\bk|^{-2}$, in the limit that $|\bk|\to\infty$. In the limit that $\alpha$ tends to zero, one recovers the TQG dispersion relation derived in \cite{holm2021stochastica}.

For the reader's convenience, we repeat the equations solved by the pair of variables $(b^\alpha,\omega^\alpha)$ in the $\alpha$-TQG system in \eqref{ceAlpha}--\eqref{meAlpha}, as formulated now in vorticity-streamfunction form with fluid velocity given by ${\bu}^\alpha =\nabla^\perp \psi^\alpha$. This formulation is given by
\begin{align}
\frac{\partial}{\partial t} b^\alpha + J(\psi^\alpha,b^\alpha) &= 0,
\label{ceZero-redux}\\
\frac{\partial}{\partial t}\omega^\alpha + J(\psi^\alpha,\omega^\alpha-b^\alpha) &= -\frac{1}{2}J(h,b^\alpha),
\label{meZero-redux}\\
\omega^\alpha &= (\Delta-1) (1-\alpha \Delta )\psi^\alpha + f.
\label{constrtZero-redux}
\end{align}
Here, as before, $J(a,b)=\nabla^\perp a\cdot \nabla b=a_x b_y - b_x a_y$ denotes the Jacobian of any smooth functions $a$ and $b$ defined on the $(x,y)$ plane.

Equilibrium states of the $\alpha$-TQG system in the equations \eqref{ceZero-redux}-\eqref{constrtZero-redux} satisfy $J(\psi^\alpha_e, \omega^\alpha_e) = J(\psi^\alpha_e, b^\alpha_e)=J(h, b^\alpha_e)=0.$
The equilibrium TQG state considered in \cite{holm2021stochastica} was given via the specification of the following
gradient fields 
\[
\begin{array}{c}
\nabla\psi^\alpha_{e}=-U\yhat,\ \nabla\omega^\alpha_{e}=(U-\beta)\yhat,\ \nabla f=-\beta\yhat\\
\\
\nabla b^\alpha_{e}=-B\yhat,\ \nabla\rmh=-H\yhat
\end{array}
\]
where we have taken $f=1-\beta y$ and the equilibrium parameters $U, \beta, B, H \in \reals$ are all constants.
Linearising the $\alpha$-TQG system around the steady state produces the
following evolution equations for the perturbations $\omega'$, $b'$
and ${\psi}'$
\begin{subequations}\label{eq: linear alpha tqg}
\begin{align}
\omega_{t}'+U\omega'_{x}+(U+B-\beta){\psi}_{x}' & =(U-H/2)b'_{x}
\\
b'_{t}+Ub_{x}'-B{\psi}_{x}' & =0\\
 \Big((1-\alpha\Delta)(1-\Delta)\Big){\psi}' & =-\,\omega'. \label{eq: linear alpha tqg psi}
\end{align}
\end{subequations}
Since these equations are linear with constant coefficients, one has
the plane wave solutions
\[
\omega'=e^{i(\bk\cdot\bx-\nu t)}\hat{\omega},\quad b'=e^{i(\bk\cdot\bx-\nu t)}\hat{b},\quad{\psi}'=e^{i(\bk\cdot\bx-\nu t)}\hat{{\psi}}.
\]
Further, from \eqref{eq: linear alpha tqg psi} we obtain
\[
-\,\hat{\omega}=(|\bk|^{2}+1)(\alpha|\bk|^{2}+1)\hat{{\psi}}.
\]
Substituting these solutions into the linearised equations we have
\[
\begin{aligned}(\nu-kU)\hat{\omega}-k(U+B-\beta)\hat{{\psi}} & =-k(U-H/2)\hat{b}\\
\left(\nu-kU\right)\hat{b} & =-kB\hat{{\psi}}
\end{aligned}
\]
where $k$ is the first component of $\bk$.
From the above we obtain a quadratic formula for the Doppler-shifted phase speed $C = C(\alpha):=\left(\nu(\alpha)-kU\right)/k$
\[
C^{2}(|\bk|^{2}+1)(\alpha|\bk|^{2}+1)+CX+Y=0.
\]
where $X:=U+B-\beta$
and $Y:=(U-H/2)B$. Thus, the dispersion relation for the thermal Rossby wave for TQG possesses two branches, 
corresponding to two different phase velocities,
\begin{equation}\label{eq: doppler shift phase speed}
C(\alpha)=\frac{-X\pm\sqrt{X^{2}-4Y(|\bk|^{2}+1)(\alpha|\bk|^{2}+1)}}
{2 (|\bk|^{2}+1)(\alpha|\bk|^{2}+1)}.
\end{equation}
Upon setting $\alpha=0$ in \eqref{eq: doppler shift phase speed}, we recover the Doppler-shifted phase speed of  thermal Rossby waves for the TQG system in \cite{holm2021stochastica}. 
{
When $\alpha=0$, $B=0$, and $U=0$ in \eqref{eq: doppler shift phase speed}, the remaining dispersion relation differs slightly from the dispersion relation for QG Rossby waves in non-dimensional variables. This is because the Casimirs in the Hamiltonian formulation of the TQG system differ from the those of standard QG.
}

As discussed in \cite{holm2021stochastica} for the TQG system, when $Y>0$,  the Doppler shifted phase velocity in \eqref{eq: doppler shift phase speed} for the linearised wave motion becomes complex at high wavenumber $|\bk| \gg 1$; namely for $(|\bk|^{2}+1)(\alpha|\bk|^{2}+1)\ge X^{2}/(4Y)$, for both branches of the dispersion relation. However, the growth-rate of the instability found from the imaginary part of $C(\alpha)$ in \eqref{eq: doppler shift phase speed} decays to zero as $O( |\bk|^{-2})$, for $|\bk|\gg1$. Indeed, the linear stability analysis leading to $C(\alpha)$ predicts that the maximum growth rate occurs at a finite wavenumber $|\bk|_{\max}$ which depends on the value of $\alpha$, beyond which the growth rate of the linearised TQG wave amplitude falls rapidly to zero. One would expect that simulated solutions of $\alpha$-TQG would be most active at the length-scale corresponding to $|\bk|_{\max}$, at which the linearised thermal Rossby waves are the most unstable. If this maximum activity length-scale is near the grid truncation size, then numerical truncation errors could cause additional numerical stability issues! We have experience such problems during our testing of the numerical algorithm for the TQG system. Even for equilibrium solution ``sanity" check tests, we have found that unless the time step was taken to be incredibly small, high wavenumber truncation errors have caused our numerical solutions to eventually blow up.

\begin{figure}[h!]
\centering
\includegraphics[width=0.6\textwidth]{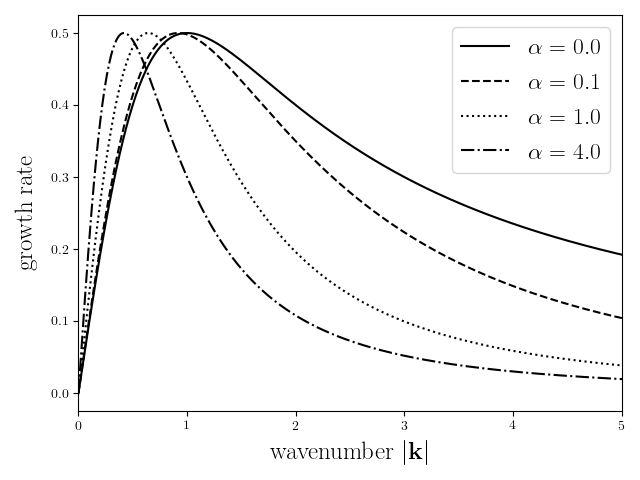}
\caption{An example of the growth-rate of linearised TQG waves as determined from the imaginary part of $C(\alpha)$ in equation \eqref{eq: doppler shift phase speed}, plotted for different values of $\alpha$. We observe that increasing $\alpha$ shifts the wavenumber at the maximum growth-rate to lower values.}
\label{fig: growth rate}
\end{figure}

The magnitude of $\alpha$  controls the unstable growth rate of linearised TQG waves at asymptotically high wavenumbers. That is, the presence of $\alpha$ regularises the wave activity at high wavenumbers, see Figure \ref{fig: growth rate}. Thus, from the perspective of numerics, the $\alpha$ regularisation is available to control numerical problems that may arise from the inherent model instability  at high wavenumbers, without the need for additional dissipation terms in the equation.

\subsection{Numerical example}\label{sec: numerics example}
The numerical setup we consider for this paper is as follows.
The spatial domain $\domain$ is an unit square with doubly periodic boundaries.
We chose to discretise $\domain$ using a grid that 
consists of $256\times256$ cells, i.e. $\cardinality(\domain_\delta) = 256\times256$. This was the maximum resolution we could computationally afford, to obtain results over a reasonable amount of time.

We computed $\alpha$-TQG solutions for the following values of $\alpha$ -- $\frac{1}{16^{2}},\frac{1}{32^{2}},\frac{1}{64^{2}},\frac{1}{128^{2}},\frac{1}{180^{2}},\frac{1}{220^{2}},\frac{1}{256^{2}}$ and $0$.
Note when $\alpha=0$ we get the TQG system.
For all cases, the following initial conditions were used,
\begin{align}
\omega(0,x,y) & =\sin(8\pi x)\sin(8\pi y)+0.4\cos(6\pi x)\cos(6\pi y)\label{eq: pv initial condition}\\
 & \quad+0.3\cos(10\pi x)\cos(4\pi y)+0.02\sin(2\pi y)+0.02\sin(2\pi x) \nonumber \\
b(0,x,y) & =\sin(2\pi y)-1, \label{eq: buoyancy initial condition}
\end{align}
see Figure \ref{fig: initial conditions} for illustrations,
as well as the following bathymetry and rotation fields
\begin{align}
h(x,y) & =\cos(2\pi x)+0.5\cos(4\pi x)+\frac{1}{2}\cos(6\pi x) \label{eq: bathymetry}\\
f(x,y) & =0.4\cos(4\pi x)\cos(4\pi y). \label{eq: rotation}
\end{align}
For time stepping, we used $\Delta t = 0.0005$ in all cases to facilitate comparisons. This choice satisfies the CFL condition for the $\alpha=0$ case. 

We computed each solution for $5000$ time steps\footnote{Go to https://youtu.be/a2a4xzft3Pg for a video of the full simulation.}. Figures \ref{fig: snapshot buoyancy}, \ref{fig: snapshot pv} and \ref{fig: snapshot velocity magnitude} show $\alpha$-TQG solution snapshots of buoyancy, potential vorticity and velocity magnitude respectively, at the $1280$'th and $2600$'th time steps, for $\alpha$ values $0, \frac1{128^2}, \frac1{64^2}$ and $\frac1{16^2}$. The interpretation of the regularisation parameter $\alpha$ is that its square root value corresponds to the fraction of the domain's length scale that get regularised.
At the $1280$'th time step, the flows are in early spin-up phase.
At the $2600$'th time step, although the flows are still in spin-up phase, much more flow features have developed.
The sub-figures illustrate via comparisons, how $\alpha$ controls the development of small scale features and instabilities. Increasing $\alpha$ leads to more regularisation at larger scales.

In view of Proposition \ref{prop:alphaConv}, we investigated numerically the convergence of $\alpha$-TQG to TQG using our numerical setup. Consider the relative error between TQG buoyancy and $\alpha$-TQG buoyancy, 
\begin{align}
e_b(t, \alpha) &:= \frac{\|b(t) - b^\alpha(t)\|_{H^1}}{\|b(t)\|_{H^1}}\label{eq: relative error buoyancy}
\end{align}
and the relative error between TQG potential vorticity and $\alpha$-TQG potential vorticity,
\begin{align}
    e_\omega(t, \alpha) &:= \frac{\|\omega(t) - \omega^\alpha(t)\|_{L^2}}{\|\omega(t)\|_{L^2}}.\label{eq: relative error pv}
\end{align}
The norms were chosen in view of Proposition \ref{prop:alphaConv}.
Figures \ref{fig: relative error buoyancy time series} and \ref{fig: relative error pv time series} show plots of $e_b(t, \alpha)$ and $e_\omega(t, \alpha)$
as functions of time only, for fixed $\alpha$ values $\frac{1}{16^2}, \frac{1}{32^2}, \frac1{64^2}, \frac1{128^2}$ and $\frac1{220^2}$, from the initial time up to and including the $2800$'th time step.
We observe that starting from $0$, the relative errors $e_b$ and $e_\omega$ initially increase over time but plateau at around and beyond the $2000$'th time step. 
Up to the $1400$'th time step, 
the plotted relative errors remain less than $1.0$, and are arranged in the ascending order of $\alpha$, i.e. smaller $\alpha$'s give smaller relative errors.

We note that, if we compare and contrast the relative error results with the solution snapshots of buoyancy and potential vorticity, particularly at the $2600$'th time step (shown in Figures \ref{fig: snapshot buoyancy 1.30} and \ref{fig: snapshot pv 1.30}), we observe "discrepancies" -- for lack of a better term -- between the results. For example, in Figure \ref{fig: snapshot buoyancy 1.30}, if we compare the $\alpha=\frac1{128^2}$ snapshot with the $\alpha=\frac1{16^2}$ one, the former
show small scale features that are much closer to those that exist in the reference TQG $\alpha=0$ solution. However, according to the relative error measurements, 
the two regularised solutions are more or less equivalent in their differences to the reference TQG solution at time step $2600$.

If we fix the value of the time parameter in $e_b(t,\alpha)$ and $e_\omega(t, \alpha)$, and vary $\alpha$, we can estimate the convergence rate of $\alpha$-TQG to TQG for our numerical setup. Proposition \ref{prop:alphaConv} predicts a convergence rate of $1$ in $\alpha$, using the $H^1$ norm on the buoyancy differences and the $L^2$ norm on the potential vorticity differences.
Note that, in \eqref{rateOfConv}, the left hand side evaluates a supremum over a given compact time interval. We see in Figures \ref{fig: relative error buoyancy time series} and \ref{fig: relative error pv time series} that the relative errors are monotonic up to around the $1500$'th time step mark. So for a given time point $T$ between the initial time and the $1500$'th time step, we assume we can evaluate $e_b$ and $e_\omega$ at $T$ to estimate the supremum over $[0,T]$. 

Figures \ref{fig: relative error convergence buoyancy} and \ref{fig: relative error convergence pv} show plots -- in log-log scale -- of $e_b$ and $e_\omega$ as functions of $\alpha$ only at the $600$'th, $800$'th, $1000$'th and $1200$'th time steps, for all the chosen $\alpha$ values. These numbers of time steps correspond to $T=0.3$, $T=0.4$, $T=0.5$ and $T=0.6$ respectively. We also plotted in Figures \ref{fig: relative error convergence buoyancy} and \ref{fig: relative error convergence pv} linear functions of $\alpha$ to provide reference order 1 slopes.
Comparing the results to the reference, we see that the theoretical convergence rate of 1 is attained for $T=0.3$ ($600$'th time step), $T=0.4$ ($800$'th time step) and $T=0.5$ ($1000$'th time step). However for $T=0.6$ ($1200$'th time step), the convergence rates are at best order $1/2$.

\begin{figure}
\centering
\includegraphics[width=0.6\textwidth]{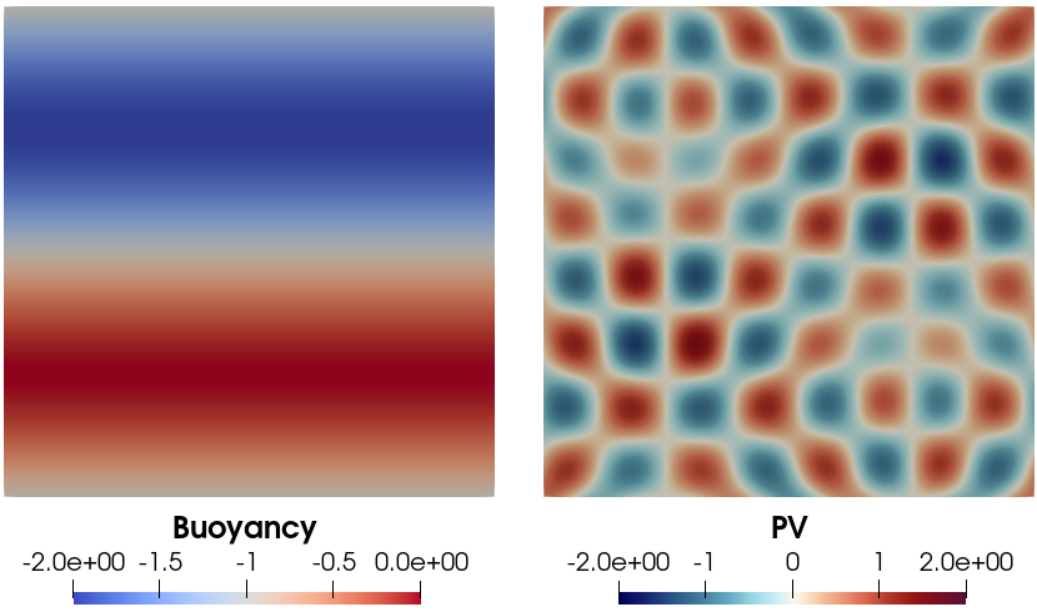}
\caption{Initial conditions -- buoyancy \eqref{eq: buoyancy initial condition} on the left, and potential vorticity \eqref{eq: pv initial condition} on the right. The different colours of the PV plot correspond to positive and negative values of PV, which can be interpreted as clockwise and anticlockwise eddies respectively.}
\label{fig: initial conditions} 
\end{figure}

\begin{figure}[t]
\centering
 \begin{subfigure}[b]{0.6\textwidth}
     \centering
     \includegraphics[trim={340 25 196 5},clip,width=\textwidth]{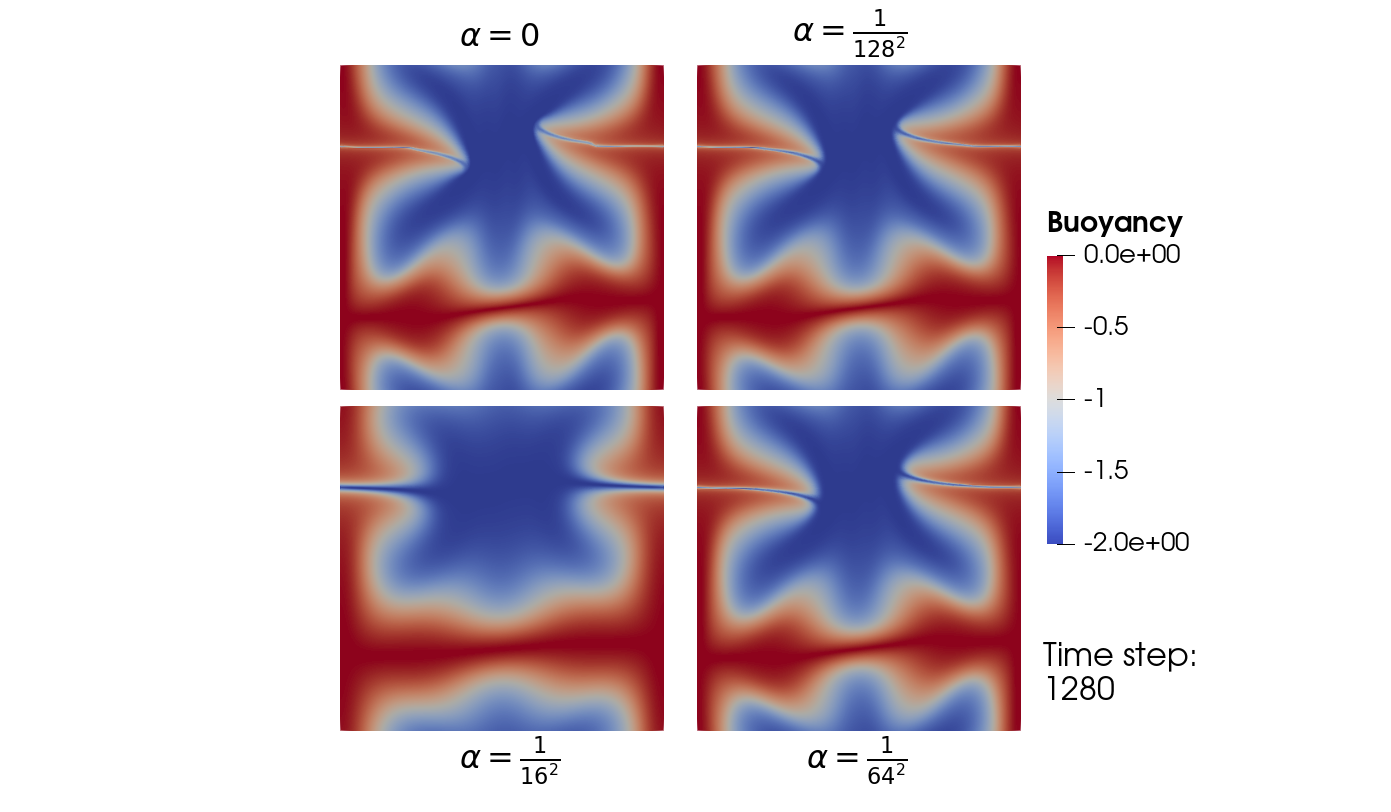}
     \caption{}
     \label{fig: snapshot buoyancy 0.64}
 \end{subfigure}
 \vfill
 \begin{subfigure}[b]{0.6\textwidth}
     \centering
     \includegraphics[trim={340 25 196 5},clip, width=\textwidth]{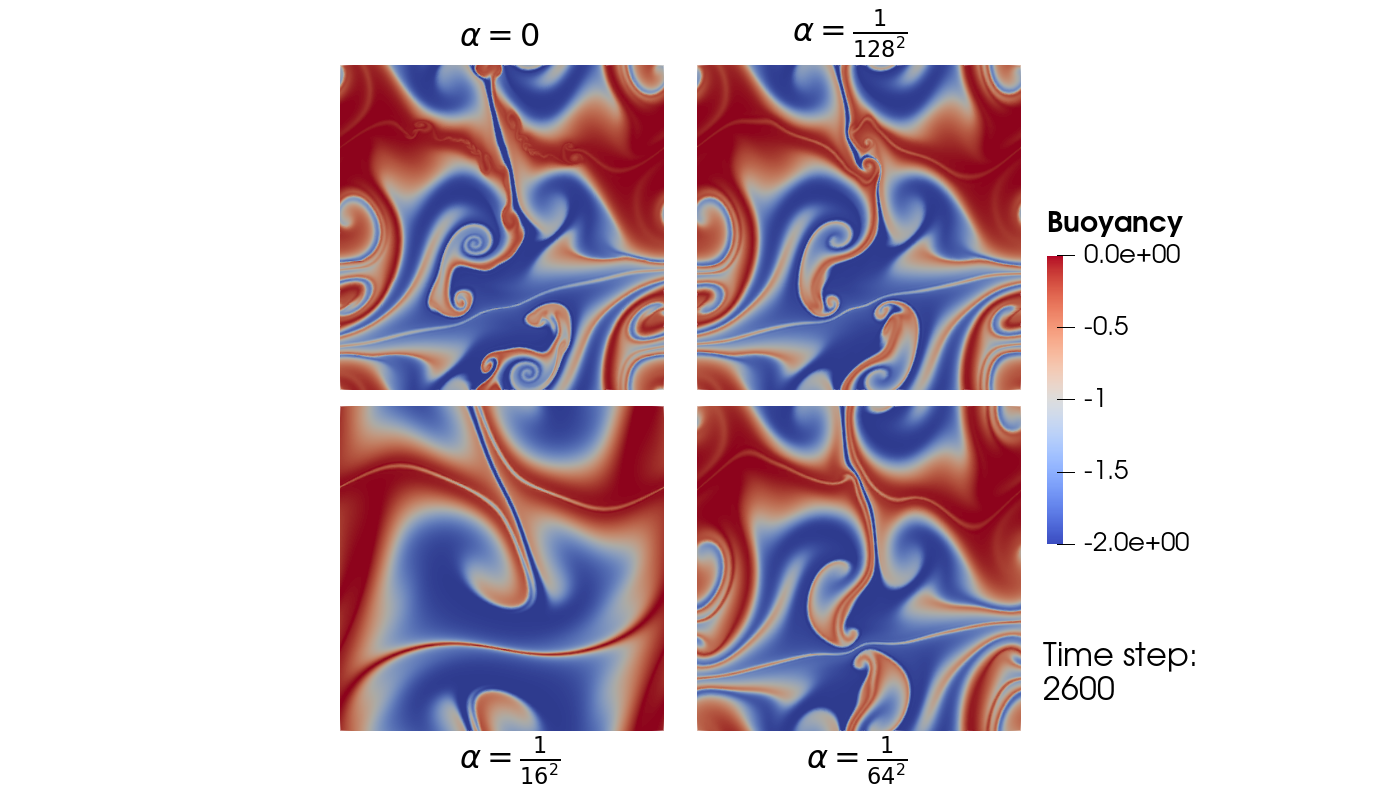}
     \caption{}
     \label{fig: snapshot buoyancy 1.30}
 \end{subfigure}
 \caption{Comparisons of solution snapshots of the $\alpha$-TQG buoyancy field that correspond to four different values of $\alpha$, at two different points in time. Potential vorticity and velocity magnitude snapshots of the same solutions are shown in Figure \ref{fig: snapshot pv} and Figure \ref{fig: snapshot velocity magnitude} respectively. Shown in each sub-figure are the results corresponding to $\alpha = 0$ (top left), $\alpha = \frac1{128^2}$ (top right), $\alpha = \frac1{64^2}$ (bottom right) and $\alpha = \frac1{16^2}$ (bottom left). When $\alpha=0$, the solution is of the TQG system. Sub-figure (A) shows the solutions at the $1280$'th time step, or equivalently when $t=0.64$. Sub-figure (B) shows the solutions at the $2600$'th time step, or equivalently when $t=1.30$. The flows in both sub-figures are at different stages of spin-up, with the flow at $t=1.30$ showing more developed features. The $\alpha$ parameter is  interpreted as the fraction of the domain's length-scale squared value at which regularisation is applied. Due to regularisation, the flows do develop differently. Nevertheless, we observe that as $\alpha$ gets smaller, the $\alpha$-TQG flow features converge to that of the TQG flow.}
 \label{fig: snapshot buoyancy}
\end{figure}

\begin{figure}[t]
\centering
 \begin{subfigure}[b]{0.6\textwidth}
     \centering
     \includegraphics[trim={340 25 196 5},clip,width=\textwidth]{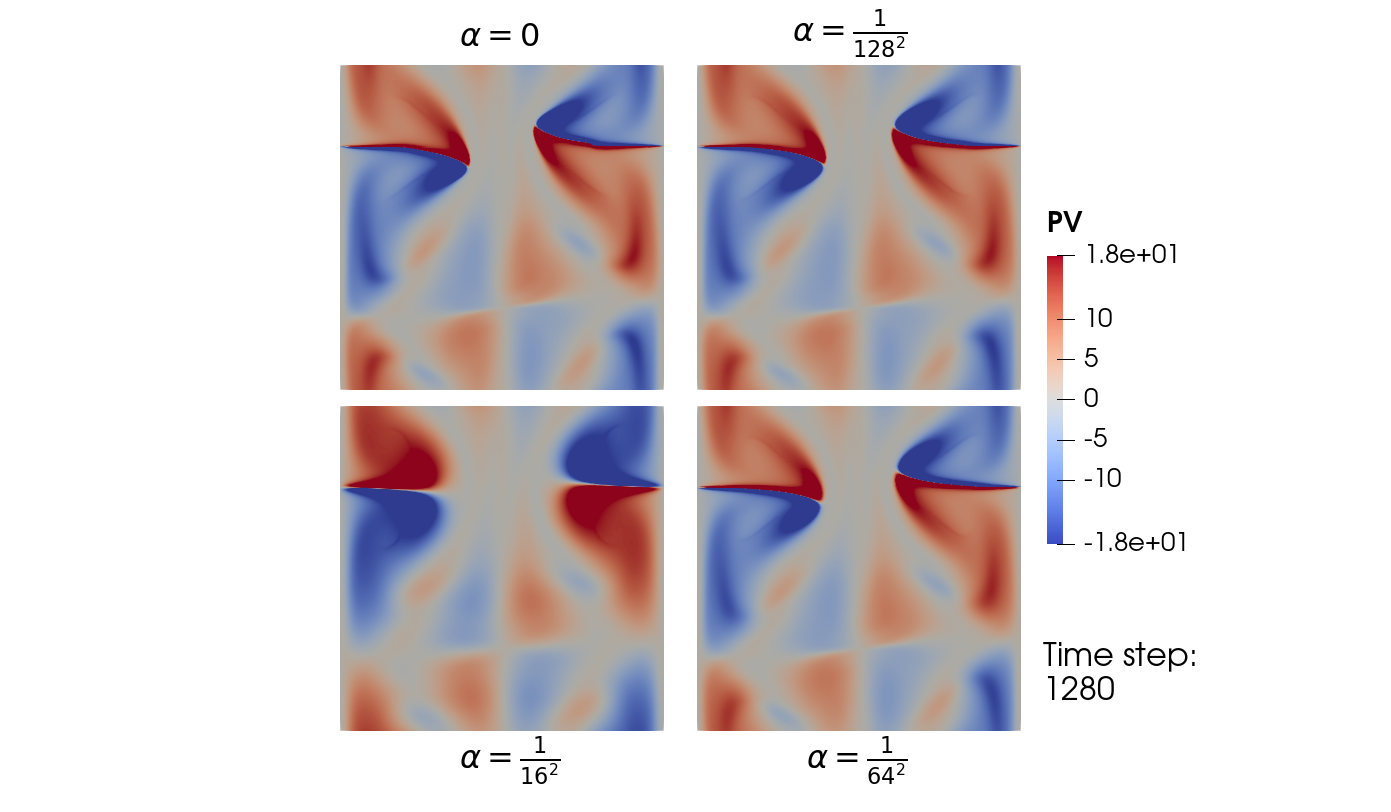}
     \caption{}
     \label{fig: snapshot pv 0.64}
 \end{subfigure}
 \vfill
 \begin{subfigure}[b]{0.6\textwidth}
     \centering
     \includegraphics[trim={340 25 196 5},clip, width=\textwidth]{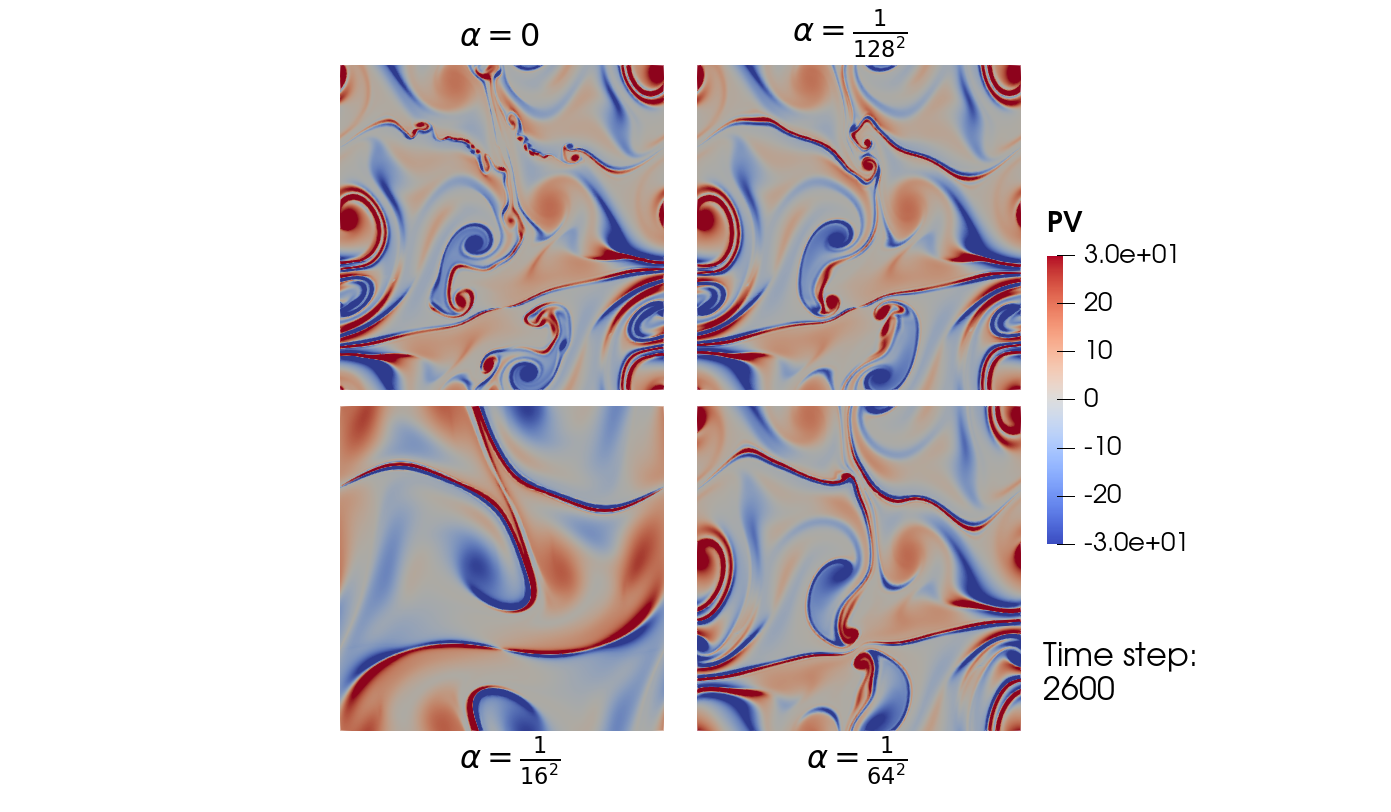}
     \caption{}
     \label{fig: snapshot pv 1.30}
 \end{subfigure}
 \caption{Comparisons of solution snapshots of the $\alpha$-TQG potential vorticity field that correspond to four different values of $\alpha$, at two different points in time. See the caption of Figure \ref{fig: snapshot buoyancy} for explanations of the arrangements of these plots. Buoyancy and velocity magnitude snapshots of the same solutions are shown in Figure \ref{fig: snapshot buoyancy} and Figure \ref{fig: snapshot velocity magnitude} respectively. As in Figure \ref{fig: initial conditions}, the different colours of the PV field correspond to positive and negative values of PV, which can be interpreted as clockwise and anticlockwise eddies respectively. In sub-figure (B), we observe more clearly the regularisation effects of $\alpha$. As $\alpha$ 
 increases, the flows develop in different ways -- features at smaller scales no longer develop, and larger features evolve differently as a result.
 Nevertheless, we observe that as $\alpha$ gets smaller, the $\alpha$-TQG flow features converge to that of the TQG flow.}
 \label{fig: snapshot pv}
\end{figure}

\begin{figure}[t]
\centering
 \begin{subfigure}[b]{0.6\textwidth}
     \centering
     \includegraphics[trim={340 25 151 5},clip,width=\textwidth]{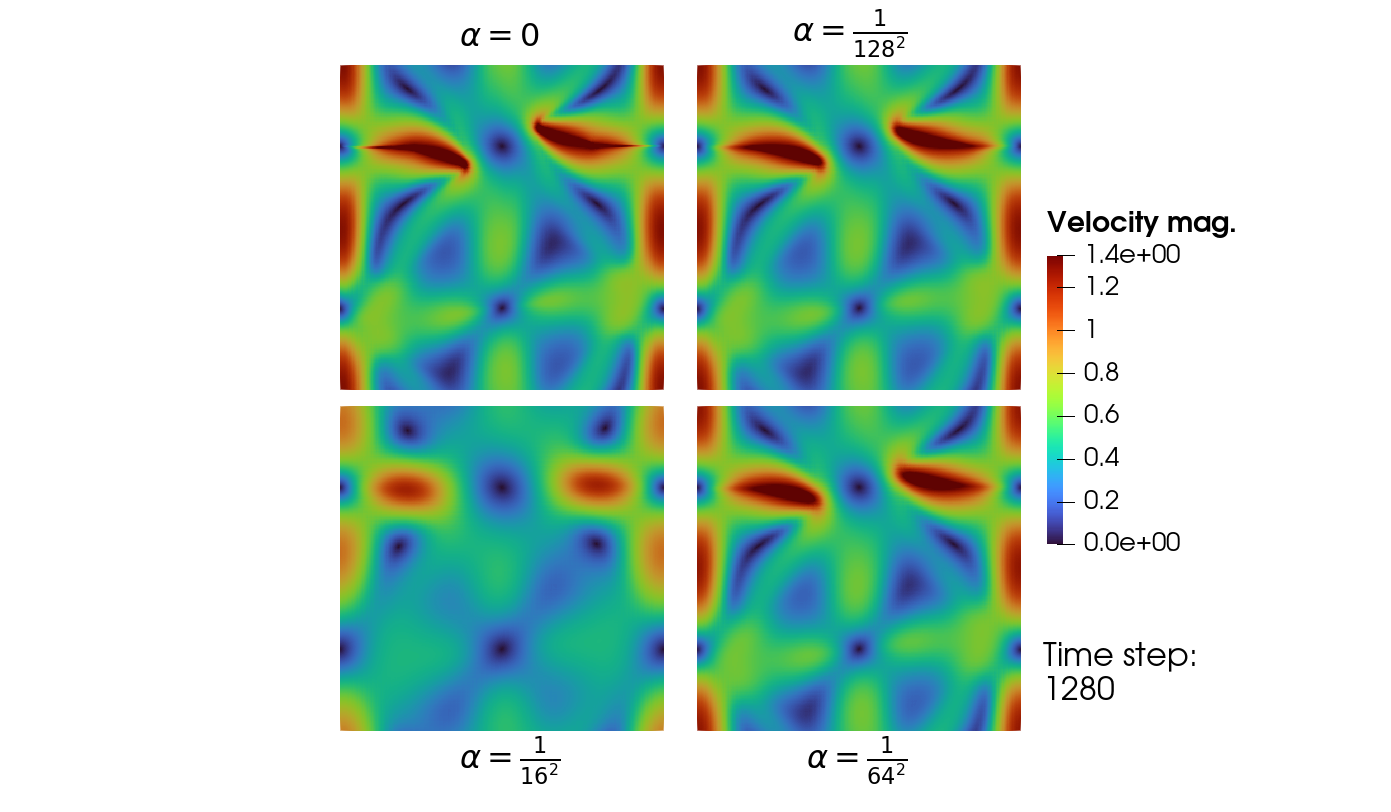}
     \caption{}
     \label{fig: snapshot velocity magnitude 0.64}
 \end{subfigure}
 \vfill
 \begin{subfigure}[b]{0.6\textwidth}
     \centering
     \includegraphics[trim={340 25 151 5},clip, width=\textwidth]{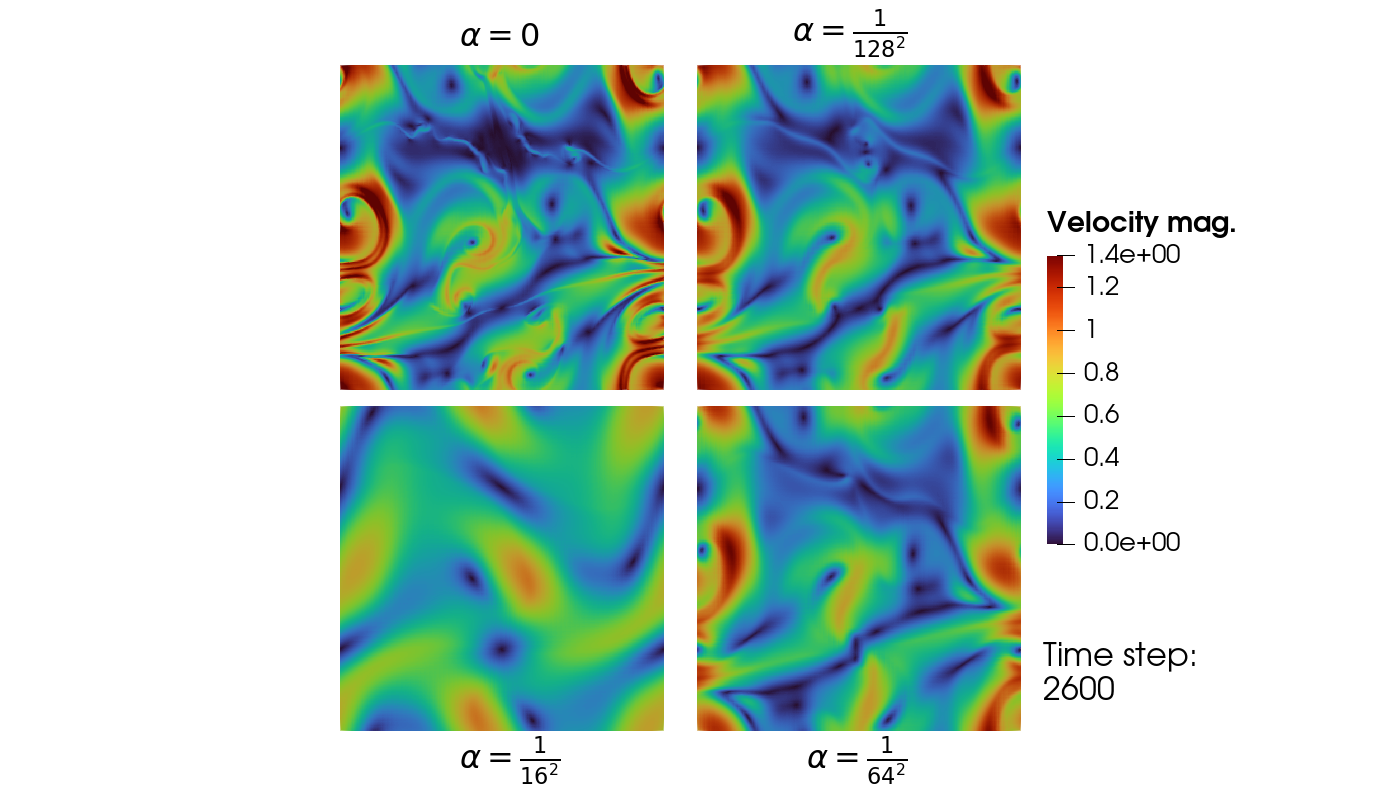}
     \caption{}
     \label{fig: snapshot velocity magnitude 1.30}
 \end{subfigure}
 \caption{Comparisons of solution snapshots of the $\alpha$-TQG velocity magnitudes that correspond to four different values of $\alpha$, at two different points in time. See the caption of Figure \ref{fig: snapshot buoyancy} for explanations of the arrangements of these plots. 
 Buoyancy and potential vorticity snapshots of the same solutions are shown in Figure \ref{fig: snapshot buoyancy} and Figure \ref{fig: snapshot pv} respectively. 
 Using the same scale for colouring, we observe in both sub-figures the strength of the colours weaken as $\alpha$ increases. This indicates smoothing of large velocity magnitudes.
 Additionally, in sub-figure (B), we observe how features at smaller scales get smoothed out. In particular, in the $\alpha=\frac1{16^2}$ plot, we see that essentially only large scale vorticies remain. Hence, considering the velocity field as a part of the system's nonlinear advection operator, these figures help us to better visualise how the flow features developed in Figure \ref{fig: snapshot buoyancy} and Figure \ref{fig: snapshot pv}. }
 \label{fig: snapshot velocity magnitude}
\end{figure}

\begin{figure}[t]
\centering
 \begin{subfigure}[b]{0.7\textwidth}
     \centering
     \includegraphics[width=\textwidth]{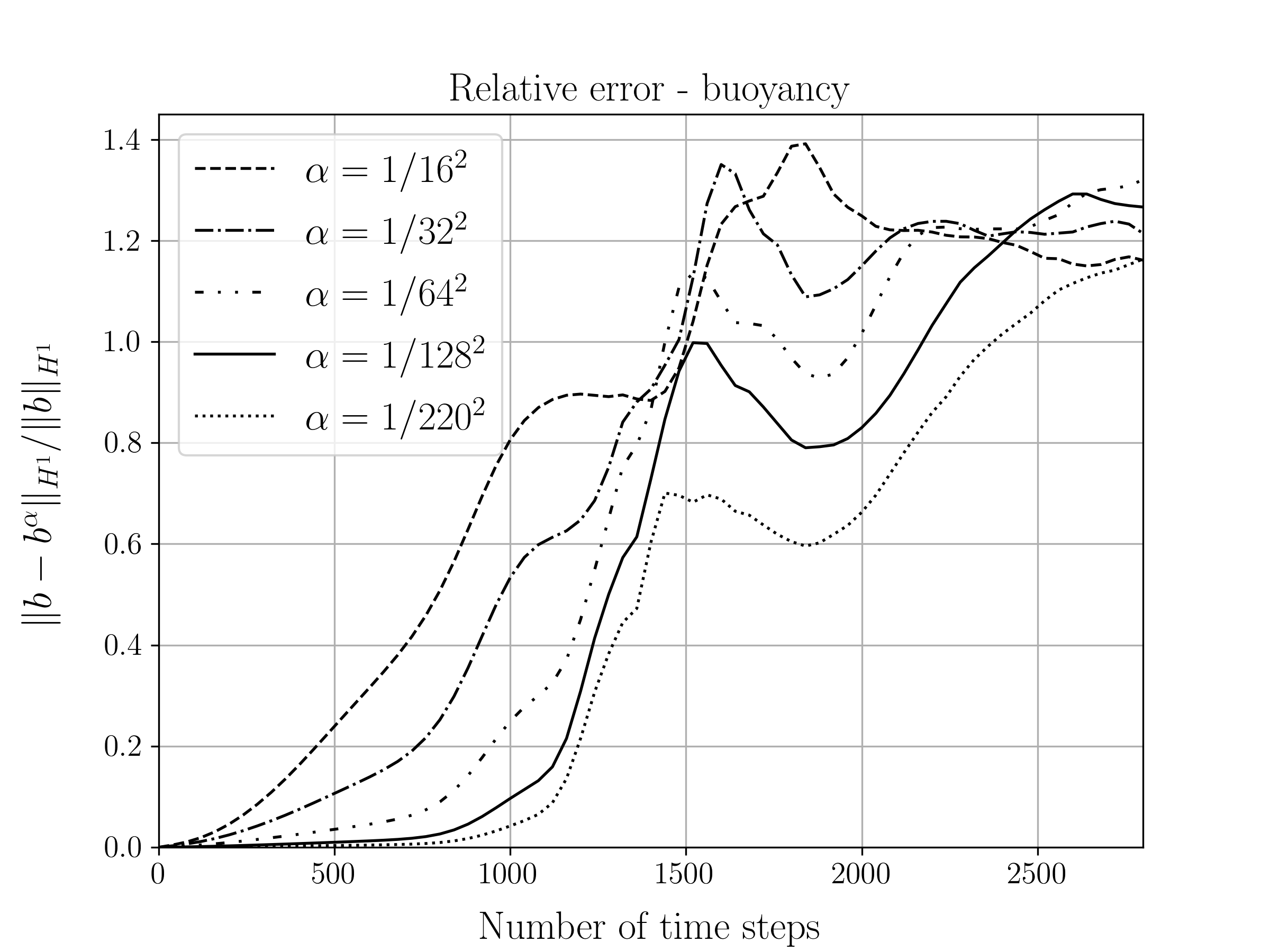}
     \caption{
     }
     \label{fig: relative error buoyancy time series}
 \end{subfigure}
 \begin{subfigure}[b]{0.7\textwidth}
     \centering
     \includegraphics[width=\textwidth]{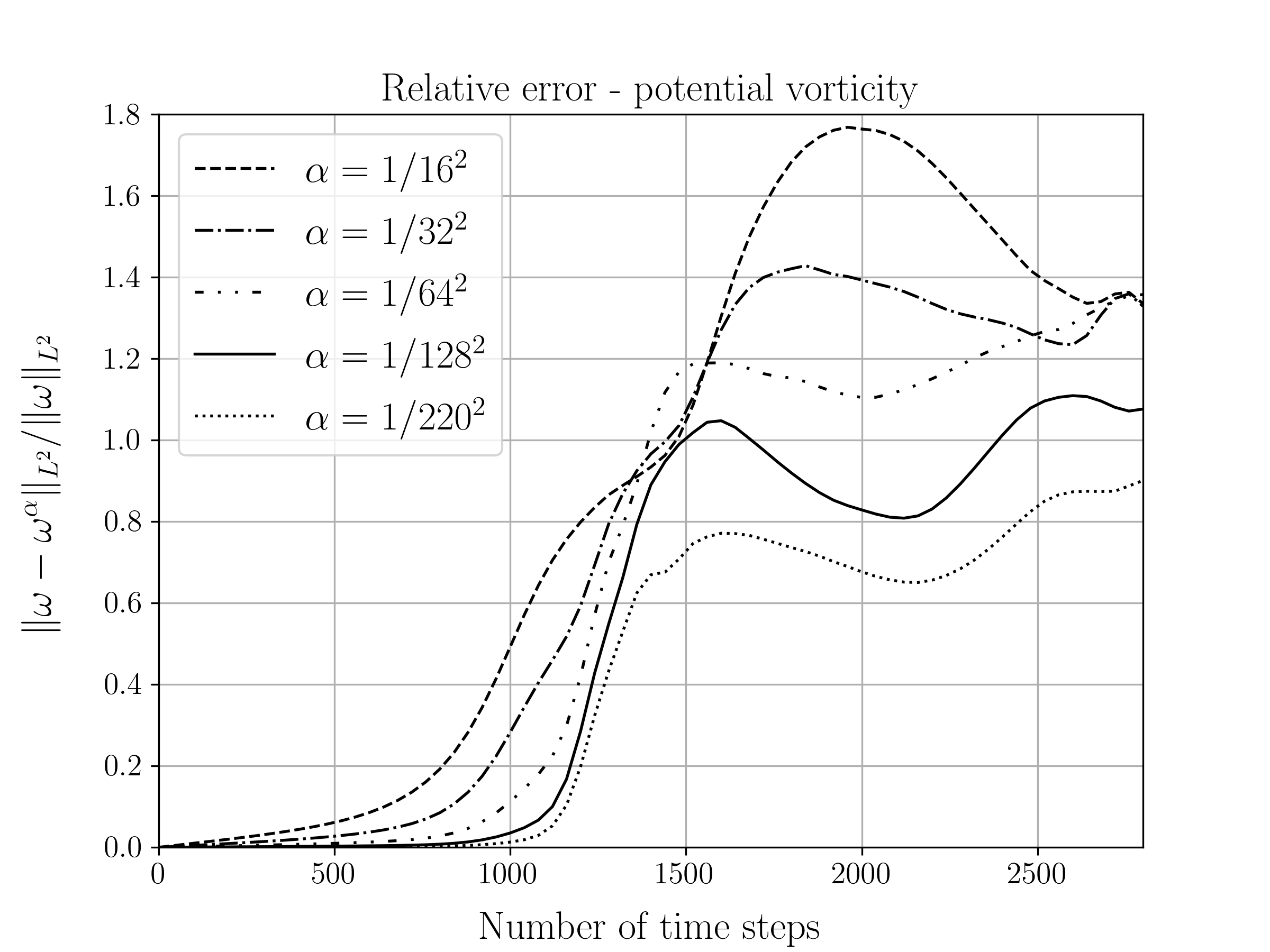}
     \caption{
     }
     \label{fig: relative error pv time series}
 \end{subfigure}
 \caption{Sub-figures (A) and (B) show, respectively, plots of the relative error functions $e_b(t, \alpha)$ and $e_\omega(t, \alpha)$
 as functions of time only, for the fixed $\alpha$ values $\frac{1}{16^2}$, $\frac1{32^2}$, $\frac1{64^2}$, $\frac1{128^2}$ and $\frac1{220^2}$. See equations \eqref{eq: relative error buoyancy} and \eqref{eq: relative error pv} for the definitions of $e_b$ and $e_\omega$ respectively. The plots are shown from $t=0$ up to and including $t=1.4$, which is equivalent to $2800$ time steps using $\Delta t = 0.0005$.
 We observe that, starting from $0$, all plots of  $e_b$ and $e_\omega$ increase initially, and plateau at around the $2000$'th time step. Further, up to the $1400$'th time step the plotted relative errors remain less than $1.0$ and are arranged in the ascending order of $\alpha$. Relating these relative error results at the $2600$'th time step to the solution snapshots at the same time point (shown in figures \ref{fig: snapshot buoyancy 1.30} and \ref{fig: snapshot pv 1.30}), we see that although the snapshots show convergence of flow features, the relative error values suggest all the $\alpha$-TQG solutions are more or less equally far away from the TQG flow.
 }
  \label{fig: relative errors time series}
\end{figure}

\begin{figure}[t]
\centering
 \begin{subfigure}[b]{0.7\textwidth}
     \centering
     \includegraphics[width=\textwidth]{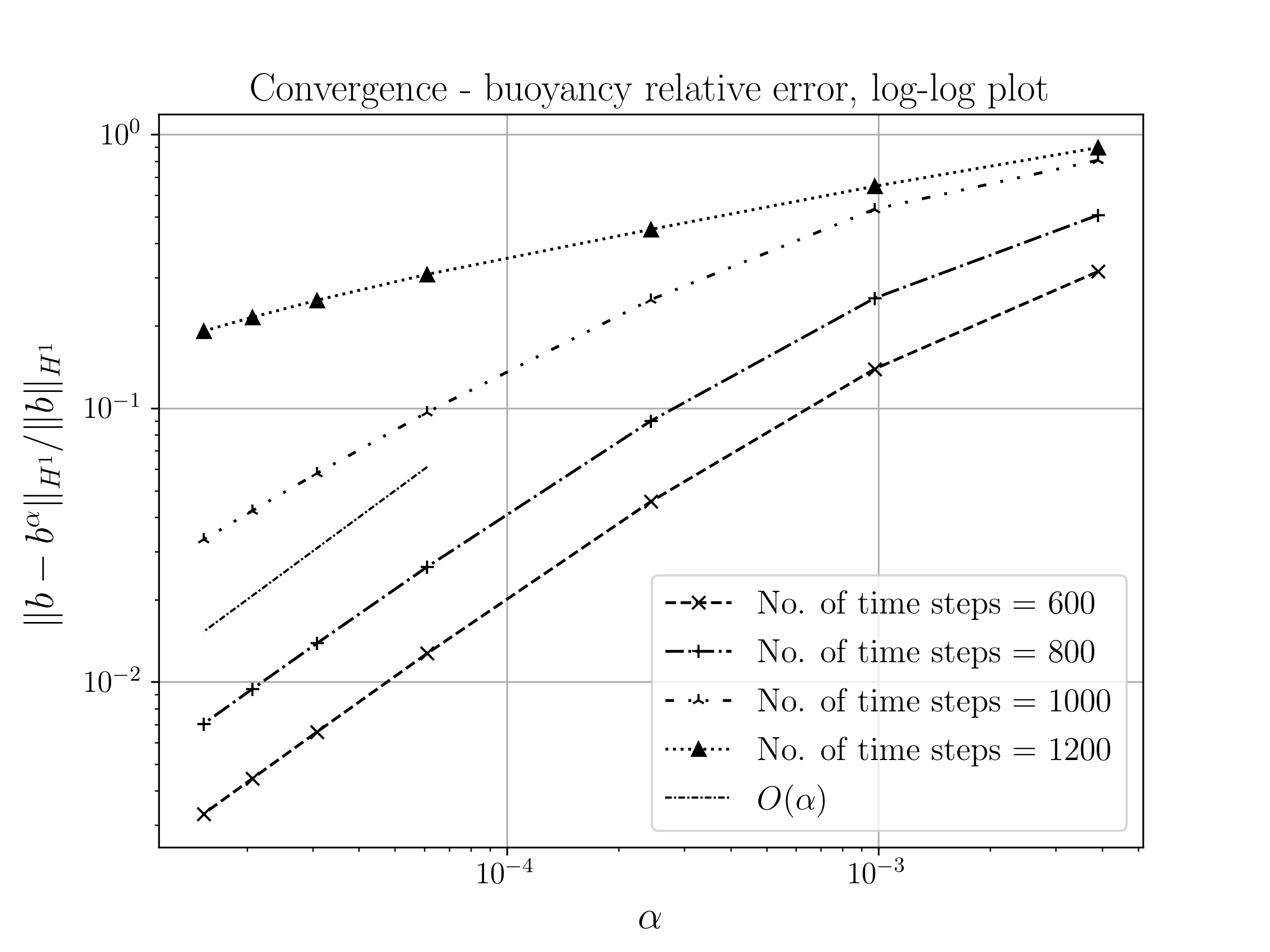}
     \caption{}
     \label{fig: relative error convergence buoyancy}
 \end{subfigure}
 \begin{subfigure}[b]{0.7\textwidth}
     \centering
     \includegraphics[width=\textwidth]{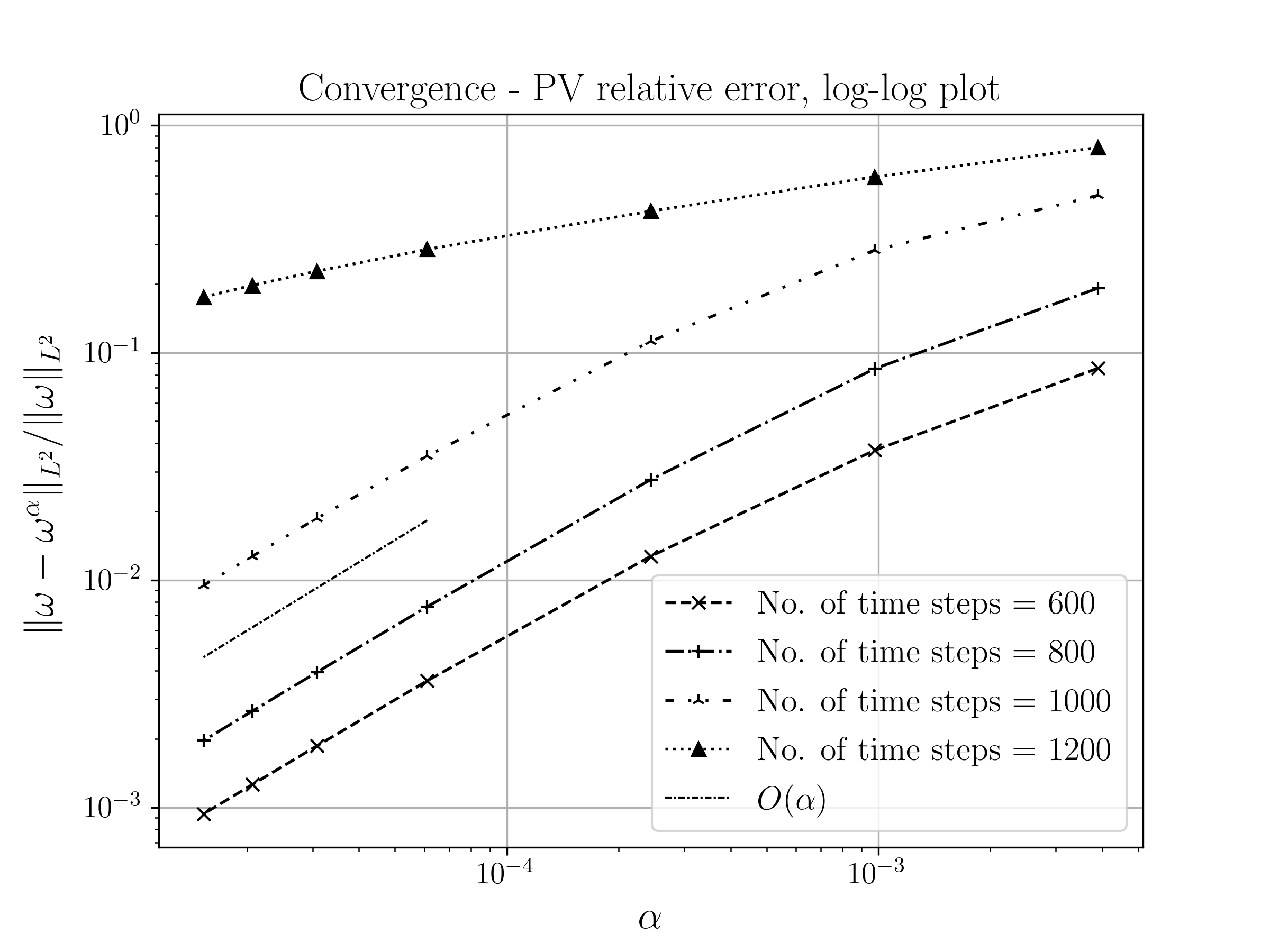}
     \caption{}
     \label{fig: relative error convergence pv}
 \end{subfigure}
 \caption{
 Sub-figures (A) and (B) show, respectively, plots in log-log scale of the relative error functions $e_b(t, \alpha)$, and $e_\omega(t, \alpha)$
 as functions of $\alpha$ only, at the fixed time values $t=0.3$ (equivalently at the $600$'th time step), $t=0.4$ (equivalently at the $800$'th time step), $t=0.5$ (equivalently at the $1000$'th time step) and $t=0.6$ (equivalently at the $1200$'th time step). 
  In view of Proposition \ref{prop:alphaConv}, we assumed that, based on the initial monotonicity of the relative errors (see Figure \ref{fig: relative errors time series}), the supremum over $[0,T]$ in \eqref{rateOfConv} can be estimated by evaluating $e_b$ and $e_\omega$ at $T$.
 Proposition \ref{prop:alphaConv} predicts that up to a certain time, the rate of convergence of $e_b$ and $e_\omega$ to $0$ should be no less than order 1 in $\alpha$.
 In both sub-figures we have plotted, as the reference for comparison, linear functions of  $\alpha$. Thus, comparing slopes to the reference, we see that numerically we get order $1$ convergence for $t=0.3$, $t=0.4$ and $t=0.5$. However, for $t=0.6$, the rate of convergence is no more than $1/2$.
 }
\end{figure}


\section{Conclusion and outlook}
In conclusion, we have formulated the thermal quasi-geostrophic (TQG) model equations in the regime of approximations relevant to GFD and we have explored their analytical and numerical properties. In respect to the analytical properties, we have shown that the TQG model and its $\alpha$-regularized version, the $\alpha$-TQG model, admit  unique local strong solutions that are stable in a larger space with weaker norm, that both models have a maximum time of existence, solutions of the latter model converge to solutions of the former model on any time interval in which a solution to the former lives, and we have also determined conditions under which the solution of the $\alpha$-TQG will blow up. 

With respect to numerics, we described our discretisation methods for approximating TQG and $\alpha$-TQG solutions and showed that the FEM semi-discrete scheme conserves numerical energy in the TQG case.
Using example simulation results, we numerically verified that $\alpha$-TQG solutions converge to that of the TQG system and attains 
the order $1$ in $\alpha$ convergence rate predicted by
Proposition \ref{prop:alphaConv}.
Additionally, we derived a dispersion relation for the linearised $\alpha$-TQG system and showed how $\alpha$-regularisation could be used to control the development of high wavenumber instabilites. 
Given that we have convergence of $\alpha$-TQG solutions to TQG, the linear stability results can be viewed as generalisations of those shown in \cite{holm2021stochastica} for the TQG system.

We now end this section with some open problems.
\begin{itemize}
    \item 
Can one construct a global-in-time strong solution (or a global weak solution) of either the TQG model equations  \eqref{ce}--\eqref{constrt}, or the $\alpha$-TQG model model equations \eqref{ceAlpha}--\eqref{constrtAlpha}?
\item 
Can one give a Beale--Kato--Majda condition for the blowup of a strong solution to the TQG in terms of a single unknown? That is, is there a BKM condition in terms of either $b$ or $\omega$ (or $\bu$) that does not require a combination of both variables?
\item 
As discussed in Section \ref{subsec-TQGderived}, an application of the Stochastic Advection by Lie Transport (SALT) approach to the Hamiltonian formulation of the TQG equations results in a stochastic Hamiltonian formulation of TQG equations in \eqref{ceZero}-\eqref{constrtZero}. This stochastic version of TQG preserves an infinite family of integral conserved quantities, as is shown in \cite{holm2021stochastica}. The SALT version of TQG represents an outstanding challenge for uncertainty quantification and data assimilation which will surely spur us on to further investigations of these problems. 
\item 
As also discussed in Section \ref{subsec-TQGderived}, the parallels between TQG and the Rayleigh-B\'enard equations should also attract our attention for the consideration of a SALT version of the deterministic Rayleigh-B\'enard convection equations.

\end{itemize}

\section*{Acknowledgements}
This work has been partially supported by European Research Council (ERC) Synergy grant STUOD-DLV-856408. We would like to thank our friends and colleagues for their encouraging comments, especially W. Bauer, F. J. Beron-Vera, B. Chapron, L. Cope, C. J. Cotter,  E. Dinvay, O. Lang, E. Memin, A. Radomska - Botelho Moniz, S. Takao. 

\section{Appendix}
\label{sec:appendix}
\noindent We present a result in this section,  which along with its proof, can be found in \cite{kato1984nonlinear}. Our construction of local solutions relies on this result and thus, we present it here for the sake of completeness.
\\
We consider the abstract equation
\begin{align}
\label{abstractEqXX}
\frac{\partial}{\partial t}u 
+
\mathcal{A}(t,u)
 =0,
\qquad
t\geq0
 ,
\qquad
u
\big\vert_{t=0}
=
u_0,
\end{align}
where $\mathcal{A}$ is a nonlinear operator.
\begin{theorem}
\label{thm:appen}
Let $\{V,H,X \}$ be real separable Banach spaces such that
\begin{itemize}
\item the embeddings $V\hookrightarrow H \hookrightarrow X$ are continuous and dense;
\item $H$ is a Hilbert space;
\item There is a continuous, nondegenerate bilinear form $(\, ,\,)$ on ${X \times V}$ such that $( u,v )=\langle u,v \rangle_H$ for $u\in H$ and $v\in V$.
\end{itemize}
Let $\mathcal{A}$ be a (sequentially) weakly continuous  map on $[0,T_*]\times H$ into $X$ such that
\begin{align}
\label{bilinear}
(\mathcal{A}(t,v), v)\geq -\rho(\Vert v \Vert_H^2)\qquad \text{for}\quad t\in[0, T_*], \quad v\in V,
\end{align}
where $\rho(r)\geq0$ is a monotone increasing function of $r\geq0$. Then for any $u_0\in H$, there exists $T>0$, $T\leq T_*$, and a solution $u$ of \eqref{abstractEqXX} in the class
\begin{align}
\label{cw0T}
u\in C_w([0,T]; H) \cap C_w^1([0,T];X).
\end{align}
Moreover,
\begin{align*}
\Vert u(t) \Vert_H^2 \leq r(t), \qquad t\in [0,T],
\end{align*}
where $r$ is a monotone increasing function on $[0,T]$; $T$ and $r$ can be chosen so as to depend only on $\rho$ and $\Vert u_0 \Vert_H$.
\end{theorem}
\begin{remark}(a)
If $\mathcal{A}$ maps $[0,T_*] \times V$ into $H$, we may replace \eqref{bilinear} with
\begin{align}
\label{bilinear1}
\langle \mathcal{A}(t,v), v\rangle_H\geq -\rho(\Vert v \Vert_H^2)\qquad \text{for}\quad t\in[0, T_*], \quad v\in V.
\end{align}
(b) $T$ and $r$ can be determined by solving the scalar differential equation
\begin{align}
\label{abstractODE}
\frac{\dd}{\dd t}r= 2\rho(r), \qquad r(0)=\Vert u_0 \Vert_H^2.
\end{align}
$T$ may be any value such that $r$ exists on $[0,T]$. If the solution to \eqref{abstractODE} is not unique, $r$ should be the maximal soluton.\\
(c) $u(t)\rightarrow u_0$ holds strongly in $H$ as $t\rightarrow0$. In other words, the solution \eqref{cw0T} is strongly continuous at the initial time $t=0$.
\end{remark}

%
\bibliographystyle{spmpsci}

\end{document}